\documentclass[12pt]{amsart}
\input epsf
\usepackage{amsmath}
\usepackage{epsfig}
\usepackage{latexsym}
\usepackage{amsfonts}
\usepackage[psamsfonts]{amssymb}
\def\C{{\bf C}}    

\def\R{{\bf R}}    

\def\D{\mathbf{D}}
\def\S{\mathbf{S}}
\def\P{\mathcal{P}}
\def\TT{\mathcal{T}}

\newtheorem{thm}{Theorem}[section]
\newtheorem{example}[thm]{Example}
\newtheorem{prop}[thm]{Proposition}
\newtheorem{lemma}[thm]{Lemma}

\newtheorem{cor}[thm]{Corollary}
\newtheorem{rmk}[thm]{Remark}
\newtheorem{df}[thm]{Definition} 

\begin{document}
\author{Andrei Gabrielov}
\date{\today}
\title{Classification of generic spherical quadrilaterals}

\begin{abstract}
Generic spherical quadrilaterals are classified up to isometry.
Condition of genericity consists in the requirement that the
images of the sides under the developing map belong to four distinct circles
which have no triple intersections.
Under this condition, it is shown that the space of quadrilaterals
with prescribed angles consists of finitely many open curves. Degeneration
at the endpoints of these curves is also determined.


\end{abstract}
\maketitle
\section{Introduction}
A {\sl spherical polygon} $Q$ is a surface homeomorphic to a closed disk equipped
with a Riemannian metric of constant positive curvature $1$, with $n$ conic singularities on the boundary,
labeled $a_0,\ldots,a_{n-1}$ counterclockwise, and such that the boundary arcs $[a_j,a_{j+1}]$
are geodesic. The singularities $a_j$ are the {\sl corners} of $Q$, and the boundary arcs $[a_j,a_{j+1}]$ are its {\sl sides}.

These objects appear in several areas of recent research. One of them is
the problem of classification of spherical metrics with conic singularities
on the sphere, see \cite{BMM,Lin-deg,CL,E,E2,E3,E4,EG1,EGT1,EGT2,EMP,LT,MZ,MP1,MP2} and references there.
In particular, \cite{E4} is a recent survey of the known results related to such metrics.
When all singularities lie on a circle on the Riemann sphere,
and the metric is symmetric with respect to that circle, the sphere is obtained by
gluing of two such polygons related by an anti-conformal isometry.
Thus spherical polygons provide an important class of examples of spherical metrics.
In fact, conditions for existence or non-existence of spherical metrics with prescribed angles on a sphere
in \cite{MP1}, and on tori in \cite{CL}, were obtained with the help of spherical polygons.

Another related problem is description of real solutions of Painlev\'e VI equations
with real parameters \cite{EGP}. In this problem, real special points
(zeros, poles, $1$-points and fixed points) of a solution correspond to
circular quadrilaterals. When the monodromy of the linear equation related
to the solution of Painlev\'e VI is unitarizable (conjugate to a subgroup of
$SU(2)$), these circular quadrilaterals are spherical quadrilaterals.

Complete classification of spherical triangles is known \cite{Klein}, \cite{E}, \cite{EGS}.

The structure of the set of spherical quadrilaterals with prescribed
angles strongly depends on the number of integer angles
(the angles with the radian measure an integer multiple of $\pi$).

Spherical quadrilaterals with at least one integer angle have been previously classified up to isometry:
when all angles are integers, in \cite{EG0};
when two angles are integers in \cite{EGT1};
when only one angle is integer in \cite{EGT3}.

Spherical quadrilaterals with four half-integer angles were classified in \cite{EG2}.
Any quadrilateral with such angles has two opposite sides mapped to the same circle by its developing map.

If $S$ is a surface with a Riemannian metric of curvature $1$, then
every point of $S$ has a neighborhood isometric to an open set
of the standard unit sphere $\S^2\subset \R^3$. This isometry $f$ is conformal,
therefore it is analytic and permits an analytic continuation along any
path in $S$. If $S$ is simply connected, we obtain a map $\Phi: S\to \S^2$
which is called the {\sl developing map}. The developing map is defined by
the metric up to a composition with a rotation of $\S^2$.

Conversely, if $\Delta$ is a disk in the complex plane, and $\Phi:\Delta\to \C$
a locally univalent meromorphic function such that
$\Phi(z)\sim c_j(z-a_j)^{\alpha_j},\; 0\leq j\leq n-1$
at $n$ boundary points $a_j$, and the arcs $\Phi([a_j,a_{j+1}])$ belong to
great circles in $\C$, then $\Phi$ is a developing map of a
spherical quadrilateral with angles $\pi\alpha_j$.
The metric on $\Delta$ is recovered by the formula for its length element
$ds=2|\Phi'|/(1+|\Phi|^2)$.

We call $Q$ a {\sl spherical quadrilateral} (resp., {\sl triangle, digon}) if $n=4$ (resp., $n=3$, $n=2$).
For convenience, we often drop ``spherical'' and refer simply to polygons (quadrilaterals, triangles, digons).

If a spherical polygon $Q$ has a {\sl removable} corner with the angle $1$,
the metric on $Q$ is non-singular at such a corner, thus $Q$
is isometric to a polygon with fewer corners.
However, we allow polygons with removable corners, since they may appear as building blocks of other polygons.

In this paper we consider classification of generic spherical quadrilaterals,
with the sides mapped to four generic (distinct, no triple intersections) great circles of the Riemann sphere $\S$
(although circle configurations with triple intersections will be considered in Section \ref{section:chains}).
All angles of a generic spherical quadrilateral are non-integer.
The four circles define a partition $\P$ of $\S$ with eight triangular faces and six quadrilateral faces,
such that each edge of $\P$ separates a triangular face from a quadrilateral one.
This partition is combinatorially equivalent to the
boundary of the convex hull of midpoints of the edges of a cube in $\R^3$ (see Fig.~\ref{cube}).
Two planar projections of the partition $\P$ are shown in Fig.~\ref{partition}.

\begin{figure}
\centering
\includegraphics[width=4in]{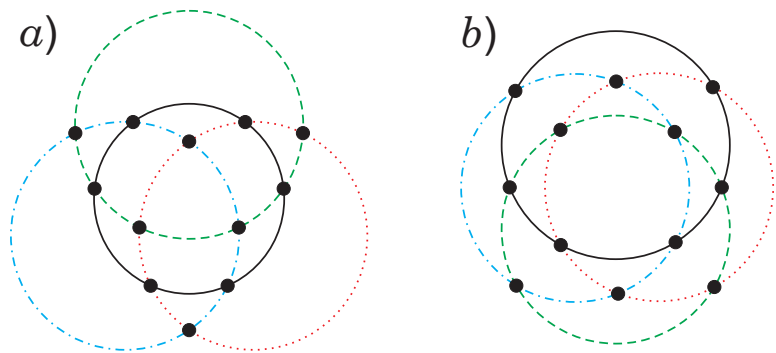}
\caption{Partition $\P$ of the Riemann sphere $\S$ by four great circles.}\label{partition}
\end{figure}
\begin{figure}
\centering
\includegraphics[width=3in]{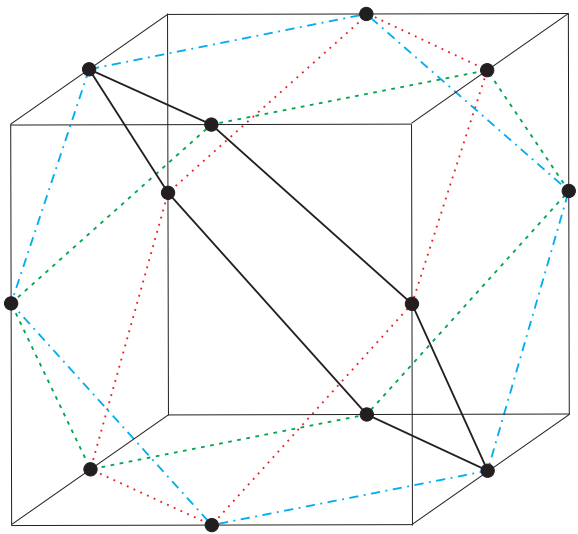}
\caption{Surface of the convex hull of midpoints of the edges of a cube, combinatorially equivalent to the partition $\P$ of $\S$.}\label{cube}
\end{figure}

Preimage of $\P$ defines a {\sl net} of a quadrilateral $Q$: a cell decomposition of $Q$
considered up to a label-preserving homeomorphism, so it is a combinatorial object (see Definition \ref{net}).
Preimage of each of the circles is the subset of the net consisting of {\sl arcs}, simple paths with the ends at vertices of the net, that may contain
corners of $Q$ only at their ends. Each side of $Q$ is a {\sl boundary arc} of its net.
An arc which is not a side of $Q$ is an {\sl interior arc}.
An interior arc with two ends at distinct corners of $Q$ is a {\sl diagonal arc}.
Note that a diagonal arc of a spherical quadrilateral may have its ends at two adjacent corners.
An arc is a {\sl loop} if it has both ends at the same vertex of the net.
A quadrilateral is {\sl irreducible} if its net does not contain a diagonal arc.
Note that an arc of a generic spherical quadrilateral may have both ends at its adjacent corners but not at its opposite corners (see Lemma \ref{nodiagonals}). An irreducible quadrilateral is {\sl primitive} if its net does not contain a loop.

For example, Fig.~\ref{partition}b can be interpreted as a net of a generic spherical quadrilateral $Q$ obtained by removing
from $\S$ an open quadrilateral face of the partition $\P$. Each circle in Fig.~\ref{partition}b consists of two arcs,
a side of $Q$ and a diagonal arc with the ends at two adjacent corners of $Q$. Thus $Q$ is not irreducible.
The net of $Q$ does not have loops. A net of an irreducible but not primitive generic spherical quadrilateral $P_1$ is shown
in Fig.~\ref{pseudo-diagonal}. It contains a {\sl pseudo-diagonal} consisting of four loops.
Removing two quadrilateral faces $puqx$ and $pvqw$ from the net of $P_1$, we get a net of the quadrilateral $Q$.
Both faces $puqx$ and $pvqw$ of $P_1$ are mapped to the same quadrilateral face of $\P$.
The edges of the net of $P_1$ in Fig.~\ref{pseudo-diagonal} (and the edges of the nets of quadrilaterals shown in other figures) have four colors (styles) indicating preimages of the four circles of the partition $\P$.

Classification of generic quadrilaterals is done as follows:
We start with four {\sl basic} quadrilaterals (see Fig.~\ref{basic}),
then build all primitive quadrilaterals using {\sl side extensions} in Section \ref{section:primitive} (see Figs.~\ref{4circles-x}-\ref{4circles-uvw}).
After that, irreducible quadrilaterals are obtained by replacing the quadrilateral face of the net of a primitive quadrilateral
containing two opposite corners in its boundary (at most one such face exists)
with one of the quadrilaterals $P_\mu$ ($\mu=1,2,\ldots$) (see Definition \ref{pmu}).
Finally, all generic quadrilaterals are obtained in Section \ref{generic} by attaching digons to {\sl short} (shorter than the full circle)
sides of irreducible quadrilaterals (see Theorem \ref{quad}).

Conditions on the fractional parts of the angles of spherical quadrilaterals with a given net define a pyramid $\Pi$
in the unit cube of $\R^4$ (see Proposition \ref{abcd}) or a pyramid obtained from $\Pi$ by replacing some of the angles by their complementary angles.
These conditions are compatible with the {\sl closure condition} in \cite{MP1} that implies that,
for the existence of a generic quadrilateral with the fractional parts $(\alpha,\beta,\gamma,\delta)$ of the angles, and with even (resp., odd) sum of the integer parts, the distance $d_1$ between the point $(\alpha,\beta,\gamma,\delta)$ and the odd (resp., even) integer lattice must be greater than 1.
In fact, the union of $\Pi$ and all pyramids obtained from $\Pi$ by taking an even number of
complements coincides with the 4-dimensional cross-polytope (also known as 16-cell or co-cube) which contains all points in the unit cube of $\R^4$
at the distance $d_1>1$ from the odd integer lattice, other than the point $(\frac12,\frac12,\frac12,\frac12)$
corresponding to a spherical rectangle.
It was shown in \cite{EG2} that a spherical rectangle cannot be generic: two of its opposite sides are mapped to the same circle.

In Section \ref{section:chains} we consider chains of spherical quadrilaterals and their nets.
If all four angles of a generic spherical quadrilateral $Q$ are fixed, quadrilaterals with the same net $\Gamma$ as $Q$ constitute
a one-parametric family (an open segment) $I_\Gamma$ in the space of all spherical quadrilaterals.
An important function on this family is the {\sl modulus} $K$ of the quadrilateral.
Every conformal quadrilateral can be mapped conformally onto a rectangle, so that the corners
$(a_0,\ldots,a_3)$ correspond to the corners of the rectangle $(0,1,1+iK,iK)$, where $K>0$ is the modulus.
We say that a sequence (or a family) of quadrilaterals conformally degenerates if $K$ tends to $0$ or $\infty$.

Unlike quadrilaterals with at least one integer angle considered in \cite{EG0,EGT1,EGT2,EGT3},
different generic quadrilaterals in $I_\Gamma$ are mapped to conformally non-equivalent four-circle configurations.
Six angles between four great circles satisfy a single relation. Only four angles are fixed in $I_\Gamma$
as fractional parts of the angles of a quadrilateral.
This leaves a ``fifth angle'' (the angle between the circles corresponding to the opposite sides of a quadrilateral)
that is not fixed in $I_\Gamma$. The main results of Section \ref{section:chains} are 
conditions on the ``fixed'' angles of a four-circle configuration that allow it to be deformed to a
configuration with a triple intersection (Proposition \ref{degen})
and relations between the fractional parts of the angles of a quadrilateral and the angles of the corresponding four-circle configuration,
depending on the net of a quadrilateral(Propositions \ref{even} and \ref{odd}).

At the ends of the segment $I_\Gamma$, a quadrilateral may conformally degenerate, or converge to a spherical quadrilateral with a non-generic four-circle
configuration containing a triple intersection, or else, after appropriate conformal transformations, converge to a non-spherical
(circular) quadrilateral with a four-circle configuration which is not conformally equivalent to a spherical one.
In the second case, the segment $I_\Gamma$ can be extended beyond the quadrilateral with a non-generic four-circle configuration
to a segment $I_{\Gamma'}$ of generic quadrilaterals with a different net $\Gamma'$.
A maximal family of spherical quadrilaterals obtained by such extensions is called a {\sl chain} of quadrilaterals, and
the sequence of nets of such family is called a chain of nets (see Definition \ref{chain}).
Any chain contains quadrilaterals with finitely many nets. It is an open segment in the space of spherical quadrilaterals, and the quadrilaterals at its ends either conformally degenerate or converge to non-spherical quadrilaterals beyond which the chain cannot be extended.
The chains of generic spherical quadrilaterals depend not only on their nets but also on some inequalities between the fractional parts of their angles.
This makes counting the chains of generic quadrilaterals much harder than counting the chains of the quadrilaterals
with at least one integer angle considered in \cite{EGT1,EGT2,EGT3}. As an example,
Propositions \ref{x-even} and \ref{x-odd} describe different possibilities
for the chains of quadrilaterals with nets $X_{kl}$ and $X'_{pq}$, depending on the integer and fractional parts of their angles.

The {\sl length} of a chain of quadrilaterals is the non-negative number of ``links'' in it corresponding to four-circle configurations with
triple intersections. Thus the number of nets in a chain is greater by one than its length. If the quadrilaterals conformally degenerate
in the limits at both ends of a chain, the two degenerations are either of the same or of the opposite kind,
depending on the parity of the length of the chain.
If the length of a chain is even (for example, if all quadrilaterals in a chain of length $0$ have the same net)
then the modulus $K$ of the quadrilaterals in that chain converges to distinct values $0$ and $\infty$ at the two ends of the chain.
In that case, the chain contains at least one quadrilateral with each value $K>0$ of the modulus.
If the length of a chain is odd, the modulus $K$ of the quadrilaterals in that chain converges
to the same value (either $0$ or $\infty$) at both ends of the chain.
In that case, the chain contains at least two quadrilaterals with either sufficiently small or sufficiently large values of the modulus $K$.
Thus classification of the chains of quadrilaterals allows one, in principle, to obtain lower bounds for the number of
quadrilaterals with the given angles and modulus,
and to count quadrilaterals with the given angles and either small enough or large enough value of the modulus.
This is a hard combinatorial problem, which is not addressed in this paper.
Note that, according to Lemma \ref{onebigangle}, if the angles at three corners of a quadrilateral $Q$ are less than $1$,
then the net of $Q$ is of one of the types $X$, $X'$, $\bar X$, $\bar X'$. 
Thus Propositions \ref{x-even} and \ref{x-odd} provide the answer for such quadrilaterals.

When the fractional parts of the angles of quadrilaterals with a given net $\Gamma$ satisfy an additional equality (see Remark \ref{rmk:tangent})
then a family $I_\Gamma$ of spherical quadrilaterals may degenerate at one end of the segment $I_\Gamma$ so that the modulus converges to
a finite positive value.
Applying an appropriate family of linear-fractional transformations to the sphere, one can replace the family of four-circle configurations
to which the family $I_\Gamma$ is mapped by a conformally equivalent family of configurations of four not necessarily great circles,
converging to a configuration with a single quadruple intersection (see Fig.~\ref{quadruple}),
so that the corresponding family of circular quadrilaterals would be converging to a non-degenerate circular (non-spherical) quadrilateral
(see Example \ref{chain-x-quadruple} and Figure \ref{fig:chainxquadruple}).
This phenomenon was observed in \cite{LW}, \cite{EG2}, \cite{MP2} and \cite{EGMP}.
Note that in all previously considered cases the angles were half-integer, while there is a single equality satisfied by the angles of the quadrilaterals in Example \ref{chain-x-quadruple}. This equality is compatible with the {\sl non-bubbling condition} in \cite{MP2} (see also \cite{E4}).

The author thanks Alexandre Eremenko who opened to him the fascinating world of spherical
polygons and spherical metrics with conical singularities. He explained that, although generic spherical
triangles were classified by Felix Klein in the beginning of the last Century, classification of generic spherical quadrilaterals remained
an open problem, and asked whether it could be solved using the nets, a technique developed in \cite{EG0}.
After almost ten years after that conversation, a (partial) answer to his question is presented in this paper.
The author is grateful to Prof.~Eremenko and the anonymous referees for suggesting numerous improvements to this text.

\section{Generic quadrilaterals and their nets}\label{section:nets}

\begin{df}\label{net}\normalfont Let $Q$ be a generic spherical quadrilateral with the sides mapped to
four circles of a partition $\P$ of the sphere $\S$.
Preimage of $\P$ defines a cell decomposition $\Gamma$ of $Q$, called the {\sl net} of $Q$.
The {\sl vertices}, {\sl edges} and {\sl faces} of $\Gamma$ are connected components
of the preimages of vertices, edges and faces of $\P$.
For simplicity, we call them vertices, edges and faces of $Q$.
The net $\Gamma$ has the same types of faces (triangles and quadrilaterals)
as the partition $\P$, with the same adjacency rules: if two faces of $\Gamma$ have a common edge then one of them is a triangle
and another one a quadrilateral.

The corners of $Q$ are vertices of its net $\Gamma$. The {\sl order} of a corner $p$ of $Q$ is the integer part of the angle of $Q$ at $p$.
Since the angle of $Q$ at $p$ is not integer and the four-circle configuration does not have triple intersections, the two sides of $Q$
adjacent to $p$ map to two different circles, thus the degree of a corner as a vertex of $\Gamma$ is even.
In addition, the net $\Gamma$ of $Q$ may have {\sl interior vertices} of degree 4 and
{\sl lateral vertices} (on the sides but not at the corners) of degree 3.

The {\sl order} of a side $L$ of $Q$ is the number of edges of $\Gamma$ in $L$.
The union of three consecutive edges on the same circle of $\P$ is a half-circle
with both ends at the intersection of the same two circles of $\P$.
Since the opposite sides of $Q$ are mapped to different circles,
corners at the ends of each of its side cannot be mapped to intersection of the same two circles,
thus the order of a side of a generic quadrilateral $Q$ cannot be divisible by 3.
A side of $Q$ is {\sl short} (shorter than a full circle) if its order is less than 6, and
{\sl long} otherwise.

Two spherical polygons $Q$ and $Q'$ are
{\sl combinatorially equivalent} if there is an orientation preserving
homeomorphism $Q\to Q'$ mapping the corners $a_j$ of $Q$ to the corners $a'_j$ of $Q'$
and the net $\Gamma$ of $Q$ to the net $\Gamma'$ of $Q'$.
\end{df}

\begin{df}\label{arcs}\normalfont Let $Q$ be a generic spherical quadrilateral with the sides mapped to four circles of a partition $\P$ and the net $\Gamma$.
If $C$ is a circle of $\P$ then the intersection $\Gamma_C$ of $\Gamma$ with the preimage
of $C$ in $Q$ is called the $C$-{\sl net} of $Q$.
All interior vertices of $\Gamma_C$ have degree two.
A $C$-{\sl arc} of the net of $Q$ (or simply an arc when $C$ is not specified) is
a simple path in $\Gamma_C$ with the ends at vertices of $\Gamma$, which may have corners of $Q$ only at its ends.
The {\sl order} of an arc is the number of edges in it.
An arc is a {\sl loop} if it is a closed path.
An arc $\gamma$ of $Q$ is {\sl lateral} if it is a subset of a side of $Q$.
Otherwise, $\gamma$ is an {\sl interior arc}.
An arc is {\sl maximal} if it cannot be extended to a larger arc.

A {\sl diagonal} arc is an interior arc of $Q$ with both ends at distinct corners of $Q$.
A spherical quadrilateral $Q$ is {\sl irreducible} if it does not have a diagonal arc.
An irreducible quadrilateral is {\sl primitive} if it does not contain a loop.
\end{df}

\subsection{Quadrilaterals $P_\mu$ and pseudo-diagonals}\label{sub:pmu}

\begin{lemma}\label{loop}
Let $Q$ be an irreducible spherical quadrilateral that contains a loop $\gamma$.
Then $\gamma$ has a vertex at a corner of $Q$.
\end{lemma}

\begin{proof}
If all vertices of $\gamma$ are interior vertices of $Q$, then $\gamma$ bounds a disk $D$ inside $Q$
which is mapped one-to-one onto a disk in the sphere $\S$ bounded by a circle $C$ of the partition $\P$,
and $\gamma$ maps one-to-one onto $C$.
Since all vertices of $\gamma$ are interior vertices of the net $\Gamma$ of $Q$,
the faces of $\Gamma$ adjacent to $\gamma$ map one-to-one to the faces of $\P$ adjacent to $C$.
The union of $D$ and all these faces is a spherical triangle $T$
which maps one-to-one to the complement of a triangular face of the partition $\P$ (see Fig.~\ref{partition}a).
Note that the corners of $T$ cannot be lateral vertices of $\Gamma$, since they have degree 4.
If a corner of $T$ is an interior vertex of $\Gamma$,
then $\Gamma$ must have a triangular face completing the image of $T$ to the full sphere.
Since $Q$ is not a sphere, this is possible only when the other two corners of $T$ are corners of $Q$.
Since any two corners of $T$ are connected by an interior arc, $Q$ is not irreducible, a contradiction.
\end{proof}

\begin{lemma}\label{p1}
Let $Q$ be an irreducible spherical quadrilateral that contains a loop $\gamma$.
Then $Q$ contains a quadrilateral combinatorially equivalent to the quadrilateral
$P_1$ shown in Fig.~\ref{pseudo-diagonal}.
\end{lemma}

\begin{proof}
According to Lemma \ref{loop}, the loop $\gamma$ has a vertex $p$ on the boundary of $Q$.
It cannot be a lateral vertex (otherwise $\gamma$ would be a side of $Q$) thus $p$ is a corner of $Q$.
Since $Q$ is irreducible, all other vertices of $\gamma$ must be interior vertices of the net $\Gamma$ of $Q$.
The loop $\gamma$ bounds a disk $D\subset Q$ mapped one-to-one to a disk bounded by a circle $C$ of $\P$.
Thus $\Gamma$ must contain the union of $D$ and six more faces adjacent to $\gamma$ outside $D$,
which is the quadrilateral $G$ with corners $p,\,u,\,q,\,b$ shaded in Fig.~\ref{pseudo-diagonal}.
Since $b$ is connected to $p$ by an interior arc, it cannot be a corner of $Q$.
Also, $b$ cannot be a lateral vertex of $\Gamma$ since its degree is greater than 3.
Thus $b$ is an interior vertex, and $\Gamma$ contains a triangular face with vertices $b,\,q,\,v$.
Similarly, $v$ cannot be a corner of $Q$, as it is connected to $p$ by an interior arc.
Thus $v$ is an interior vertex, and $\Gamma$ contains a quadrilateral face with vertices $p,\,v,\,q,\,w$.
This implies that $Q$ contains the quadrilateral $P_1$ with corners $p,\,u,\,q,\,w$ shown in Fig.~\ref{pseudo-diagonal}, completing the proof of Lemma \ref{p1}.
Note that vertices $p$ and $q$ of $P_1$ must be opposite corners of $Q$.
\end{proof}

\begin{figure}
\centering
\includegraphics[width=4.8in]{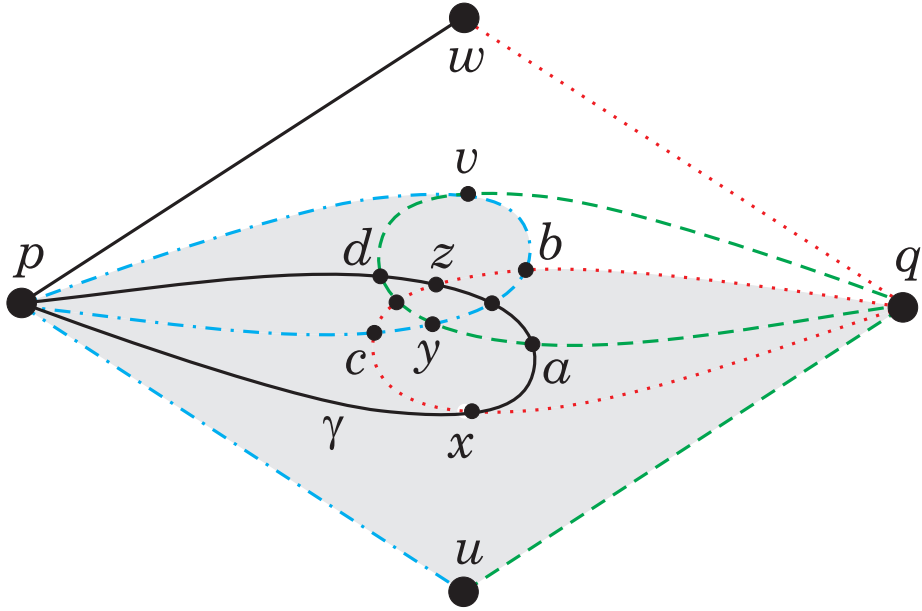}
\caption{The net of the quadrilateral $P_1$ with corners $p,\,u,\,q,\,w$.}\label{pseudo-diagonal}
\end{figure}

\begin{df}\label{pmu}\normalfont
Starting from the quadrilateral $P_1$ shown in Fig.~\ref{pseudo-diagonal},
we define a sequence $P_\mu$ of non-primitive irreducible
quadrilaterals as follows.
All sides of $P_\mu$ have order 1.
Two opposite corners $p$ and $q$ of $P_\mu$ have
order $2\mu$, and two other corners have order 0.
The net of $P_\mu$ contains $\mu+1$ quadrilateral faces
having both $p$ and $q$ as their opposite vertices.
These faces are separated by $\mu$ ``pseudo-diagonals,''
each consisting of four loops, two of them having a common vertex $p$ and another two a common vertex $q$.
For $\mu\ge 1$, the quadrilateral $P_{\mu+1}$ is obtained by replacing
any one of these $\mu+1$ faces of $P_\mu$ by the quadrilateral $P_1$.

For convenience, we define $P_0$ to be a spherical quadrilateral which maps one-to-one to a single quadrilateral face of the partition $\P$.
\end{df}

\subsection{Spherical digons}\label{sub:digons}
A spherical digon $D$ has two corners with equal angles, which may be integer and even removable,
and two short sides. Since the boundary of $D$ consists of arcs of at most two circles,
geometrically (up to conformal equivalence) $D$ is completely determined by the angles at its corners.
(See \cite{EGT1}, Theorem 4.1.)
However, since we need spherical digons as building blocks of spherical quadrilaterals,
we define their nets as preimages of all four circles of the partition $\P$.
The following Lemma provides classification of combinatorially distinct irreducible spherical digons.

\begin{lemma}\label{digon-irreducible}
There are three combinatorially distinct types of irreducible digons: $D_{15}$, $D_{24}$ and $D_{33}$ (see Figs.~\ref{D15}, \ref{D24} and \ref{D33}).
Digons $D_{15}$ and $D_{24}$ have integer corners of order 1.
Digon $D_{33}$ has non-integer corners of order 0.
There are two sorts, $D^a_{15}$ and $D^b_{15}$, of digons $D_{15}$
(see Fig.~\ref{D15}). Their nets have reflection symmetry.
There are two sorts, $D^a_{24}$ and $D^b_{24}$, of digons $D_{24}$
(see Fig.~\ref{D24}). Their nets are reflection symmetric to each other.
There are two sorts, $D^a_{33}$ and $D^b_{33}$, of digons $D_{33}$
(see Fig.~\ref{D33}). Their nets have two reflection symmetries.
\end{lemma}

\begin{proof}
Let $D$ be an irreducible digon with the corners $p$ and $q$,
and the sides $L$ and $L'$.
Then $L$ and $L'$ map to some circles $C$  and $C'$ of the partition $\P$ (possibly, to the same circle).
If the equal angles at the corners of $D$ at $p$ and $q$ are non-integer then $C'\ne C$,
and $D$ is a digon of a partition of $\S$ by the two circles $C$ and $C'$.
Since $D$ is irreducible, it is either $D^a_{33}$ or $D^b_{33}$,
with the face of the net of $D$ adjacent to its corner being a quadrilateral or a triangle,
respectively.

If the angles at the corners of $D$ are integer then $C'=C$,
$D$ maps one-to-one onto a disk $\D$ bounded by $C$,
and, since $D$ is irreducible, the images of $p$ and $q$ are two vertices of the partition $\P$ on $C$ which are not connected by an
interior arc.
Thus $D$ is one of the digons $D^a_{15}$, $D^b_{15}$, $D^a_{24}$, $D^b_{24}$.
\end{proof}

\begin{figure}
\centering
\includegraphics[width=4in]{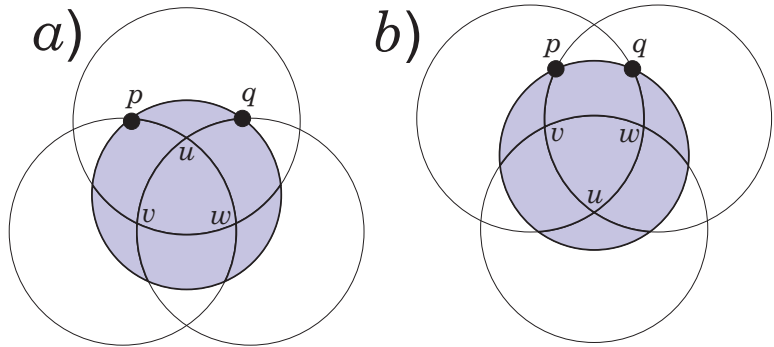}
\caption{Irreducible digons of type $D_{15}$.}\label{D15}
\end{figure}

\begin{figure}
\centering
\includegraphics[width=4in]{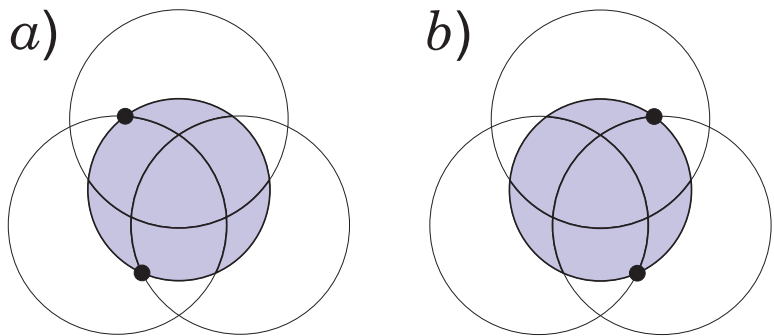}
\caption{Irreducible digons of type $D_{24}$.}\label{D24}
\end{figure}

\begin{figure}
\centering
\includegraphics[width=4in]{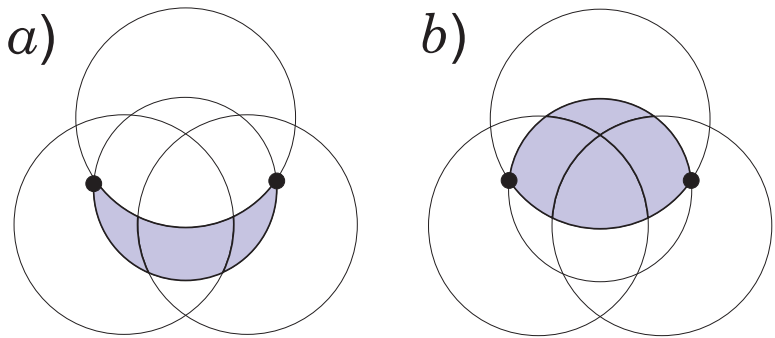}
\caption{Irreducible digons of type $D_{33}$.}\label{D33}
\end{figure}

\begin{thm}\label{digon}
Any spherical digon $D$ is a union of $k>0$ irreducible digons
of the same type with common vertices,
glued together along their common sides.
The type of irreducible digons in $D$ is called the type of $D$.
A digon $D_{15}$ can be only glued to a digon $D_{15}$
of a different sort.
A digon $D_{24}$ can be only glued to a digon $D_{24}$ of the same sort.
A digon $D_{33}$ can be only glued to a digon $D_{33}$ of a different sort.
The corners at both vertices of $D$ have order $k$ if its type is $D_{15}$ or $D_{24}$,
and $[k/2]$ if its type is $D_{33}$.
\end{thm}

The proof is an easy exercise. Sometimes it is convenient to allow $k=0$ in Theorem \ref{digon}, to denote an empty digon.

\begin{rmk} \normalfont Note that a generic quadrilateral $Q$ cannot contain a digon $D$ of type $D_{33}$ having both corners at the corners of $Q$, since the same two circles intersect at both corners of $D$.
\end{rmk}
\medskip

\subsection{Spherical triangles}\label{sub:triangles}
A spherical triangle has its sides on at most three circles, and classification of such triangles goes back to Klein
(see also \cite{EGT3}, Section 6).
Since we need spherical triangles as building blocks of generic spherical quadrilaterals, we consider only spherical triangles with all sides mapped to some circles of the partition $\P$ and all corners mapped to intersection points of the circles of $\P$.
All irreducible spherical triangles are primitive, and can be classified as follows.
Triangle $T_n$ (see Fig.~\ref{triangles-te}a) has an integer corner of order $n$ and two non-integer corners
of order $0$. The angles at its non-integer corners are equal when $n$ is odd and complementary (adding up to $1$)
when $n$ is even. Triangle $E_n$ (see Fig.~\ref{triangles-te}b) has a non-integer corner of order $n$ and two non-integer corners of order $0$.

\begin{rmk}\label{t1-net}\normalfont
Fig.~\ref{triangles-te} does not show preimages of the fourth circle of $\P$ (which does not pass through any corners of a triangle). We'll need these preimages to understand the nets of generic quadrilaterals obtained by attaching triangles to basic quadrilaterals. In particular, the triangle $T_1$ geometrically is a digon $D_{33}$ (see Fig.~\ref{D33}) with one of the side vertices of the net of $D_{33}$ being the integer corner of $T_1$.
Figs.~\ref{triangles-t1} and \ref{triangles-e1} show possible nets for triangles $T_1$ and $E_1$, respectively, with the preimages of all four circles of $\P$ included. Note that the net of $T_n$ always has an edge connecting its integer corner with its base (the side opposite to its integer corner). This property will be important for understanding the chains of quadrilaterals later in this paper.
\end{rmk}

\begin{figure}
\centering
\includegraphics[width=4.8in]{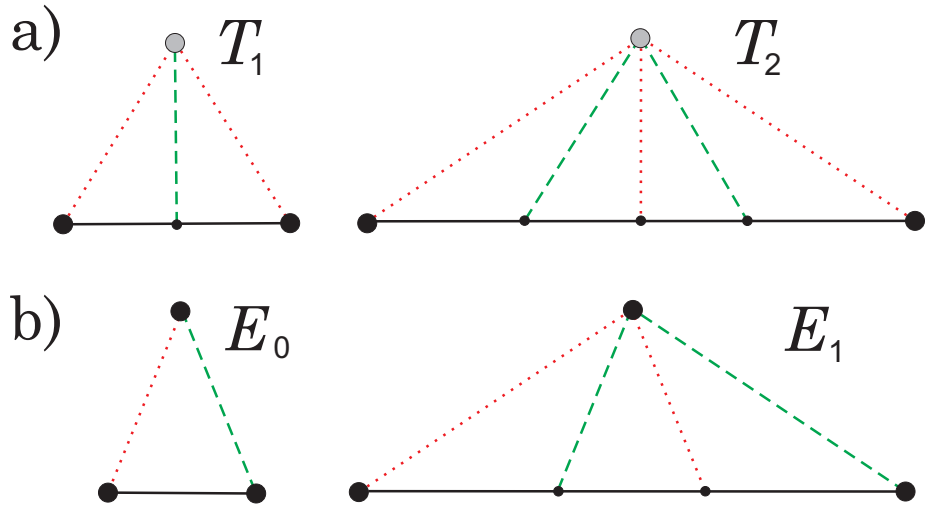}
\caption{Irreducible triangles $T_n$ and $E_n$.}\label{triangles-te}
\end{figure}

\begin{figure}
\centering
\includegraphics[width=4in]{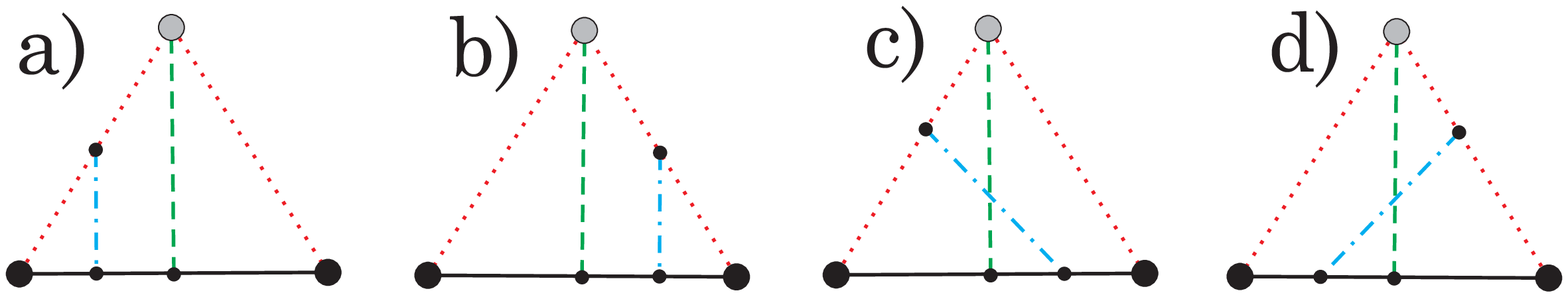}
\caption{Nets for a triangle $T_1$.}\label{triangles-t1}
\end{figure}

\begin{figure}
\centering
\includegraphics[width=4in]{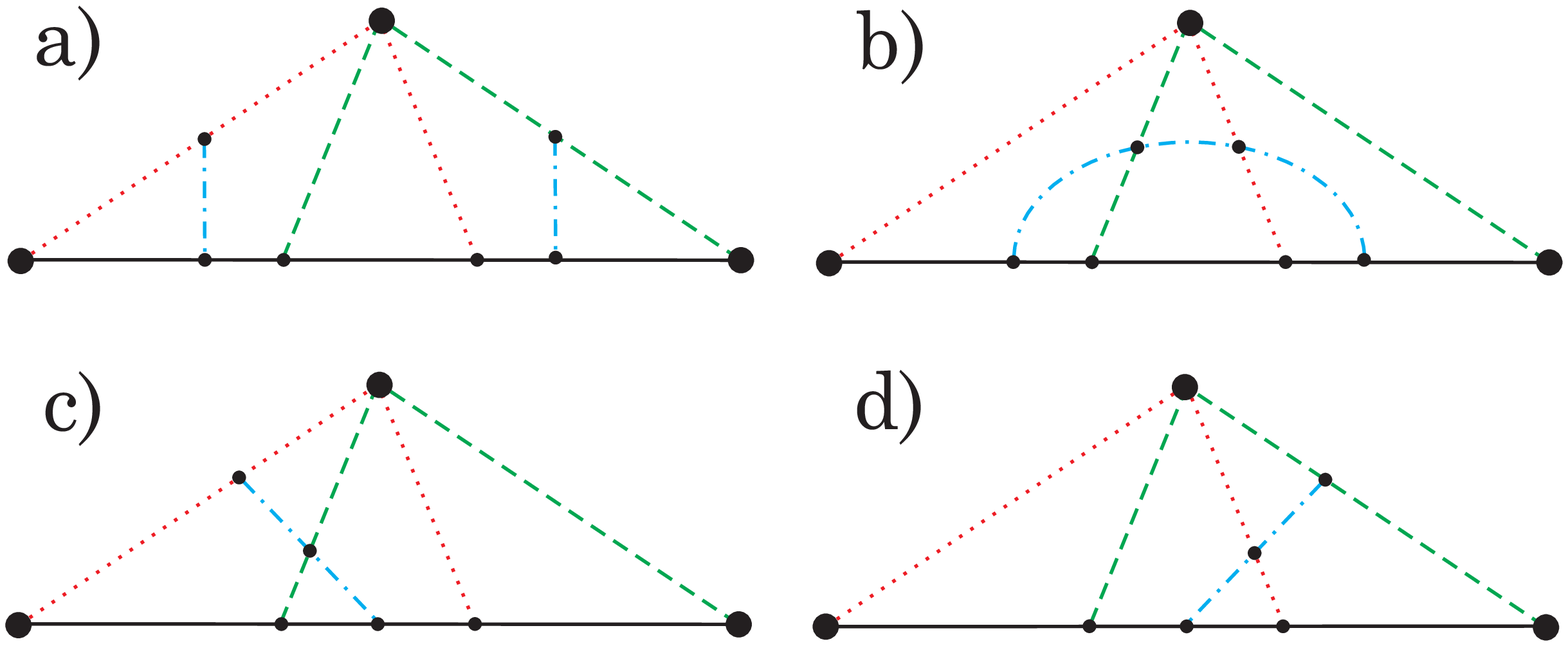}
\caption{Nets for a triangle $E_1$.}\label{triangles-e1}
\end{figure}

\subsection{Interior arcs of generic spherical quadrilaterals}\label{sub:arcs}
An arc $\gamma$ of a spherical quadrilateral $Q$ was defined (see Definition \ref{arcs}) as a simple path in the preimage $\Gamma_C$ in $Q$
of a circle $C$ of the partition $\P$, with the corners of $Q$ allowed only at the ends of $\gamma$. An arc is interior if it does not belong to a side of $Q$.
It follows from Theorem~\ref{diagonal} below (see Remark~\ref{rmk-arcs}) that an interior arc of a generic quadrilateral $Q$ cannot have both ends at the lateral vertices on the opposite sides of $Q$.

\begin{df}\label{df-arcs}{\rm
An interior arc $\gamma$ of a spherical quadrilateral $Q$ is {\sl one-sided} if both ends of $\gamma$ are on the same side of $Q$, and at least one of them is not a corner of $Q$.
If the ends of $\gamma$ are lateral vertices on two adjacent sides of $Q$, it is {\sl two-sided}.
If one end of $\gamma$ is a corner $p$ of $Q$, and another end
is a lateral vertex on the side of $Q$ not adjacent to $p$, it is a {\sl separator}.
A separator arc partitions $Q$ into a quadrilateral and a triangle.}
\end{df}

\begin{lemma}\label{nodiagonals}
The net of a generic quadrilateral $Q$ does not contain an arc with the ends at two opposite corners of $Q$.
\end{lemma}

\begin{proof}
Let $\gamma$ be an arc of the net of $Q$, and let $C$ be a circle of $\P$ such that $\gamma$ belongs
 to the preimage of $C$. Since the sides of $Q$ map to four distinct circles, only one of two opposite
 corners of $Q$ belongs to the preimage of $C$. Hence $\gamma$ cannot have the ends at two opposite corners of $Q$.
 \end{proof}

The opposite is also true: if a spherical quadrilateral $Q$ with all non-integer corners has two opposite sides mapped to the same circle then the net of $Q$, defined as the preimage of the partition $\TT$
of the Riemann sphere $\S$ by the three great circles to which the sides of $Q$ are mapped, contains an interior arc with the ends at two opposite corners of $Q$. This follows from a more general statement about spherical polygons over a three-circle partition of the sphere
(see \cite{EG2}, Theorem 2.2).

\begin{thm}\label{diagonal} Let $\TT$ be a partition of the Riemann sphere $\S$ by three distinct great circles.
Let $Q$ be a spherical $n$-gon having each side mapped to one of the circles of $\TT$, and all corners mapped to vertices of $\TT$.
If $n>3$ then the net of $Q$, defined as the preimage of the partition $\TT$, contains an interior arc with the ends at two non-adjacent corners of $Q$.
\end{thm}

\begin{rmk}\label{notwocircles} \normalfont
Note that a quadrilateral $Q$ with all non-integer corners cannot have all sides mapped to
only two circles. Due to Theorem~\ref{diagonal} such a quadrilateral would be a union of two spherical triangles with all corners at the corners of $Q$. Each of these triangles would have an integer corner, but a spherical triangle with an integer corner $p$ cannot have $p$ at the intersection of the two circles to which its sides
are mapped.
\end{rmk}

\begin{rmk}\label{rmk-arcs}\normalfont
Theorem~\ref{diagonal} implies that a generic spherical quadrilateral $Q$ cannot have an interior arc $\gamma$ with the ends at lateral vertices on opposite sides of $Q$. Such an arc would partition $Q$ into two quadrilaterals, one of them having all sides mapped to three circles (since one of the sides of $Q$ maps to the same circle of $\P$ as $\gamma$). According to Theorem~\ref{diagonal}, such a quadrilateral has an interior arc with the ends at its opposite corners. This is impossible since both corners adjacent to its side $\gamma$ have order $0$.
\end{rmk}

\begin{figure}
\centering
\includegraphics[width=4.8in]{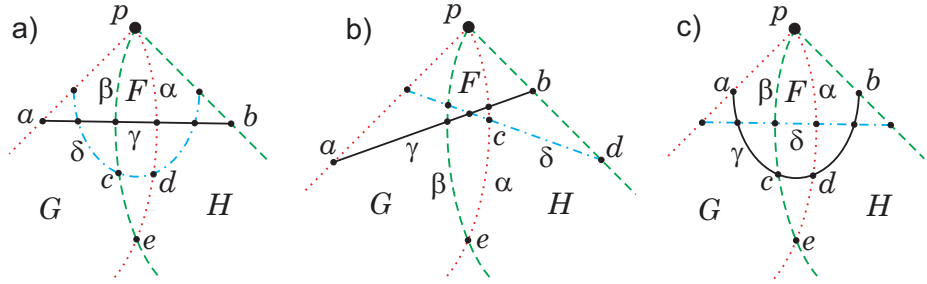}
\caption{Illustration for the proof of Lemma \ref{two-sided}.}\label{fig:twosided}
\end{figure}

\begin{figure}
\centering
\includegraphics[width=4in]{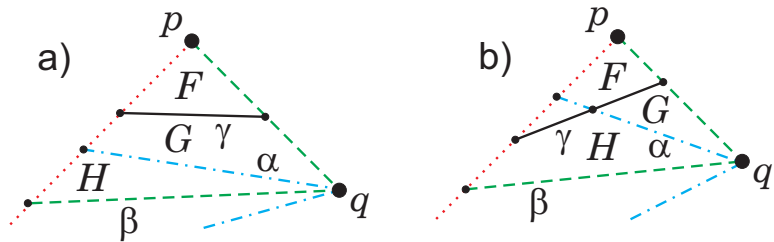}
\caption{Illustration for the proof of Lemma \ref{two-sided1}.}\label{fig:twosided1}
\end{figure}

\begin{lemma}\label{two-sided}. Let $\gamma$
be a two-sided arc of a generic irreducible quadrilateral $Q$ with
the ends $a$ and $b$ on the sides of $Q$ adjacent to its corner $p$.
Then $p$ has order $0$.
\end{lemma}

\begin{proof} We prove the statement by contradiction, assuming that the corner $p$ has order $1$ (the case of a corner with order greater than $1$ can be treated similarly).
Then $Q$ contains an irreducible triangle $E_1$ (see Figs.~\ref{triangles-te}b and \ref{triangles-e1}) with the vertices $p,\,a,\,b$, bounded by $\gamma$ and the two sides of $Q$ adjacent to $p$.
There are two interior arcs of $Q$, say $\alpha$ and $\beta$, both with one end at $p$, such that $\alpha$ is mapped to the same circle of $\P$ as the side $ap$ of $E_1$, and $\beta$ is mapped to
the same circle as its side $bp$. Fig.~\ref{triangles-e1} shows four possibilities for the net of $E_1$.
We consider two cases (see Figs.~\ref{fig:twosided}a and \ref{fig:twosided}b) corresponding to the nets in Figs.~\ref{triangles-e1}a and \ref{triangles-e1}c. Note that the net in Fig.~\ref{triangles-e1}d is reflection symmetric to that in Fig.~\ref{triangles-e1}c, and the net in Fig.~\ref{triangles-e1}b is obtained from Fig.~\ref{fig:twosided}a by exchanging the arcs $\gamma$ and $\delta$ (see Fig.~\ref{fig:twosided}c).

Let $F$ be the face of the net $\Gamma$ of $Q$ adjacent to $p$ and bounded by the arcs $\alpha$ and $\beta$.
In Fig.~\ref{fig:twosided}a the face $F$ is a triangle and the faces of $\Gamma$ adjacent to $F$ are quadrilaterals,
so there should be an interior arc $\delta$ at the boundary of each of these faces, mapped to a circle of $\P$ other than those for $\alpha$, $\beta$ and $\gamma$. Note that vertices $c$ and $d$ of $\delta$ are interior vertices of $\Gamma$: they cannot be corners of $Q$ since they are connected to $p$ by interior arcs.
For the same reason the intersection point $e$ of $\alpha$ and $\beta$ should be an interior vertex of $\Gamma$.
Thus $\Gamma$ should contain a face $G$ bounded by the side of $Q$ extending $ap$ and segments of $\gamma$, $\delta$, $\beta$, $\alpha$, and a face $H$ bounded by the side of $Q$ extending $bp$ and segments of $\gamma$, $\delta$, $\alpha$, $\beta$ (see Fig.~\ref{fig:twosided}a). This is a contradiction: a side of $Q$ and an interior
arc mapped to the same circle cannot belong to the boundary of the same face of $\Gamma$. The same arguments apply to the case shown in Fig.~\ref{fig:twosided}c
where the arcs $\gamma$ and $\delta$ are exchanged. Note that at least one end of the arc $\delta$ in Fig.~\ref{fig:twosided}c is not a corner of $Q$, since $Q$ is irreducible.

In Fig.~\ref{fig:twosided}b the face $F$ is a quadrilateral. The same arguments as before show that $\Gamma$ should contain a face $G$ bounded by the side of $Q$ extending $ap$ and segments of $\gamma$, $\delta$, $\beta$ and $\alpha$. This is a contradiction: a side of $Q$ and an interior
arc $\alpha$ mapped to the same circle cannot belong to the boundary of the same face $G$.
\end{proof}

\begin{lemma}\label{two-sided1}. Let $\gamma$ be a two-sided arc of a generic irreducible quadrilateral $Q$, with the ends on the sides of $Q$ adjacent to its corner $p$.
Then one of the sides of $Q$ adjacent to $p$ has another end at a corner $q$ of $Q$ of order at least 1.
The net $\Gamma$ of $Q$ has a face $H$ adjacent to $q$, such that the boundary of $H$ contains segments of two separator arcs with a common end at $q$ .
\end{lemma}

\begin{proof} According to Lemma \ref{two-sided}, the corner $p$ of $Q$ has order $0$,
thus there is a unique face $F$ of $\Gamma$ adjacent to $p$, which may be either triangular or quadrilateral.

If $F$ is a triangle (see Fig.~\ref{fig:twosided1}a) then $\gamma$ is a side of $F$. Let $G$ be the quadrilateral face of $\Gamma$ adjacent to $\gamma$ from the other side. Then the boundary $\alpha$ of $G$ opposite to $\gamma$ has both ends on the sides of $Q$ adjacent to $p$. The same arguments as in the proof of Lemma \ref{two-sided} show that $\alpha$ cannot be a two-sided arc, thus $\alpha$ has one end at a corner $q$ of $Q$.
Since $Q$ is irreducible, the other end of $\alpha$ cannot be a corner of $Q$, thus $\alpha$ is a separator arc.
Note that $q$ cannot have order $0$, otherwise $\alpha$ would be a side of $Q$, and $Q$ would be a triangle.
Thus the order of $q$ is at least $1$, and $\Gamma$ contains a triangular face $H$ adjacent to $q$ and bounded
by a side of $Q$ and two separator arcs $\alpha$ and $\beta$.

If $F$ is a quadrilateral (see Fig.~\ref{fig:twosided1}b) then one of its boundary edges belongs to $\gamma$,
and the triangular face $G$ of $\Gamma$ on the other side of $\gamma$ is bounded by a side of $Q$ and the edges
of arcs $\gamma$ and $\alpha$. Note that $\alpha$ cannot be a two-sided arc of $\Gamma$. If it were a two-sided arc,
the quadrilateral face $H$ of $\Gamma$ on the other side of $\alpha$ would have two boundary arcs (other than $\gamma$ and $\alpha$) belonging to two sides of $Q$ adjacent to $p$, intersecting at a vertex of $H$.
This is impossible, since the sides of $Q$ adjacent to $p$ do not intersect at any other point.
Thus the end $q$ of $\alpha$ is a corner of $Q$, and both arcs $\alpha$ and $\beta$ at the boundary of $H$ are separator arcs.
\end{proof}

\begin{figure}
\centering
\includegraphics[width=4.8in]{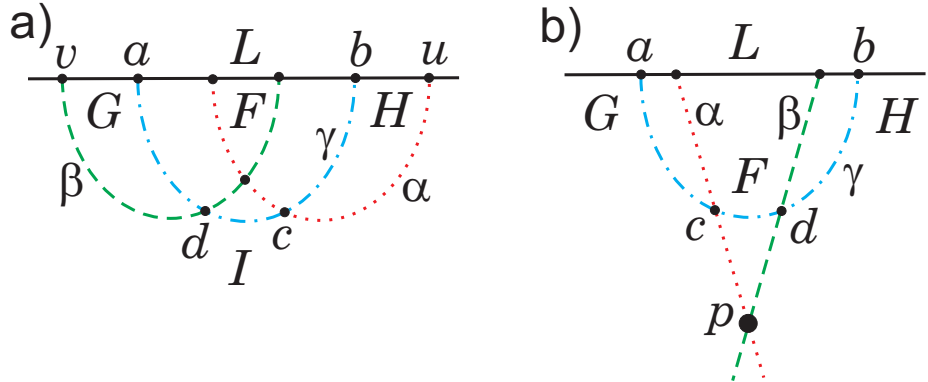}
\caption{Illustration for the proof of Lemma \ref{one-sided}.}\label{fig:onesided}
\end{figure}

\begin{lemma}\label{one-sided}
Let $\gamma$ be a one-sided arc of a generic primitive quadrilateral $Q$,
with the ends $a$ and $b$ on a side $L$ of $Q$.
Then there are two separator arcs of the net $\Gamma$ of $Q$
with a common end at a corner $p$ of $Q$ and the other ends
inside the segment $(a,b)$ of $L$.
If both $a$ and $b$ are not corners of $Q$ then the order of $p$ is greater than 1.
\end{lemma}

\begin{proof} Since the endpoints $a$ and $b$ of $\gamma$ belong to the intersection
of the same two circles, there are two lateral vertices of $\Gamma$ inside
the segment $(a,b)$ of $L$, corresponding to the intersections of $L$ with preimages of two circles
of $\P$ other than those to which $L$ and $\gamma$ are mapped.
Let $\alpha$ and $\beta$ be interior arcs of $\Gamma$ with endpoints at the two vertices inside $(a,b)$.
The face $F$ of $\Gamma$ bounded by the arcs $\alpha$ and $\beta$ is either a triangle
(see Fig.~\ref{fig:onesided}a) or a quadrilateral (see Fig.~\ref{fig:onesided}b).

If $F$ is a triangle then the faces $G$ and $H$ of $\Gamma$ adjacent to $a$ and $b$ outside $\gamma$ are triangular, thus the edges of arcs $\alpha$ and $\beta$
outside $\gamma$ (with endpoints $c$ and $d$, respectively) must have other ends mapped to the same circle as $L$. This is impossible since the face $I$ of $\Gamma$ (see Fig.~\ref{fig:onesided}a) cannot have two disjoint segments mapped to the same circle in its boundary. This argument holds also when either $a$ or $b$ is a corner of $Q$.

If $F$ is a quadrilateral then the arcs $\alpha$ and $\beta$ intersect at a point $p$ outside $\gamma$.
Note that $p$ cannot be a lateral vertex since both $cp$ and $dp$ are interior edges of $\Gamma$.
If $p$ were an interior vertex of $\Gamma$, the edges of arcs $\alpha$ and $\beta$ beyond $p$ would have
their other ends mapped to the same circle as $L$, which is impossible for the same reason as when $F$ a triangle.
Thus $p$ must be a corner of $Q$.
If both $a$ and $b$ are not corners of $Q$ then $p$ cannot have order 1. Otherwise, the boundary edges of $G$ and $H$ opposite $\gamma$ would be two sides of $Q$,
each of them having one end at $p$ and the other end on the side $L$ of $Q$, thus $Q$ would be a triangle, a contradiction.
\end{proof}

\begin{cor}\label{empty}
Let $Q$ be a generic quadrilateral with all four angles of order $0$.
Then the net of $Q$ does not have interior arcs.
\end{cor}

\begin{proof} Note that $Q$ is primitive and cannot have separator arcs. It follows from Lemmas \ref{two-sided1} and \ref{one-sided} that $Q$ cannot have two-sided or one-sided arcs.
\end{proof}

\medskip
\begin{figure}
\centering
\includegraphics[width=4.8in]{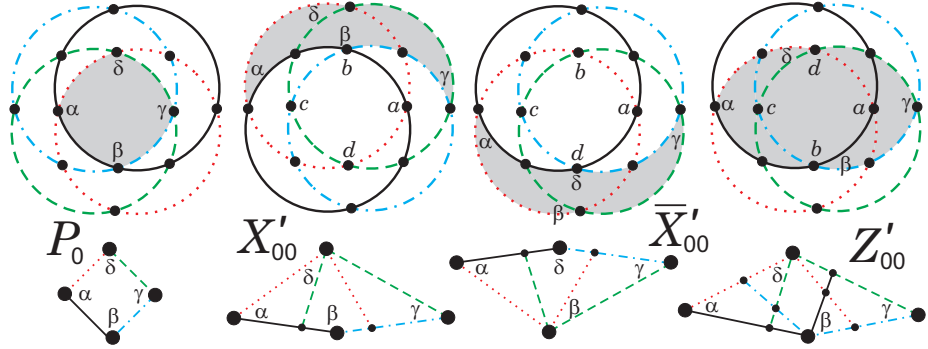}
\caption{Basic primitive quadrilaterals and their nets.}\label{basic}
\end{figure}

\section{Classification of nets of primitive and irreducible quadrilaterals}\label{section:primitive}
In this section, $Q$ is a primitive (see Definition \ref{arcs}) generic spherical quadrilateral.
It is shown below that $Q$ has at least two corners of order 0.
We order corners $(a_0,\dots,a_3)$ of $Q$ so that the order of $a_0$ is 0, and the sum of orders
of $a_0$ and $a_2$ does not exceed the sum of orders of $a_1$ and $a_3$.

\subsection{Extension of a side}\label{sub:extension}
Let $p$ be a corner of $Q$ of order 0, with the angle $\alpha<1$.
Let $L$ and $M$ be two sides of $Q$ adjacent to $p$. Suppose that $M$ has order at most 2, and let $q$ be the corner of $Q$ at the other end of $M$.
Then we can attach to $Q$ a spherical triangle $T_n$ with an integer corner at $q$ and two other corners at $p$ and $p'$, so that the side $[p,q]$ of $T_n$ is common with the side $M$ of $Q$, and the base $[p,p']$ of $T_n$ is extending $L$ beyond $p$. The union of $Q$ and $T_n$ is
a primitive spherical quadrilateral $Q'$ with the side $L'=L\cup [p,p']$, and the angle at its corner $p'$ equal $\alpha$ if $n$ is even and $1-\alpha$ if $n$ is odd. We call this operation {\sl extension of the side} $L$ of $Q$ beyond its corner $p$.
Note that extending $L'$ beyond $p'$ by attaching a triangle $T_m$ to $Q'$ is the same as a single extension attaching
a triangle $T_{n+m}$ to $Q$.

\begin{thm}\label{primitive}
Every generic primitive quadrilateral $Q$ can be obtained from one of the basic quadrilaterals $P_0$, $X'_{00}$, $\bar X'_{00}$ and $Z'_{00}$ (see Fig.~\ref{basic}) by at most two extension operations.
\end{thm}

Proof of Theorem \ref{primitive} will be given at the end of this section.
It implies that the nets of all primitive quadrilaterals belong to the following list
(see Figs.~\ref{4circles-x}, \ref{4circles-z}, \ref{4circles-r-s} and \ref{4circles-uvw} where the basic
quadrilateral is shaded).

\begin{figure}
\centering
\includegraphics[width=4.8in]{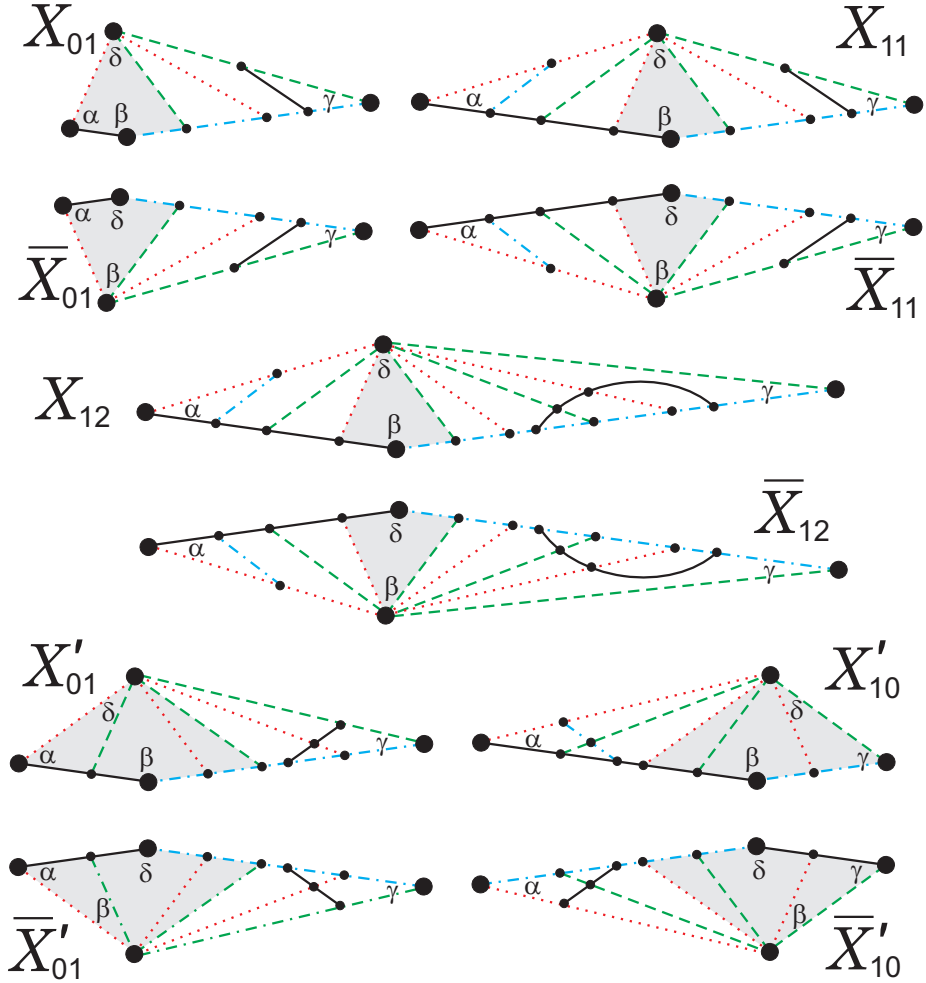}
\caption{Quadrilaterals of types $X_{kl}$, $\bar X_{kl}$, $X'_{kl}$ and $\bar X'_{kl}$.}\label{4circles-x}
\end{figure}

\begin{figure}
\centering
\includegraphics[width=4.8in]{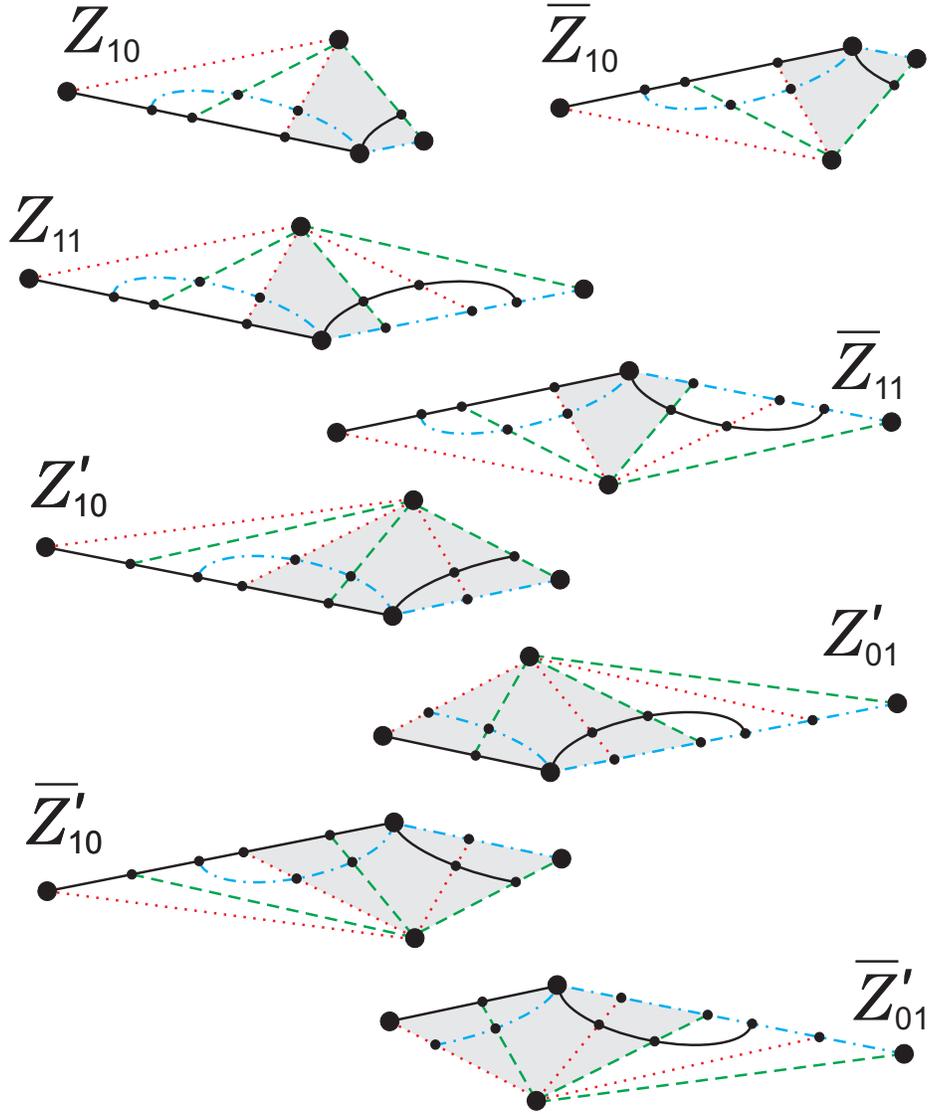}
\caption{Quadrilaterals of types $Z_{kl}$, $\bar Z_{kl}$, $Z'_{kl}$ and $\bar Z'_{kl}$.}\label{4circles-z}
\end{figure}

\begin{figure}
\centering
\includegraphics[width=4.8in]{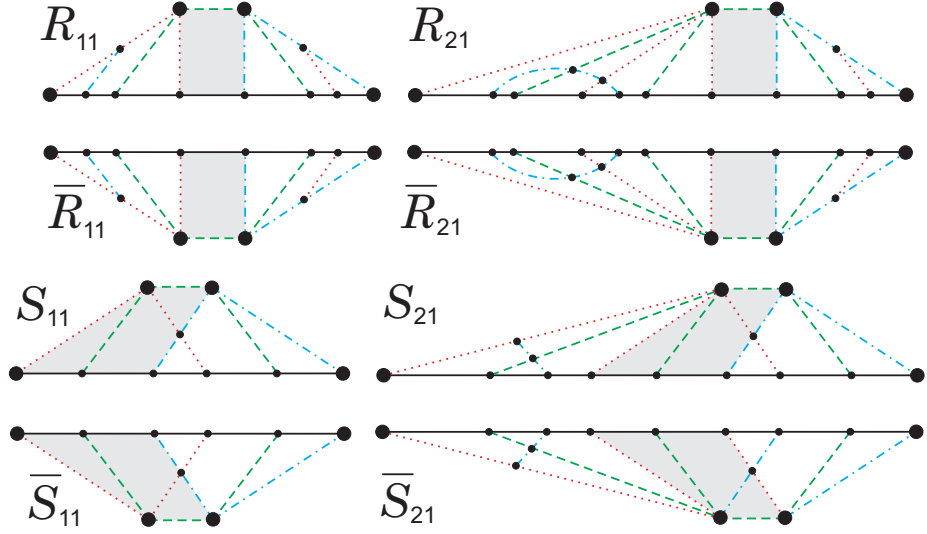}
\caption{Quadrilaterals of types $R_{kl}$, $\bar R_{kl}$, $S_{kl}$ and $\bar S_{kl}$.}\label{4circles-r-s}
\end{figure}

\begin{figure}
\centering
\includegraphics[width=4.8in]{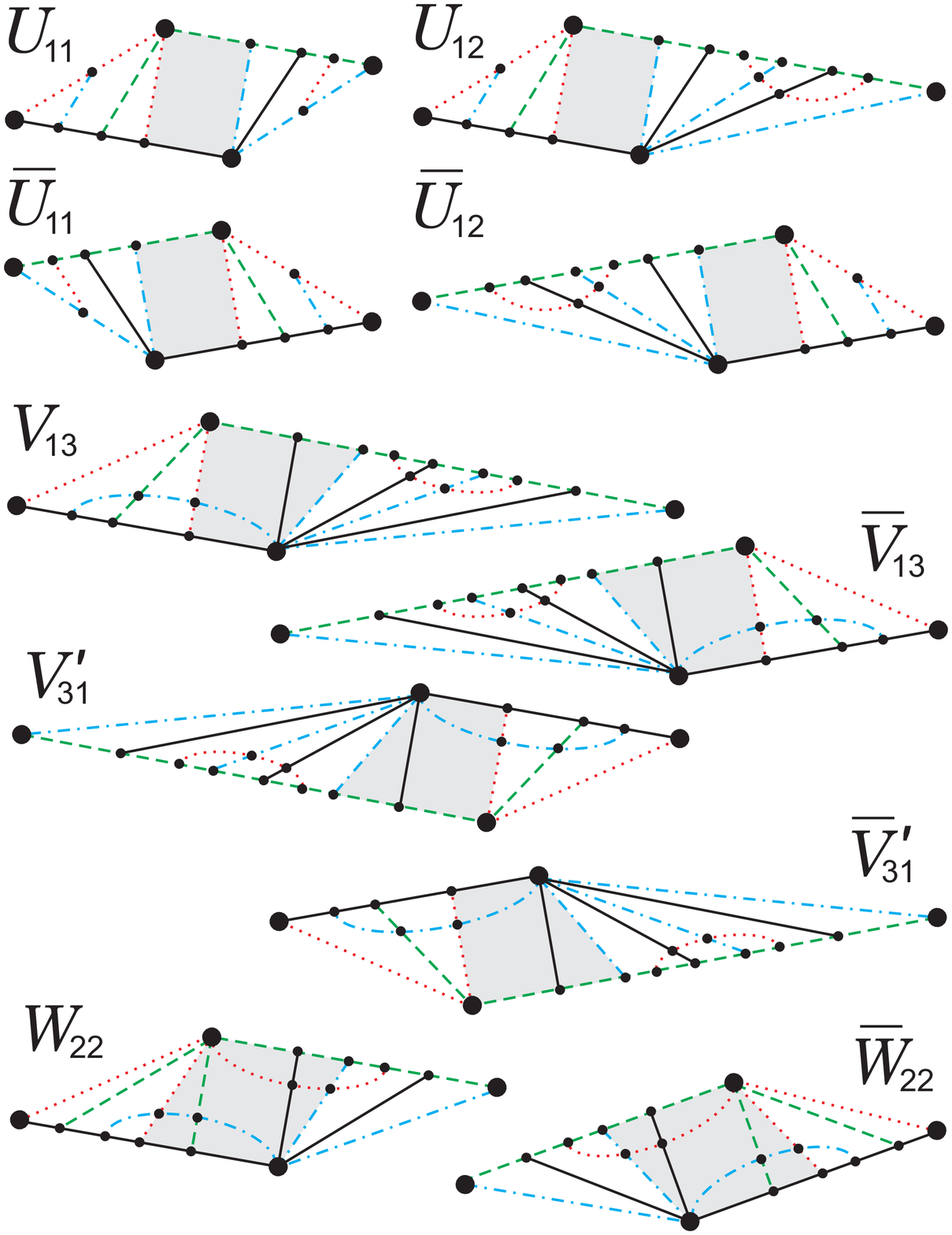}
\caption{Quadrilaterals of types $U_{kl}$, $\bar U_{kl}$, $V_{kl}$, $\bar V_{kl}$, $V'_{kl}$, $\bar V'_{kl}$, $W_{kl}$, $\bar W_{kl}$.}\label{4circles-uvw}
\end{figure}

\medskip
\noindent {\bf Type $X$.} A quadrilateral $X_{kl}$ for $k,l\ge 0$ and $k+l\ge 1$ (see Fig.~\ref{4circles-x}) can be obtained by attaching triangles $T_k$ and $T_l$ to adjacent sides of the basic quadrilateral $P_0$. When either $k=0$ or $l=0$, only one triangle is attached.
The quadrilateral $X_{kl}$ has one corner of order $k+l$, the other three corners being of order $0$.
A quadrilateral $\bar X_{kl}$ is reflection-symmetric (preserving opposite corners of order 0) to $X_{kl}$.

\medskip
\noindent {\bf Type $X'$.} A quadrilateral $X'_{kl}$ for $k,l\ge 0$ (see Fig.~\ref{4circles-x}) can be obtained by attaching triangles $T_k$ and $T_l$ to adjacent sides of the basic quadrilateral $X'_{00}$ so that $T_k$ and $T_l$ have a common vertex at the corner of $X'_{00}$ of order 1. The quadrilateral $X'_{kl}$ has one corner of order $k+l+1$, the other three corners being of order $0$.
A quadrilateral $\bar X'_{kl}$ is reflection-symmetric (preserving opposite corners of order 0) to $X'_{kl}$.

\medskip
\noindent {\bf Type $Z$.} A quadrilateral $Z_{kl}$ for $k,l\ge 0$ and $k+l\ge 1$ (see Fig.~\ref{4circles-z}) can be obtained by attaching triangles $T_k$ and $T_l$ to adjacent sides of the basic quadrilateral $\bar X'_{00}$ so that $T_k$ and $T_l$ have a common vertex at the corner of $\bar X'_{00}$ opposite to its corner of order 1. When either $k=0$ or $l=0$, only one triangle is attached. The quadrilateral $Z_{kl}$ has opposite corners of orders $k+l$ and $1$, the other two opposite corners being of order $0$.
A quadrilateral $\bar Z_{kl}$ is reflection-symmetric (preserving the corners of order 0) to $Z_{kl}$.

\medskip
\noindent {\bf Type $Z'$.} A quadrilateral $Z'_{kl}$ for $k,l\ge 0$ (see Fig.~\ref{4circles-z}) can be obtained by attaching triangles $T_k$ and $T_l$ to adjacent sides of the basic quadrilateral $Z'_{00}$ so that $T_k$ and $T_l$ have a common vertex at a corner of $Z'_{00}$ of order 1. The quadrilateral $Z'_{kl}$ has opposite corners of orders $k+l+1$ and $1$, the other two corners being of order $0$.
A quadrilateral $\bar Z'_{kl}$ is reflection-symmetric (preserving the corners of order 0) to $Z'_{kl}$.

\medskip
\noindent {\bf Type $R$.} A quadrilateral $R_{kl}$ for $k\ge l\ge 1$ (see Fig.~\ref{4circles-r-s}) can be obtained by attaching triangles $T_k$ and $T_l$ to opposite sides of the basic quadrilateral $P_0$ so that both triangles extend the same side of $P_0$. The quadrilateral $R_{kl}$ has adjacent corners of orders $k$ and $l$, the other two corners being of order $0$.
A quadrilateral $\bar R_{kl}$ is reflection-symmetric (preserving the corners of order 0) to $R_{kl}$.

\medskip
\noindent {\bf Type $S$.} A quadrilateral $S_{kl}$ for $k\ge l\ge 1$ (see Fig.~\ref{4circles-r-s}) can be obtained by attaching triangles $T_{k-1}$ and $T_l$ (or $T_k$ and $T_{l-1}$) to opposite sides of the basic quadrilateral $X'_{00}$, so that both triangles extend the same side of order 2 of $X'_{00}$.
The quadrilateral $S_{kl}$ has adjacent corners of orders $k$ and $l$, the other two corners being of order $0$.
 A quadrilateral $\bar S_{kl}$ is reflection-symmetric(preserving the corners of order 0) to $S_{kl}$.

\medskip
\noindent {\bf Type $U$.} A quadrilateral $U_{kl}$ for $k,l\ge 1$ (see Fig.~\ref{4circles-uvw}) can be obtained by attaching triangles $T_k$ and $T_l$ to opposite sides of the basic quadrilateral $P_0$, so that $T_k$ and $T_l$ have vertices at the opposite corners of the quadrilateral $P_0$ and extend its opposite sides.
The quadrilateral $U_{kl}$ has opposite corners of orders $k$ and $l$, the other two corners being of order $0$.
A quadrilateral $\bar U_{kl}$ is reflection symmetric (exchanging the opposite corners of order 0) to $U_{kl}$.

\medskip
\noindent {\bf Type $V$.} A quadrilateral $V_{kl}$ for $k\ge 1$, $l\ge 2$ (see Fig.~\ref{4circles-uvw}) can be obtained by attaching triangles $T_k$ and $T_{l-1}$ to opposite sides of the basic quadrilateral $\bar X'_{00}$, so that $T_{l-1}$ has its vertex at the corner of order 1 of $\bar X'_{00}$, and $T_k$ has its vertex at the opposite corner.
The quadrilateral $V_{kl}$ has opposite corners of orders $k$ and $l$, the other two corners being of order $0$.
A quadrilateral $\bar V_{kl}$ is reflection symmetric (exchanging the opposite corners of order 0) to $V_{kl}$. A quadrilateral $V'_{kl}$ is rotation symmetric to $V_{lk}$.
A quadrilateral $\bar V'_{kl}$ is reflection symmetric (exchanging the opposite corners of order 0) to  $V'_{kl}$.

\medskip
\noindent {\bf Type $W$.} A quadrilateral $W_{kl}$ for $k,l\ge 2$ (see Fig.~\ref{4circles-uvw}) can be obtained by attaching triangles $T_{k-1}$ and $T_{l-1}$ to opposite sides of the basic quadrilateral $Z'_{00}$, so that $T_{k-1}$ and $T_{l-1}$ have vertices at the opposite corners of order 1 of $Z'_{00}$ and extend its opposite sides.
The quadrilateral $W_{kl}$ has two opposite corners of orders $k$ and $l$, the other two corners being of order $0$.
A quadrilateral $\bar W_{kl}$ is reflection symmetric (exchanging the opposite corners of order 0) to $W_{kl}$.

Note that an extended side of a quadrilateral has order greater than 3. It is short (of order less than 6) only when extended by a single triangle $T_1$.
Quadrilaterals of types $R$ and $S$ have one extended side.
Quadrilaterals of types $X$, $X'$, $Z$, $Z'$ have either one extended side or two adjacent extended sides.
Quadrilaterals of types $U$, $V$, $W$ have two opposite extended sides.

\begin{lemma}\label{onebigangle}
Let $Q$ be a generic primitive quadrilateral with one corner $p$ of order greater than 0 and three other corners of order 0. Then the net of $Q$ is of the type either $X$ or $X'$, or one of their reflection symmetric quadrilaterals $\bar X$ and $\bar X'$.
\end{lemma}

\begin{proof} Let $q$ be the corner of $Q$ opposite to $p$, and let $F$ be the face of the net $\Gamma$
of $Q$ adjacent to $q$. It follows from Lemma \ref{two-sided1} that there are no two-sided arcs of $\Gamma$ with the ends on the sides of $Q$ adjacent to $q$.
Thus $F$ must be a quadrilateral face of $\Gamma$.
The vertex $a$ of $F$ opposite to its vertex $q$ cannot be an interior vertex of $\Gamma$. Otherwise
the two arcs of the boundary of $F$ would be two-sided, in contradiction with Lemma \ref{two-sided1}.
The vertex $a$ cannot be a lateral vertex of $\Gamma$, since in that case one of the arcs of the boundary
of $F$ would have two ends at the opposite sides of $Q$, which is forbidden, or would be a one-sided arc,
in contradiction with Lemma \ref{one-sided}. Thus $a=p$ is the corner of $Q$ opposite $q$.

It follows that $Q$ is a union of $F$ and either two triangles $T_k$ and $T_l$ with integer angles $k$ and $l$ at their common vertex $p$ attached to the sides of $F$ adjacent to $p$, such that $k+l>0$ is the order of $p$, or
two triangles $E_k$ and $E_l$ with non-integer corners of order $k$ and $l$ at their common vertex $p$ attached to the sides of $F$ adjacent to $p$, such that $k+l+1$ is the order of $p$.
In the first case, $Q$ has type $X_{kl}$ or $\bar X_{kl}$, in the second case $Q$ has type $X'_{kl}$ or $\bar X'_{kl}$.
\end{proof}

\begin{figure}
\centering
\includegraphics[width=3in]{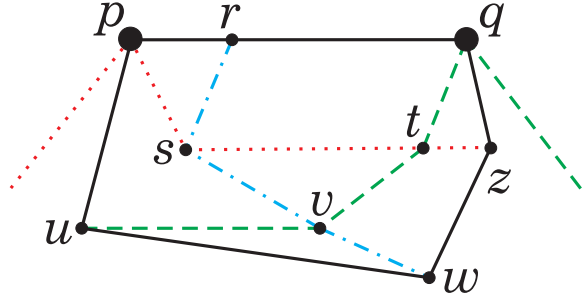}
\caption{Illustration for the proof of Lemma \ref{adjacentcorners}.}\label{fig:adjacent}
\end{figure}

\begin{lemma}\label{adjacentcorners}
Let $Q$ be a generic primitive quadrilateral with two adjacent corners $p$ and $q$
at the ends of its side $pq$.
If the angles of $Q$ at both $p$ and $q$ are greater than 1 then its side $pq$ has order 1.
\end{lemma}

\begin{proof}
Assume that $pq$ has order 2 (larger orders can be treated similarly).
Then the net of $Q$ contains two faces, one triangular and one quadrilateral, adjacent to $pq$.
We may assume that these two faces are $psr$ and $qrst$ (see Fig.~\ref{fig:adjacent}).
Note that the vertex $s$ cannot be a corner of $Q$ since it is connected to its corner $p$ by an interior arc.
Similarly, the vertex $t$ cannot be a corner of $Q$ as it is connected to $q$.
The vertex $s$ cannot be a boundary vertex of the net of $Q$ since $ps$ and $rs$ are interior arcs.
Thus $s$ is an interior vertex, and the net of $Q$ contains a quadrilateral face $puvs$ and a triangular face $svt$.
The vertex $t$ cannot be a boundary vertex since $st$ and $qt$ are interior arcs.
Thus $t$ is an interior vertex, and the net of $Q$ contains a quadrilateral face $tvwz$ and a triangular face $qtz$.
The vertex $v$ cannot be a corner of $Q$ as it is connected to $q$ by an interior arc.
It cannot be a boundary vertex, as $sv$ and $tv$ are interior arcs. Thus $v$ is an interior vertex,
and the net of $Q$ contains a triangular face $vuw$.
The same arguments as above show that both $u$ and $w$ should be interior vertices.
This contradicts irreducibility of $Q$, since $p$ and $q$ are connected by an interior arc $puwzq$.
\end{proof}

\begin{figure}
\centering
\includegraphics[width=4.8in]{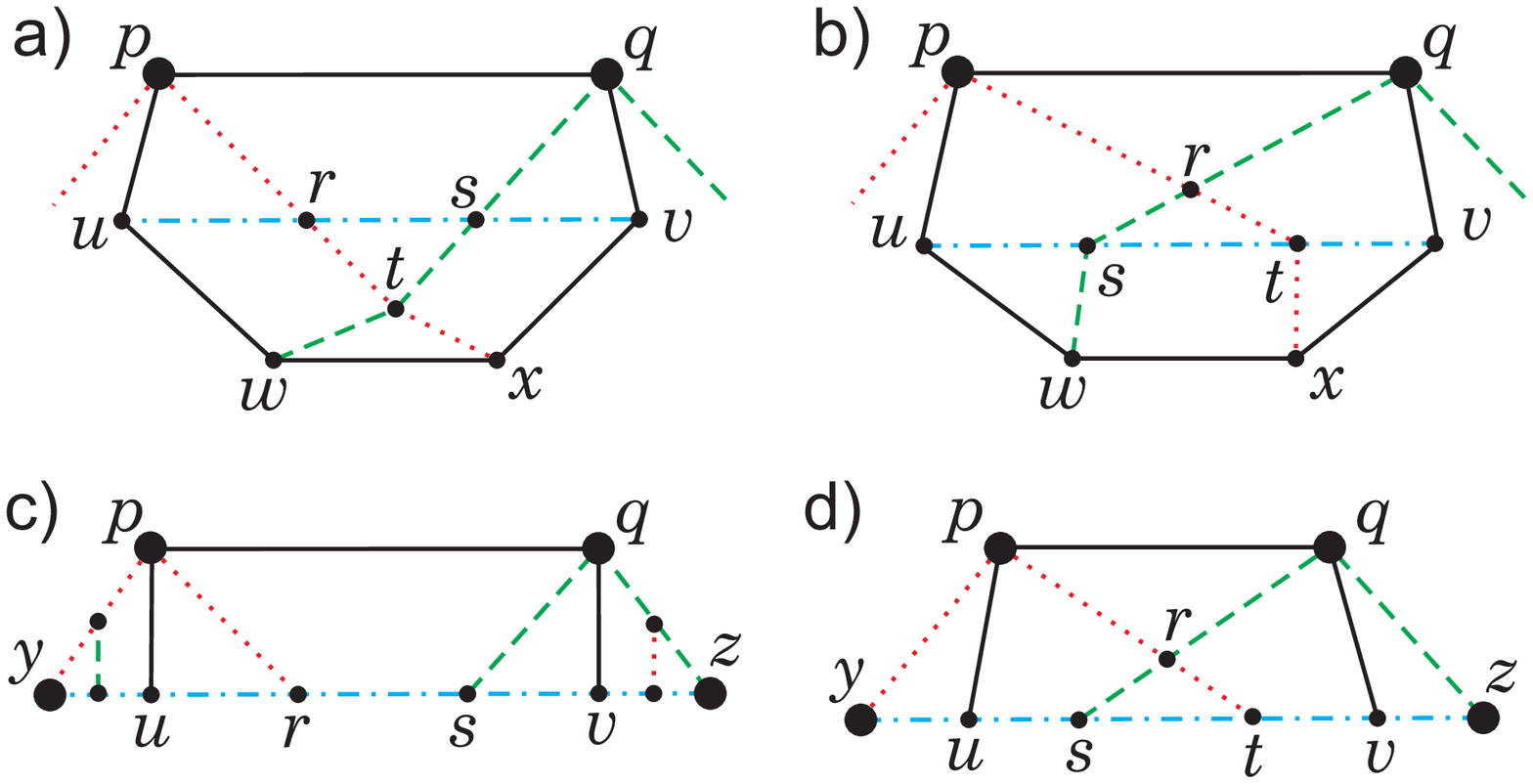}
\caption{Illustration for the proof of Lemma \ref{rs}.}\label{fig:rs}
\end{figure}

\begin{lemma}\label{rs}
Let $Q$ be a generic primitive quadrilateral with two corners $p$ and $q$
at the ends of its side $pq$, with both angles greater than 1.
Let $C$ be the circle of $\P$ to which the side of $Q$ opposite to $pq$ is mapped.
Then the net of $Q$ has no interior arcs mapped to $C$, and has the type either $R$ or $S$,
or one of their reflection symmetric quadrilaterals $\bar R$ and $\bar S$.
\end{lemma}

\begin{proof}
According to Lemma \ref{adjacentcorners}, the side $pq$ of $Q$ has order 1.
The face $F$ of the net $\Gamma$ of $Q$ adjacent to $pq$ may be either a quadrilateral $prsq$ (see Fig.~\ref{fig:rs}a) or a triangle $prq$ (see Fig.~\ref{fig:rs}b).

Consider first the case $F=prsq$. Note that its arc $rs$ is mapped to $C$.
Neither $r$ nor $s$ may be a corner of $Q$, since $r$ is connected to $p$ and $s$ is connected to $q$
by an interior arc of $\Gamma$. If one of these vertices, say $r$, is an interior vertex
of $\Gamma$, then $\Gamma$ contains the faces $pur$, $uwtr$ and $rts$.
It follows that $s$ is also an interior vertex of $\Gamma$, since it has two interior arcs $rs$ and $qs$
adjacent to it. Thus $\Gamma$ contains faces $ruwt$ and $stxv$.
Note that $t$ cannot be a corner of $Q$ since the two arcs intersecting at $t$ are preimages of the circles
of $\P$ corresponding to two opposite sides of $Q$. Since $rt$ and $st$ are interior arcs of $\Gamma$,
$t$ must be an interior vertex of $\Gamma$, and $twx$ is a face of $\Gamma$.
This implies that the arc $puwxvq$ of $\Gamma$ connects $p$ and $q$, contradicting irreducibility of $Q$.
Thus both $r$ and $s$ must be boundary vertices of $\Gamma$, and $ursv$ is part of the side
of $Q$ opposite $pq$. Extending this side till the corners $y$ and $z$ of $Q$ results in a quadrilateral
$R_{kl}$ or $\bar R_{kl}$, the union of $prsq$ and two triangles $T_k$ and $T_l$ (see Fig.~\ref{fig:rs}c
where $k=l=1$).

Next we consider the case when $F=prq$ is a triangle.
Since $r$ is not a corner of $Q$ (it is connected to both $p$ and $q$ by interior arcs)
it must be an interior vertex of $\Gamma$. Thus $\Gamma$ contains the faces $pusr$, $rst$ and $qrtv$
(see Fig.~\ref{fig:rs}b).
Note that neither $s$ nor $t$ may be corners of $Q$ (they are connected by interior arcs to $q$ and $p$, respectively).
If one of these vertices, say $s$, is an interior vertex of $\Gamma$ then $suw$ and $swxt$ are faces of $\Gamma$,
thus $t$ is also an interior vertex of $\Gamma$ (it has interior arcs $rt$ and $st$ adjacent to it)
and $txv$ is a face of $\Gamma$.
This implies that the arc $puwxvq$ of $\Gamma$ connects $p$ and $q$, contradicting irreducibility of $Q$.
Thus both $s$ and $t$ must be boundary vertices of $\Gamma$, and $ustv$ is part of the side
of $Q$ opposite $pq$. Extending this side till the corners $y$ and $z$ of $Q$ results in a quadrilateral
$S_{kl}$ or $\bar S_{kl}$, the union of $prq$ and two triangles $T_k$ and $T_l$ (see Fig.~\ref{fig:rs}c
where $k=l=1$ intersecting over $rst$).
\end{proof}

\begin{figure}
\centering
\includegraphics[width=4.8in]{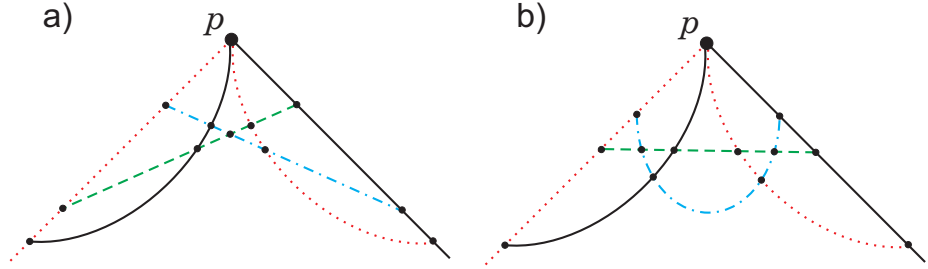}
\caption{Illustration for the proof of Lemma \ref{twocorners1}.}\label{angle1}
\end{figure}

\begin{figure}
\centering
\includegraphics[width=4.in]{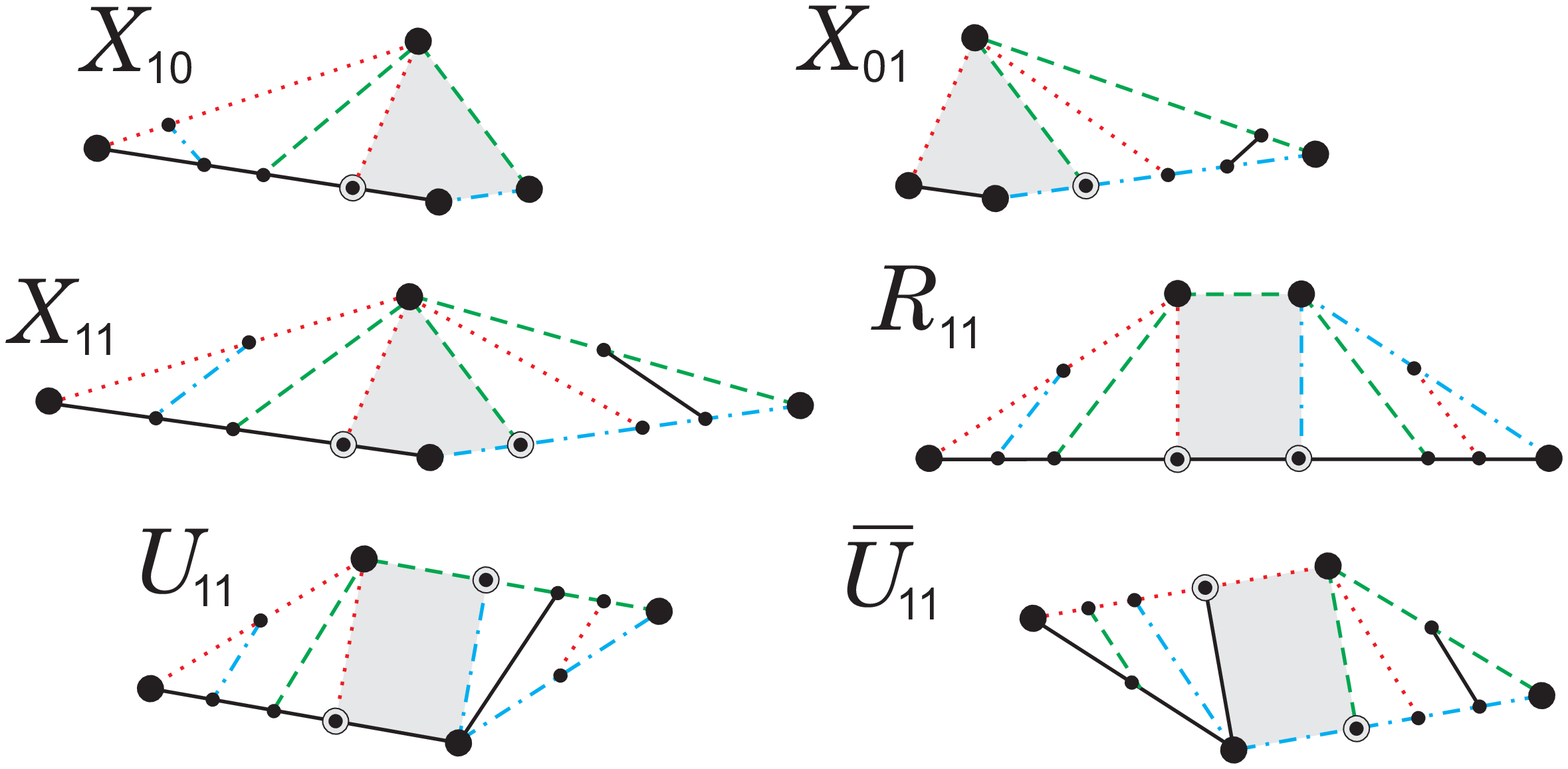}
\caption{Attaching one or two triangles $T_1$ to the quadrilateral $P_0$.}\label{boxt}
\end{figure}

\begin{figure}
\centering
\includegraphics[width=4.in]{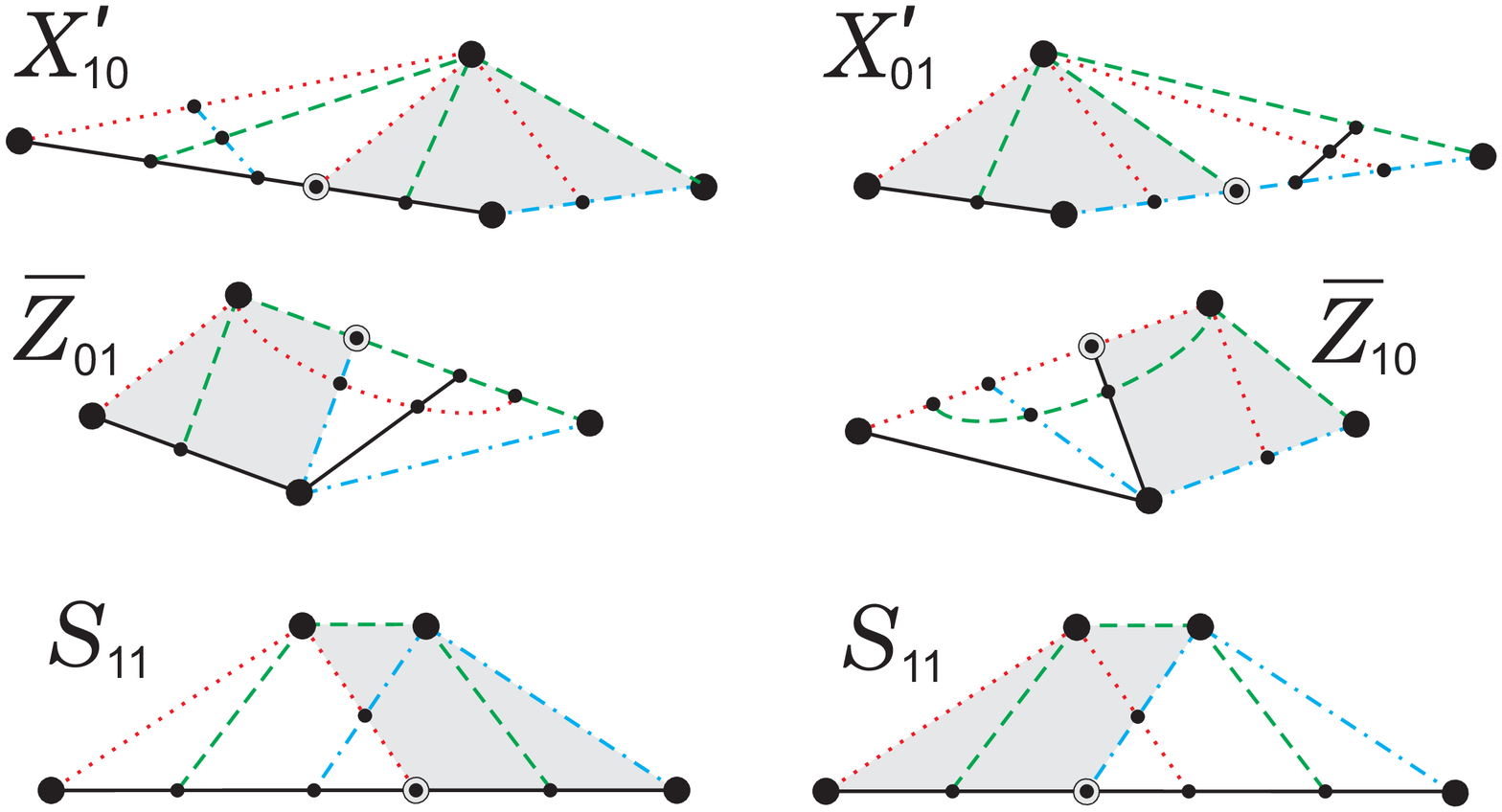}
\caption{Attaching a triangle $T_1$ to the quadrilateral $X'_{00}$.}\label{xprime}
\end{figure}

\begin{figure}
\centering
\includegraphics[width=4.8in]{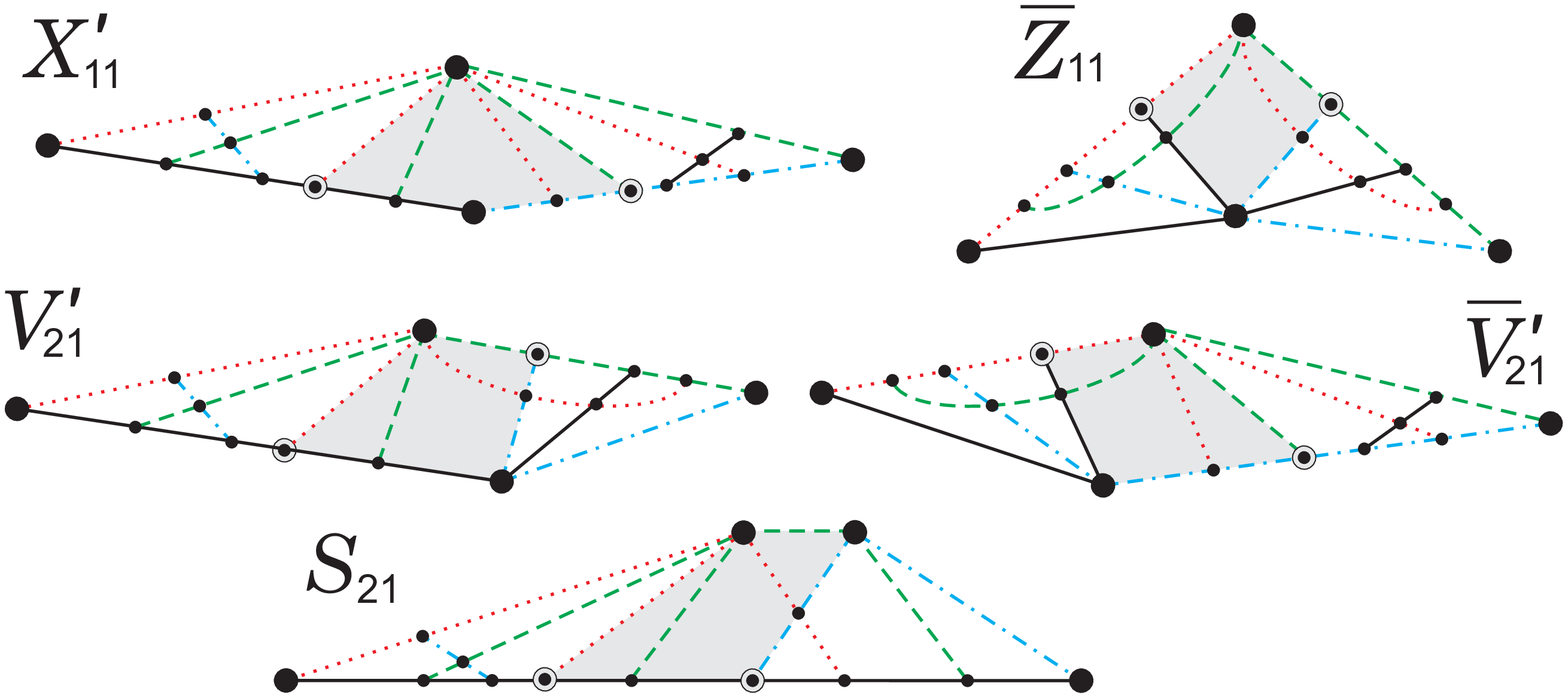}
\caption{Attaching two triangles $T_1$ to the quadrilateral $X'_{00}$.}\label{xprimet}
\end{figure}

\begin{figure}
\centering
\includegraphics[width=4.8in]{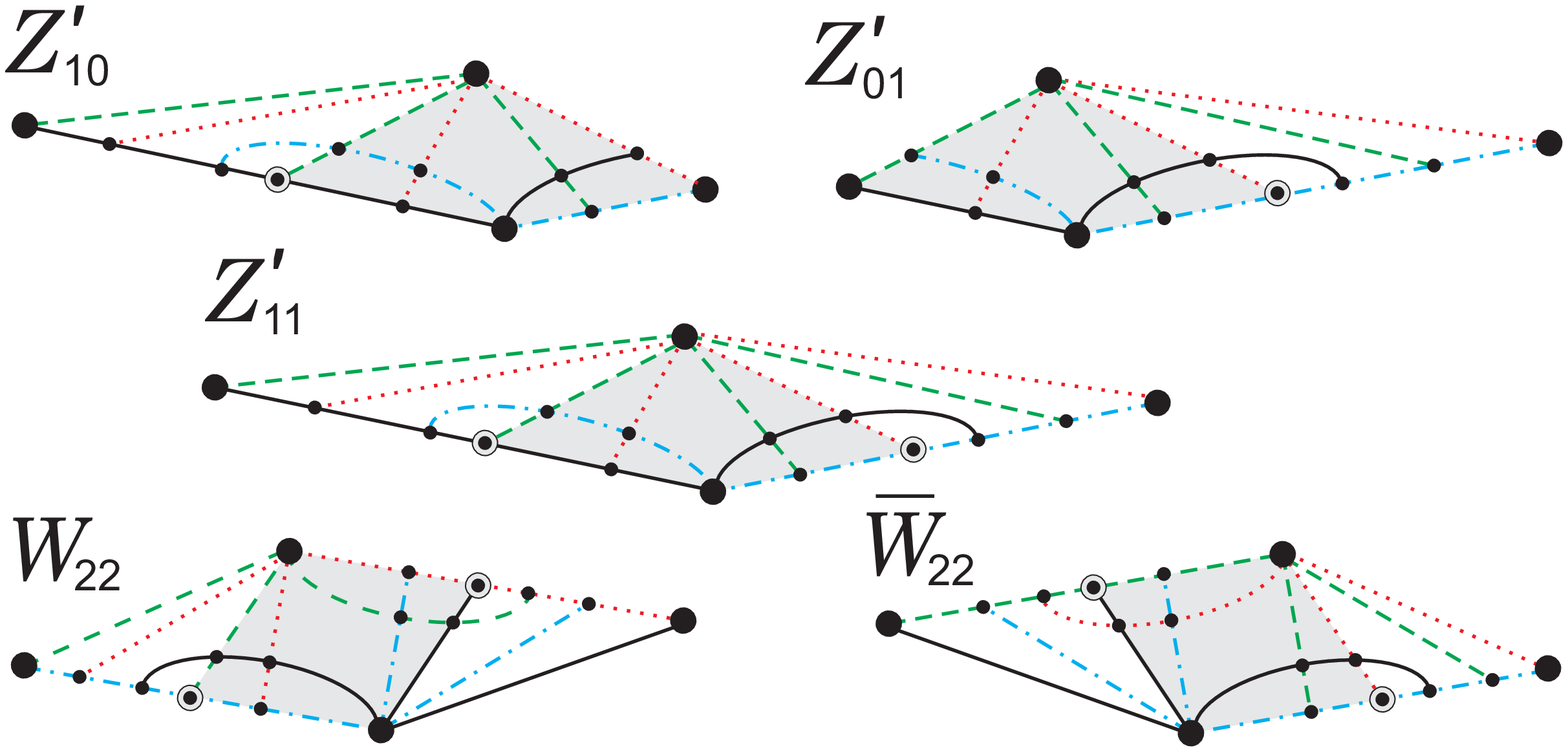}
\caption{Attaching one or two triangles $T_1$ to the quadrilateral $Z'_{00}$.}\label{zprime}
\end{figure}

\begin{lemma}\label{twocorners1}
Let $Q$ be a generic primitive quadrilateral with two opposite corners $p$ and $q$ of order $1$,
and two other corners of order $0$.
Then the net $\Gamma$ of $Q$ is one of the following: $Z'_{00}$, $Z_{01}$, $Z_{10}$, $\bar Z_{01}$, $\bar Z_{10}$, $U_{11}$, $\bar U_{11}$.
\end{lemma}

\begin{proof}
If $Q$ does not have one-sided arcs, there are two separator arcs $\alpha$ and $\alpha'$ with a common end at $p$, and
two separator arcs $\beta$ and $\beta'$ with a common end at $q$.
If the other ends of $\alpha$ and $\alpha'$ are on different sides of $Q$ then
the same is true for $\beta$ and $\beta'$, thus the net of $Q$ is $Z'_{00}$.
If the other ends of $\alpha$ and $\alpha'$ are on the same side of $Q$ then
the same is true for $\beta$ and $\beta'$, thus the net of $Q$ is either $U_{11}$ or $\bar U_{11}$.

Next, $Q$ may have a one-sided arc $\gamma$ with one end at a corner of $Q$
(due to Lemma \ref{one-sided}, if both ends of $\gamma$ were not at corners of $Q$
then $Q$ would have a corner of order greater than 1, a contradiction).
Then the net of $Q$ is one of $Z_{01}$, $Z_{10}$, $\bar Z_{01}$, $\bar Z_{10}$.

If two one-sided arcs were adjacent to the same corner of $Q$, Lemma \ref{one-sided} would imply that there should be at least 4
separator arcs with the common end at the opposite corner, a contradiction.
Finally, having two one-sided arcs with the ends at opposite corners of $Q$ would imply that each of these
corners must have, in addition to a one-sided arc, two separator arcs with a common end in it, a contradiction.
This completes the proof of Lemma \ref{twocorners1}.
\end{proof}

\begin{lemma}\label{cut}
Let $Q$ be a generic primitive quadrilateral with a corner $p$ of order greater than 1.
Then there is a separator arc of the net $\Gamma$ of $Q$ partitioning $Q$ into a generic primitive quadrilateral $Q'$ smaller than $Q$ and an irreducible triangle $T$ with an integer corner.
\end{lemma}

\begin{proof} If the net $\Gamma$ has a one-sided arc $\gamma$ with both ends on the side $L$ of $Q$ then there are, according to Lemma \ref{one-sided}, two separator arcs of $\Gamma$ with a common end at a corner of $Q$ and the other ends on $L$ inside $\gamma$.
One of these two arcs partitions $Q$ into a quadrilateral with one integer corner and a triangle with all non-integer corners, and another one partitions $Q$ into a generic quadrilateral $Q'$ smaller than $Q$ and a triangle $T$
with an integer corner. Thus we may assume that $\Gamma$ does not have any one-sided arcs.
In this case, there are at least four separator arcs with a common end at the corner $p$ of $Q$.
At least two of these arcs must be at the boundary of the same face of $\Gamma$ adjacent to $p$
and have their other ends on the same side of $Q$. Then one of these two arcs partitions $Q$ into a quadrilateral with one integer corner and a triangle with all non-integer corners, and another one partitions $Q$ into a generic quadrilateral $Q'$ smaller than $Q$ and a triangle $T$ with an integer corner. This completes the proof of Lemma \ref{cut}.
\end{proof}

\medskip
\begin{proof}[Proof of Theorem \ref{primitive}]
According to Lemma \ref{cut}, every generic primitive quadrilateral with a corner of order greater than 1 can be partitioned into a smaller generic quadrilateral and a triangle along one of the separator arcs of its net.
This implies that every generic primitive quadrilateral $Q$ can be obtained by attaching triangles
with integer corners to a quadrilateral $Q'$ with all corners of order at most 1.
If all corners of $Q'$ have order 0 then $Q'$ is the quadrilateral $P_0$ (see Corollary \ref{empty}).
If only one corner of $Q'$ has order 1, the other three corners having order 0, then the net of $Q'$ is one of
$X_{01}$, $X_{10}$, $X'_{00}$, $\bar X_{01}$, $\bar X_{10}$, $\bar X'_{00}$ according to Lemma \ref{onebigangle}.
Note that each of the quadrilaterals $X_{01}$, $X_{10}$, $\bar X_{01}$ and $\bar X_{10}$ can be partitioned into the quadrilateral $P_0$ and a triangle $T_1$.
If two adjacent corners of $Q'$ have order 1 then the net of $Q'$ is either $R_{11}$ or $S_{11}$ according to Lemma \ref{rs}.
If two opposite corners of $Q'$ have order 1 then the net of $Q'$ is either $Z'_{00}$ or one of
$Z_{01}$, $Z_{10}$, $\bar Z_{01}$, $\bar Z_{10}$, $U_{11}$, $\bar U_{11}$, according to Lemma \ref{twocorners1}.
Note that each of the latter six quadrilaterals can be further reduced to either $P_0$ or $X'_{00}$ or $\bar X'_{00}$.

There are eight options for attaching a triangle $T$ with an integer corner to the quadrilateral $P_0$.
The vertex of $T$ can be placed at any of the four corners of the quadrilateral, and the base of $T$ at the extension of one of its two sides opposite to that corner.
Combining these options would result in quadrilaterals $X_{kl}$, $\bar X_{kl}$, $R_{kl}$, $\bar R_{kl}$, $U_{kl}$ and $\bar U_{kl}$ (see Fig.~\ref{boxt}).

There are six options for attaching a triangle $T$ with an integer corner to the quadrilateral $X'_{00}$ (and to $\bar X'_{00}$).
The vertex of $T$ can be placed at the corner $p$ of order 1, and the base of $T$ at the extension
of one of its two sides opposite to $p$. Alternatively, the vertex of $T$ can be placed at the corner $q$ opposite $p$ and the base of $T$ at the extension of one of the two sides adjacent to $p$.
Finally, the vertex of $T$ can be placed at one of the corners $u$ and $v$ other than $p$ and $q$, and the base of $T$ at the extension of the side of order 2 not adjacent to the corner where the vertex of $T$ is placed.
Fig.~\ref{xprime} shows six options of attaching the triangle $T_1$ to the quadrilateral $X'_{00}$.
Attaching $T_1$ to $\bar X'_{00}$ corresponds to replacing all quadrilaterals in Fig.~\ref{xprime} by their reflections preserving the corners of order 0.
Combining these options would result in quadrilaterals $X'_{kl}$, $\bar X'_{kl}$, $Z_{kl}$, $\bar Z_{kl}$, $V_{kl}$, $\bar V_{kl}$, $V'_{kl}$, $\bar V'_{kl}$, $S_{kl}$, $\bar S_{kl}$ (see Fig.~\ref{xprimet}).

There are four options to attach a triangle $T$ to the quadrilateral $Z'_{00}$. The vertex of $T$ can be placed at one of the corners of order 1 of the quadrilateral, and the base of $T$ at the extension of one of its two sides opposite to that corner.
Combining these options would result in quadrilaterals $Z'_{kl}$, $\bar Z'_{kl}$, $W_{kl}$ and $\bar W_{kl}$ (see Fig.~\ref{zprime}).

Summing up, all primitive quadrilaterals that can be obtained by attaching at most two irreducible triangles, each with an integer corner, to one of the basic quadrilaterals $P_0$, $X'_{00}$, $\bar X'_{00}$ and $Z'_{00}$, appear in the list of primitive quadrilaterals in Theorem \ref{primitive}. This completes the proof of Theorem \ref{primitive}.
\end{proof}

\begin{cor}\label{irreducible}
Every generic irreducible quadrilateral is either primitive or is obtained from a primitive quadrilateral $Q$ listed in Theorem \ref{primitive}, of type other than $R$, $S$, $\bar R$ and $\bar S$, by replacing a quadrilateral face of the net of $Q$ adjacent to its two opposite corners with the quadrilateral $P_\mu$, for some $\mu>0$ (see Definition \ref{pmu}).
\end{cor}

\begin{proof} If $Q$ is an irreducible quadrilateral that is not primitive then, according to Lemma \ref{p1}, $Q$ contains a quadrilateral $P_1$ (see Fig.~\ref{pseudo-diagonal}).
In particular, the net of $Q$ has a quadrilateral face adjacent to two of its opposite corners.
Replacing $P_1$ with $P_0$, we obtain a smaller quadrilateral. We can repeat this operation $\mu$ times
until we get a primitive quadrilateral $Q'$. Since $Q'$ still has a quadrilateral face adjacent to two of its opposite corners, it should belong to one of the types listed in Theorem \ref{primitive} other than $R$, $S$, $\bar R$ and $\bar S$.
The original quadrilateral $Q$ is obtained from $Q'$ by replacing its quadrilateral face adjacent to two of its opposite corners with the quadrilateral $P_\mu$, as stated in Theorem \ref{irreducible}.
\end{proof}

\medskip
\noindent{\bf Notation.} The irreducible quadrilaterals obtained from the primitive quadrilaterals
$X_{kl}$, $\bar X_{kl}$, $X'_{kl}$, $\bar X'_{kl}$, $Z_{kl}$, $\bar Z_{kl}$, $Z'_{kl}$, $\bar Z'_{kl}$, $U_{kl}$, $\bar U_{kl}$, $V_{kl}$, $\bar V_{kl}$, $V'_{kl}$, $\bar V'_{kl}$, $W_{kl}$, $\bar W_{kl}$ by replacing a quadrilateral face of their net by the quadrilateral $P_\mu$
are denoted $X^\mu_{kl}$, $\bar X^\mu_{kl}$, $X'^\mu_{kl}$, $\bar X'^\mu_{kl}$, $Z^\mu_{kl}$, $\bar Z^\mu_{kl}$, $Z'^\mu_{kl}$, $\bar Z'^\mu_{kl}$, $U^\mu_{kl}$, $\bar U^\mu_{kl}$, $V^\mu_{kl}$, $\bar V^\mu_{kl}$, $V'^\mu_{kl}$, $\bar V'^\mu_{kl}$, $W^\mu_{kl}$, $\bar W^\mu_{kl}$, respectively.
The original primitive quadrilaterals are assigned the same notation with $\mu=0$.

\section{Classification of nets of generic spherical quadrilaterals}\label{generic}

\begin{thm}\label{quad} Any generic spherical quadrilateral can be obtained by attaching spherical digons of types either $D_{15}$ or $D_{24}$ to some of the short (of order less than 6) sides of an irreducible quadrilateral $Q_0$. The types of digons attached to the sides of $Q_0$ are completely determined by its net.
\end{thm}

\begin{proof} According to Lemma \ref{nodiagonals}, any diagonal arc $\gamma$ of a reducible quadrilateral $Q$ must have its ends at adjacent corners $p$ and $q$ of $Q$. Thus $\gamma$ partitions $Q$ into a digon $D$ and a quadrilateral $Q'$. Note that $D$ must have integer angles at its corners. Otherwise,
the common side of $D$ and $Q'$ would be mapped to a circle $C'$ different from the circle $C$ to which the common side of $D$ and $Q$ is mapped.
This would imply that $p$ and $q$ are mapped to the intersection of the same two circles $C$ and $C'$, which is not possible since $Q$ is a generic quadrilateral.
Thus $D$ should be of type either $D_{15}$ (when the order $k$ of $\gamma$ is odd) or $D_{24}$ (when $k$ is even).
\end{proof}

\smallskip
\begin{rmk}\label{non-unique}\normalfont
An irreducible quadrilateral $Q_0$ in Theorem \ref{quad} may be not unique.
If an irreducible quadrilateral $Q_0$ has a short side $L$ of order greater than 3,
and the quadrilateral $Q$ obtained by attaching a disk $D$ to $L$
contains a disk $D'$ other than $D$, with part of the boundary of $D'$ being at a side
of $Q$, then $D'$ can be removed from $Q$ to obtain an irreducible quadrilateral $Q_1$.
All such non-uniqueness cases are listed below. All other cases can be obtained from these
by attaching digons or pseudo-diagonals.
Note that a quadrilateral $Q$ obtained by attaching a disk $D$ to a side of order less than $3$ of an irreducible quadrilateral $Q_0$ does not contain a disk other than $D$ having part of its boundary on the side of $Q$.
\end{rmk}

\subsection{Non-uniqueness cases in Remark \ref{non-unique}}\label{sub:non-unique}

(a) The quadrilateral $S_{11}\cup D_{15}$ (see Fig.~\ref{xs}a) contains three more disks. Removing each of them results in a quadrilateral $S_{11}$.

(b) The quadrilateral $X_{01}\cup D_{24}$ (see Fig.~\ref{xs}b) contains one more disk $D_{24}$. Removing it results in a quadrilateral $X_{10}$ (see Fig.~\ref{xs}c). In the opposite direction, removing a disk from $X_{10}\cup D_{24}$ may result in $X_{01}$.

(c) The quadrilateral $X_{1k}\cup D_{24}$ with $k>0$ contains one more disk $D_{24}$ (shaded in  Fig.~\ref{xu}a). Removing it results in a quadrilateral $U_{k1}$. In the opposite direction, removing a disk from $U_{k1}\cup D_{24}$ results in $X_{1k}$. Similarly, removing a disk from $X_{k1}\cup D_{24}$ with $k>0$ (shaded in Fig.~\ref{xu}b) results in $\bar U_{k1}$,
    and removing a disk from $\bar U_{k1}\cup D_{24}$ results in $X_{k1}$.
    Note that $X_{11}$ allows to attach a disk $D_{24}$ to any of its two sides of order 4.
    Removing a disk from $X_{11}\cup D_{24}$ results either in $U_{11}$ or in $\bar U_{11}$, depending on the side of $X_{11}$ to which the disk is attached (see Figs.~\ref{xu}a and \ref{xu}b).
    The cases $\bar X_{1k}\cup D_{24}$ and $\bar X_{k1}\cup D_{24}$ are obtained by reflection symmetry preserving the opposite corners of order 0.

(d) The quadrilateral $X'_{1k}\cup D_{15}$ contains a disk $D_{24}$ (shaded in Fig.~\ref{xu}c). Removing it results in a quadrilateral $V_{k+1,1}$ (a quadrilateral $\bar Z_{10}$ when $k=0$, see Fig.~\ref{xu}d). In the opposite direction, removing a disk $D_{15}$ from $V_{k+1,1}\cup D_{24}$ (from $\bar Z_{10}\cup D_{24}$ if $k=0$) results in $X'_{1k}$. Similarly, removing a disk $D_{24}$ (shaded in Fig.~\ref{xu}e) from $X'_{k1}\cup D_{15}$ results in $\bar V_{k+1,1}$ (in $\bar Z_{01}$ if $k=0$, see Fig.~\ref{xu}f), and removing a disk $D_{15}$ from $\bar V_{k+1,1}\cup D_{24}$ (from $\bar Z_{01}$ if $k=0$) results in $X'_{k1}$. Note that $X'_{11}$ allows to attach a disk $D_{15}$ to any of its two sides of order 5. Removing a disk $D_{24}$ from $X'_{11}\cup D_{15}$  results either in $V_{21}$ or in $\bar V_{21}$, depending on the side of $X'_{11}$ to which the disk is attached (see Figs.~\ref{xu}c and \ref{xu}e).
    The cases $\bar X'_{1k}\cup D_{24}$ and $\bar X'_{k1}\cup D_{24}$ are obtained by reflection symmetry preserving the opposite corners of order 0.

(e) The quadrilateral $Z_{1k}\cup D_{24}$ contains disk $D_{15}$ (shaded area in Fig.~\ref{zv}a). Removing it results in a quadrilateral $V_{k2}$ (in $\bar X'_{10}$ if $k=0$, see Fig.~\ref{zv}b). In the opposite direction, removing a disk $D_{24}$ from $V_{k2}\cup D_{15}$ (from $\bar X'_{10}$ if $k=0$) results in $Z_{1k}$. Similarly, removing a disk $D_{15}$ (shaded in Fig.~\ref{zv}c) from $Z_{k1}\cup D_{24}$ results in $\bar V_{k2}$ (in $\bar X'_{01}$ if $k=0$, see Fig.~\ref{zv}d),
    and removing a disk $D_{24}$ from $\bar V_{k2}\cup D_{15}$ (from $\bar X'_{01}$ if $k=0$) results in $Z_{1k}$. Note that $Z_{11}$ allows to
    attach a disk $D_{24}$ to any of its two sides of order 4 (see Figs.~\ref{zv}a and \ref{zv}c). Removing a disk $D_{15}$ from $Z_{11}\cup D_{24}$ results either in $V_{12}$ or in $\bar V_{12}$, depending on the side of $Z_{11}$ to which the disk is attached. Also, $V_{12}$ (resp., $\bar V_{12}$) has one side of order 4 and another one of order 5. Thus either $D_{24}$ or $D_{15}$ can be attached to a side of $V_{12}$ (resp., $\bar V_{12}$). Removing either $D_{15}$ or $D_{24}$ would result in either $X'_{11}$ or $Z_{11}$.
    The cases $\bar Z_{1k}\cup D_{24}$ and $\bar Z_{k1}\cup D_{24}$ are obtained by reflection symmetry preserving the opposite corners of order 0.

(f) The quadrilateral $Z'_{01}\cup D_{15}$ contains another disk $D_{15}$. Removing it results in a quadrilateral $\bar Z'_{01}$. In the opposite direction, removing a disk $D_{15}$ from $\bar Z'_{01}\cup D_{15}$ results in $Z'_{01}$.
    The cases $Z'_{10}\cup D_{15}$ and $\bar Z'_{10}\cup D_{15}$ are obtained by reflection symmetry exchanging the opposite corners of order 0.

(g) The quadrilateral $Z'_{1k}\cup D_{15}$ for $k>0$ contains another disk $D_{15}$ (shaded in Fig.~\ref{zv}e). Removing it results in a quadrilateral $\bar W_{k+1,2}$. In the opposite direction, removing a disk $D_{15}$ from $\bar W_{k+1,2}\cup D_{15}$ results in $Z'_{1k}$. Similarly, removing a disk $D_{15}$ from $Z'_{k1}\cup D_{15}$ for $k>0$ results in $W_{k+1,2}$, and removing a disk $D_{15}$ from $W_{k+1,2}\cup D_{15}$ results in $Z'_{k1}$. Note that $Z'_{11}$ allows to attach a disk $D_{15}$ to any of its two sides of order 5. Removing a disk $D_{15}$ from $Z'_{11}\cup D_{15}$ results either in $W_{22}$ or in $\bar W_{22}$, depending on the side of $Z'_{11}$ to which the disk is attached.
    The cases $\bar Z'_{1k}\cup D_{15}$ and $\bar Z'_{k1}\cup D_{15}$ are obtained by reflection symmetry preserving the opposite corners of order 0.

\begin{figure}
\centering
\includegraphics[width=4.8in]{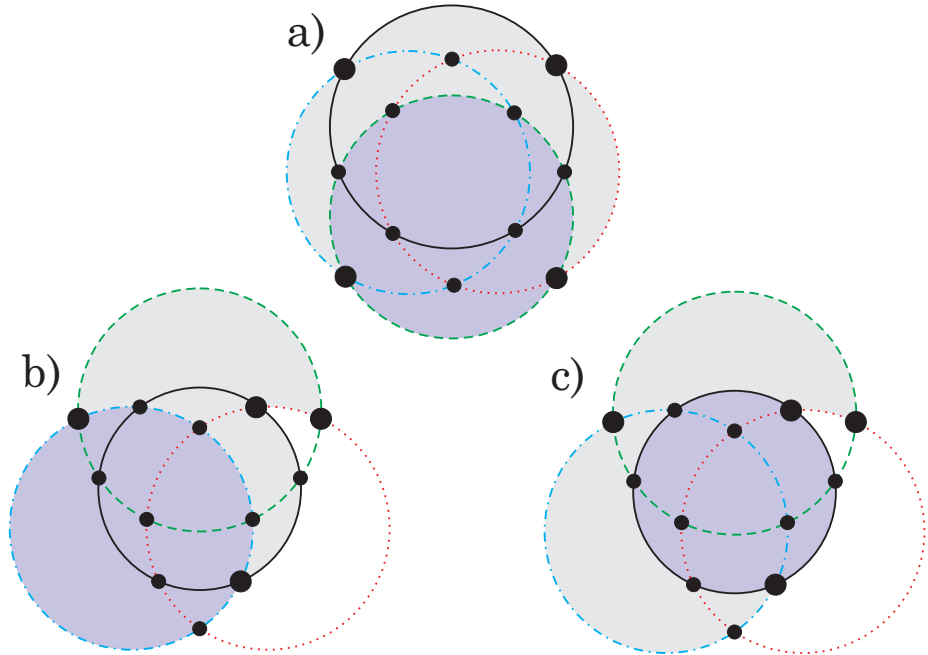}
\caption{Non-uniqueness of an irreducible quadrilateral, cases (a) and (b).}\label{xs}
\end{figure}

\begin{figure}
\centering
\includegraphics[width=4.8in]{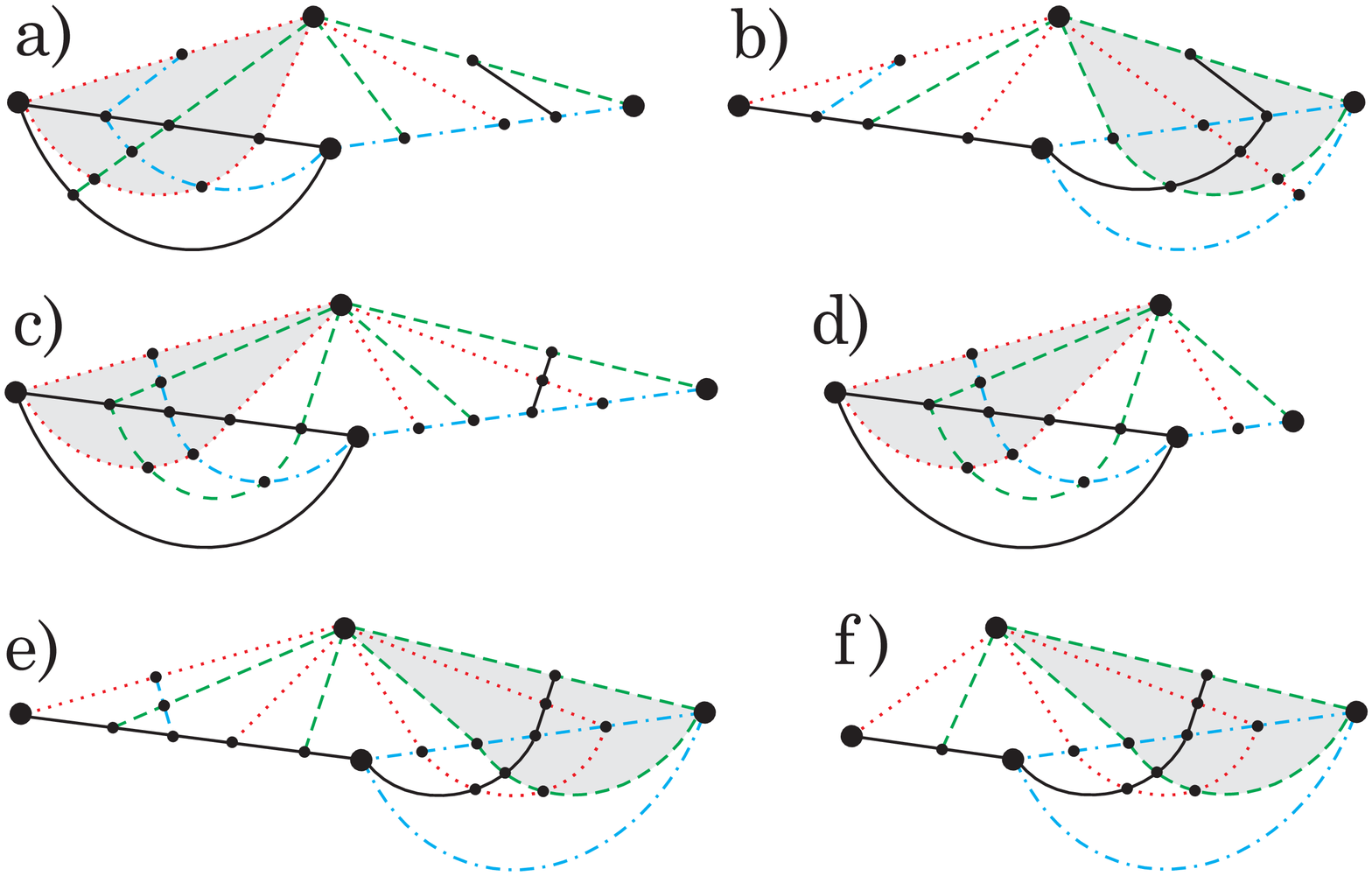}
\caption{Non-uniqueness of an irreducible quadrilateral, cases (c) and (d).}\label{xu}
\end{figure}

\begin{figure}
\centering
\includegraphics[width=4.8in]{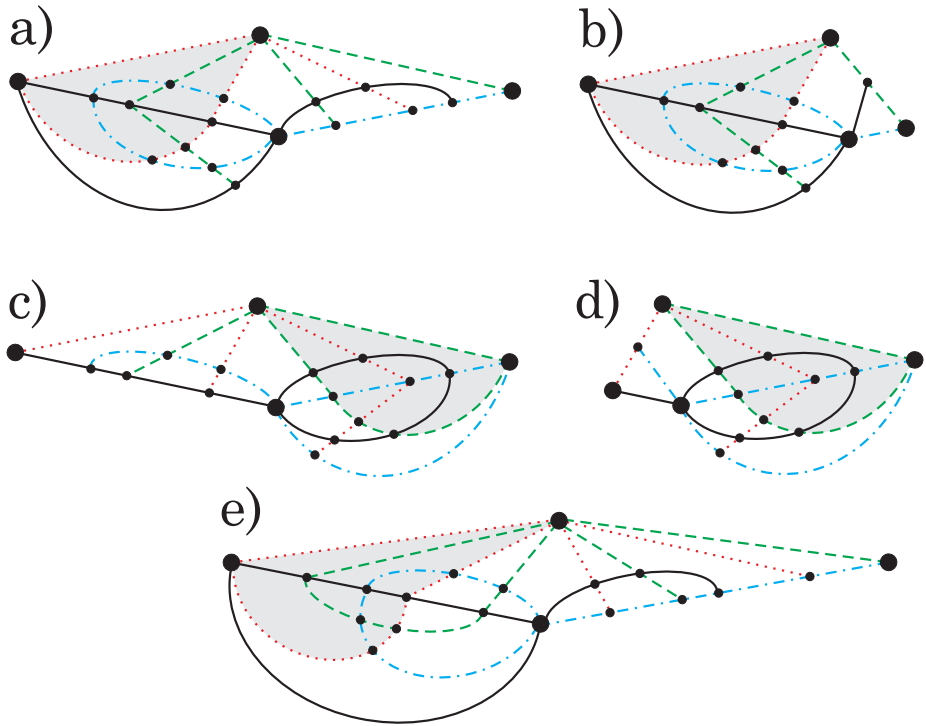}
\caption{Non-uniqueness of an irreducible quadrilateral, cases (e) - (g).}\label{zv}
\end{figure}

\begin{figure}
\centering
\includegraphics[width=3.in]{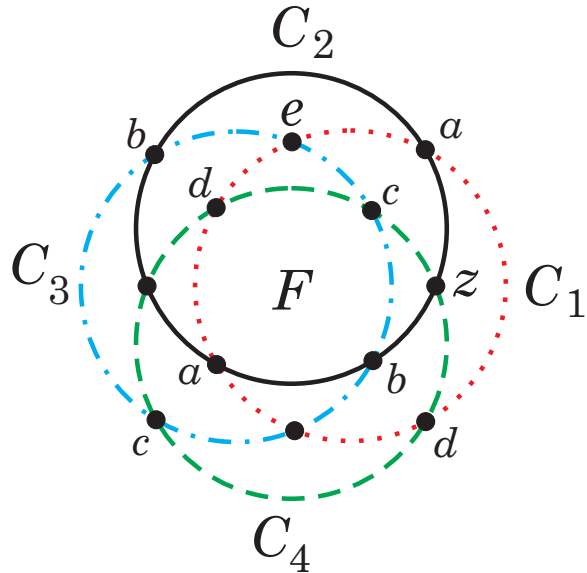}
\caption{The face $F$ of the four-circle configuration $\P$ with fixed angles $(a,b,c,d)$.}\label{fig:abcd}
\end{figure}

\begin{figure}
\centering
\includegraphics[width=4.8in]{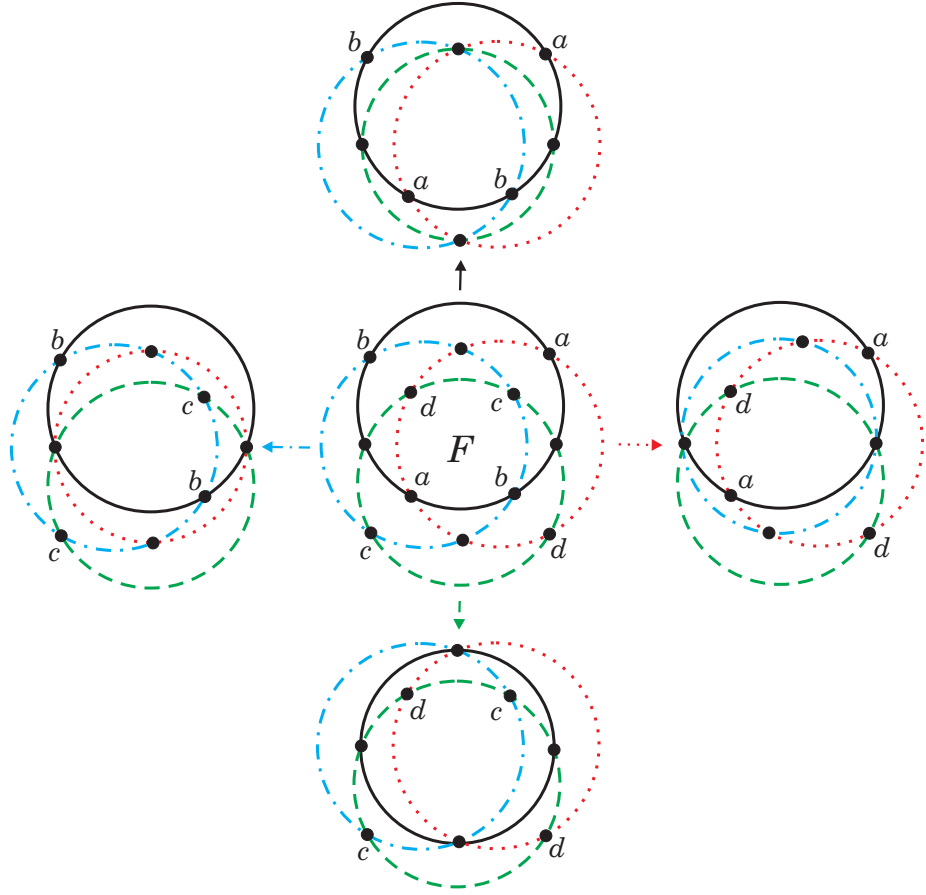}
\caption{Degenerations of a family of four-circle configurations to triple intersections.}\label{5t}
\end{figure}

\begin{figure}
\centering
\includegraphics[width=4.8in]{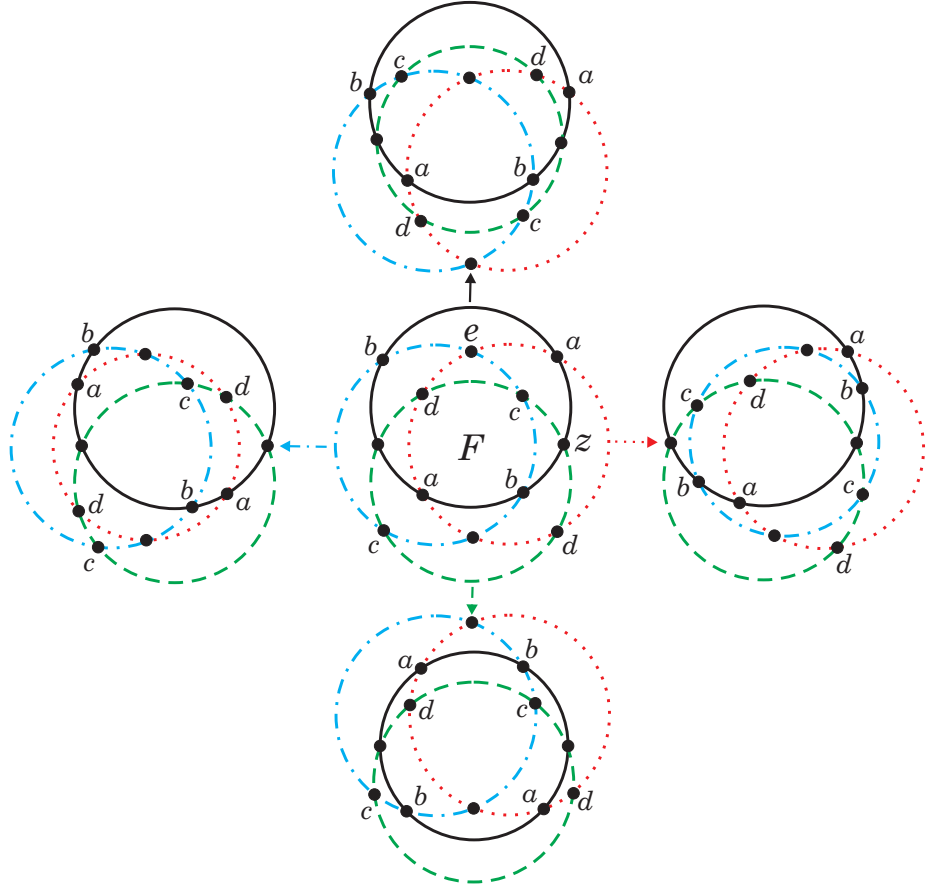}
\caption{Transformations of a family of four-circle configurations beyond triple intersections.}\label{5a}
\end{figure}

\begin{figure}
\centering
\includegraphics[width=4.8in]{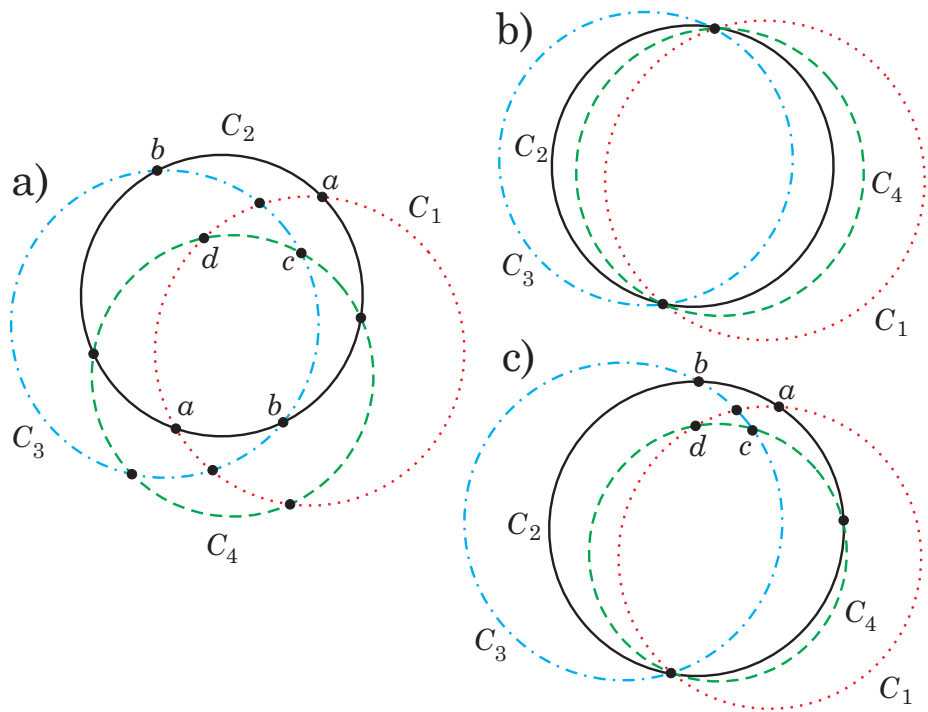}
\caption{Degeneration of a family of four-circle configurations (a) to a configuration with quadruple intersections (b)
and to a non-spherical four-circle configuration (c).}\label{quadruple}
\end{figure}

\begin{figure}
\centering
\includegraphics[width=4in]{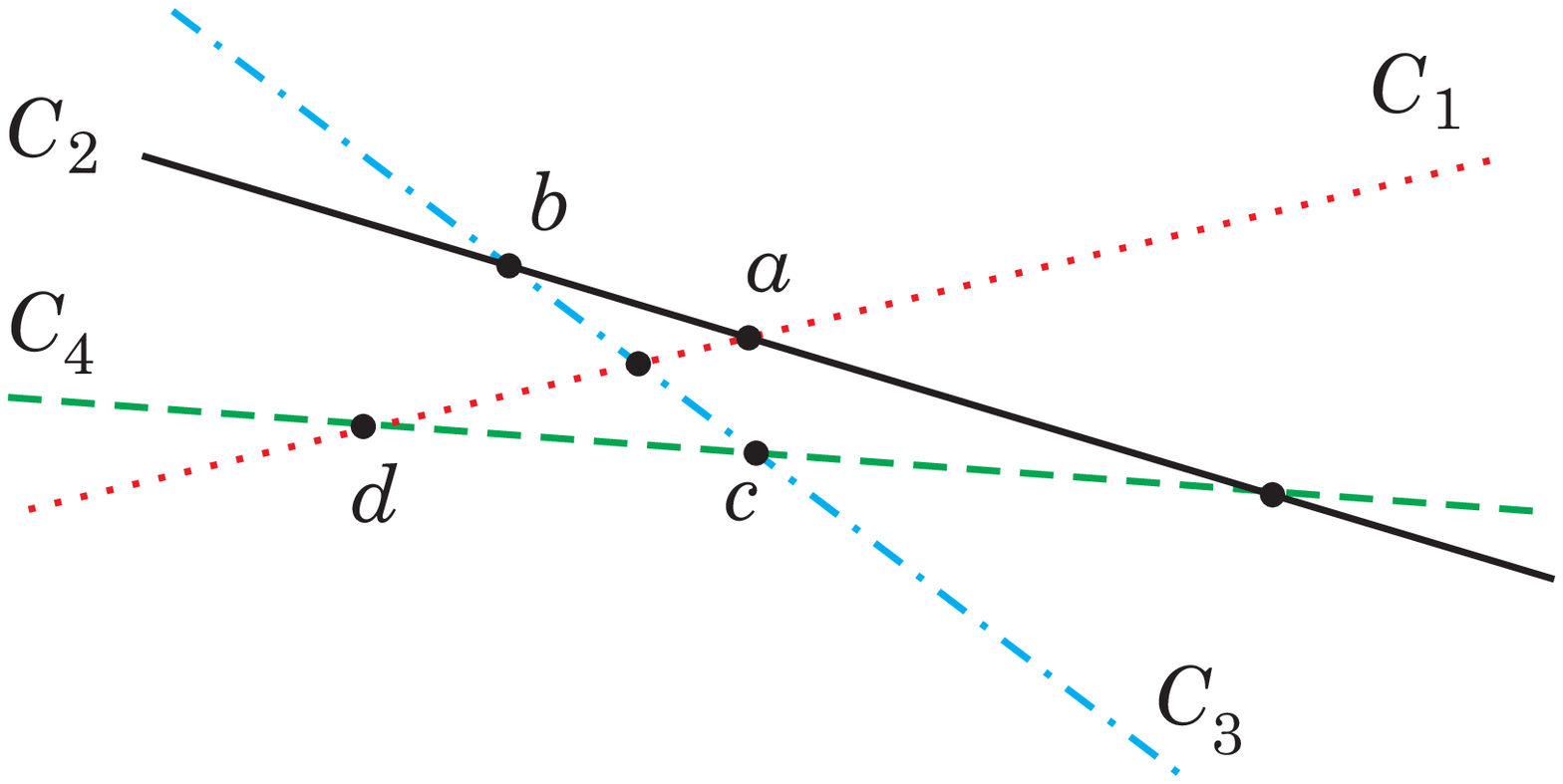}
\caption{A four-line configuration conformally equivalent to the configuration in Fig.~\ref{quadruple}c.}\label{euclid}
\end{figure}

\section{Chains of generic spherical quadrilaterals}\label{section:chains}
In this section we show that generic quadrilaterals with a given net $\Gamma$ and fixed four angles form an open segment $I_\Gamma$ in the set of all generic quadrilaterals, parameterized by the angle between any two circles of the four-circle configurations corresponding to opposite sides of the quadrilaterals.
In the limits at the ends of the interval $I_\Gamma$, a quadrilateral either conformally degenerates or converges to a non-generic spherical quadrilateral $Q'$ (with the sides mapped to a non-generic four-circle configuration), or converges to a non-spherical quadrilateral
after appropriate conformal transformations.
When a non-generic spherical quadrilateral $Q'$ is the limit of two families of generic quadrilaterals with the same angles and distinct nets $\Gamma_-$ and $\Gamma_+$, we say that the two families of generic quadrilaterals belong to a chain of quadrilaterals, and their nets $\Gamma_-$ and $\Gamma_+$ belong to a chain of nets (see Definition \ref{chain} below).

\begin{rmk}\label{chains:nets}\normalfont
The chains of spherical quadrilaterals with at least one integer angle considered in \cite{EGT1} - \cite{EGT3}
were completely determined by the integer parts of their angles.
For generic quadrilaterals, the chains depend also on the fractional parts of their angles.
\end{rmk}

Consider a generic quadrilateral $Q_0$ with the net $\Gamma_0$ and the sides mapped to four circles of a partition $\P_0$ of the sphere.
We claim that the set of all generic quadrilaterals $Q$ with the same net $\Gamma_0$ and the same angles as $Q_0$,
obtained from $Q_0$ by continuous deformation $\P$ of the partition $\P_0$,
constitute an open segment, with the limit at each end
corresponding either to a non-generic four-circle partition $\P'$ of the sphere with a triple intersection of circles, or to a three-circle partition.
At this limit, the quadrilateral $Q$ may conformally degenerate, its conformal modulus converging to either $0$ or $\infty$
(see \cite[Sections 14-15]{EGT2} and \cite[Section 7]{EGT3}).
If it does not conformally degenerate, it usually converges to a spherical quadrilateral $Q'$ over a non-generic partition $\P'$ (see Remark \ref{rmk:tangent} below for an exception).
In this case, the deformation of $Q$ can be extended through $Q'$ to another segment of generic quadrilaterals, with a net $\Gamma_1$ different from $\Gamma_0$.
A chain of quadrilaterals is a sequence of such extensions, starting and ending with degenerate quadrilaterals. The corresponding sequence of nets is called a chain of nets (see Definition \ref{chain} below).

Let $C_1,\dots,C_4$ be the circles of the partition $P$ to which the sides of $Q$ are mapped,
indexed according to the order of the sides of $Q$.
Since the angles of $Q$ are fixed, the angles $a,\,b,\,c,\,d$ between
the circles of $\P$ are also fixed. Here $a$ is the angle between $C_1$ and $C_2$,
$b$ is the angle between $C_2$ and $C_3$, $c$ is the angle between
$C_3$ and $C_4$, and $d$ is the angle between $C_4$ and $C_1$ (see Fig.~\ref{fig:abcd}).
The angle between two circles is defined up to its complement, but we can choose
the values $a,\,b,\,c,\,d$ uniquely by requiring that there is a quadrilateral face $F$ of the partition $\P$ having exactly these angles.
In fact, there are exactly two such faces of $\P$, having the same angles $a,\,b,\,c,\,d$ (with opposite cyclic orders).
The cyclic order of the circles at the sides of these faces is either the same or opposite to the cyclic order of the circles $C_1,\dots,C_4$ to which the sides of $Q$ are mapped.

\begin{prop}\label{abcd}
Let $\P$ be a partition of the sphere defined by a generic configuration of four great circles, and let $F$
be a quadrilateral face of $\P$ with the angles $a,\,b,\,c,\,d$. Then the following inequalities are satisfied:
\begin{equation}\label{sum}
0<a,b,c,d<1,\quad 0<a+b+c+d-2<2\min(a,b,c,d).
\end{equation}
The subset of the unit cube in $\R^4$ defined by these inequalities is an open convex pyramid $\Pi$ with
the vertex $(1,1,1,1)$ and the base an octahedron $P$ in the plane $a+b+c+d=2$ having vertices $(0,0,1,1)$, $(0,1,0,1)$, $(0,1,1,0)$, $(1,0,0,1)$, $(1,0,1,0)$, $(1,1,0,0)$.
\end{prop}

\begin{proof}
By definition, $(a,b,c,d)$ is a point of the unit cube in $\R^4$. Since the area $A=a+b+c+d-2$ of $F$ is positive, we have $a+b+c+d>2$. Since $F$ is an intersection of the four digons with the angles $a$, $b$, $c$, $d$ and the areas
$2a$, $2b$, $2c$, $2d$, respectively, we have $A<2\min(a,b,c,d)$. Note that $\Pi$ is defined by linear inequalities, thus it is a convex polytope with each facet on a plane defined by one of these inequalities.
The octahedron $P$ belongs to the plane $a+b+c+d=2$, it is the convex hull of the six vertices listed
in Proposition \ref{abcd}, and each of its eight triangular facets belongs to a side of the unit cube.
Each of the remaining facets of $\Pi$ is a 3-simplex, the convex hull $\Delta$ of the union of $(1,1,1,1)$ and one of the facets $\delta$ of $P$. If one of the variables, say $a$, equals $1$ on $\delta$ then $\Delta$ belongs to the plane $a=1$. If $a=0$ on $\delta$ then $a=\min(a,b,c,d)$ on $\Delta$, and $\Delta$ belongs to the plane
$a+b+c+d-2=2a$. This proves that on each facet of the pyramid $\Pi$ one of the inequalities in (\ref{sum})
becomes an equation. Since $a+b+c+d>2$ in $\Pi$, and all other inequalities in (\ref{sum}) are satisfied at the center $(1/2,1/2,1/2,1/2)$ of $P$, we conclude that the pyramid $\Pi$ is indeed the set in $\R^4$ defined by the inequalities (\ref{sum}). This completes the proof.
\end{proof}

There are four triangular faces of the four-circle configuration $\P$ in Fig.~\ref{fig:abcd} adjacent to the face $F$.
The areas of the bottom and top faces are $1-a-b+e$ and $1-c-d+e$, respectively.
The areas of the left and right faces are $1-a-d+z$ and $1-b-c+z$, respectively.
Here $e$ is the angle between $C_1$ and $C_3$ and $z$ is the angle between $C_2$ and $C_4$.
When the configuration $\P$ is deformed, the areas of the top and bottom are either both decreasing
or both increasing as $e$ decreases or increases, with the same rate as $e$. It follows from the cosine theorem that the top and bottom sides of $F$ are decreasing or increasing as $e$ decreases or increases. Similarly, the areas of the left and right triangular faces adjacent to $F$ are either both decreasing
or both increasing as $z$ decreases or increases, with the same rate as $z$, and the left and right sides of $F$ are decreasing or increasing as $z$ decreases or increases. Since the area of $F$ is constant, $e$
and $z$ cannot be both increasing or both decreasing when $\P$ is deformed. This implies the following statement.

\begin{prop}\label{degen} If $a+b>c+d$ (resp., $a+b<c+d$) then a generic four-circle configuration $\P$ can be deformed until the bottom (resp., top) triangular face adjacent to $F$ is contracted to a point, so that the circles $C_1,\,C_2,\,C_3$ (resp., $C_3,\,C_4,\,C_1$) have a triple intersection, but cannot be deformed so that the top (resp., bottom) face is contracted to a point. Similarly, if $a+d>b+c$ (resp., $a+d<b+c$) then $\P$ can be deformed until the left (resp., right) triangular face adjacent to $F$ is contracted to a point, so that the circles $C_4,\,C_1,\,C_2$ (resp., $C_2,\,C_3,\,C_4$) have a triple intersection, but cannot be deformed so that the right (resp., left) face is contracted to a point.
\end{prop}

Four possible degenerations of a family of four-circle configurations to non-generic configurations with triple intersections are shown in Fig.~\ref{5t}. The color (style) of each arrow between configurations in Fig.~\ref{5t} indicates a circle that is not part of a triple intersection in the corresponding non-generic configuration.
Note that, according to Proposition \ref{degen}, at most two of the four possible degenerations (one with a horizontal arrow and another one with a vertical arrow) can be realized for any given angles $(a,b,c,d)$.
If a generic four-circle configuration can be deformed to a non-generic configuration with a triple intersection, its deformation can be extended beyond the triple intersection to a combinatorially different generic four-circle configuration.
Four possibilities for such generic configurations are shown in Fig.~\ref{5a}, on top, bottom, left and right of the original (central) configuration $\P$, depending on the four possible triple intersections to which it may degenerate. The color (style) of each arrow between configurations in Fig.~\ref{5a} indicates a circle
that is not part of a triple intersection in the corresponding non-generic configuration, as in Fig.~\ref{5t}.

Note that condition on the angles of the face $F$ of $\P$ which determines whether it can be deformed through a triple intersection to a configuration $\P'$ (either top or bottom, left or right in Fig.~\ref{5a}) is exactly the inequality for the fixed angles of $\P'$ which guarantees that the quadrilateral face $F'$ of $\P'$ with four fixed angles has positive area.

\begin{rmk}\label{rmk:tangent}\normalfont
What happens if, e.g., $a+b=c+d$? Then configuration $\P$ (see Fig.~\ref{quadruple}a) can be deformed so that in the limit
both top and bottom triangular faces adjacent to $F$ are contracted to a point, and
all four circles intersect at two opposite points in the limit (see Fig.~\ref{quadruple}b). However, combining this family of four-circle configurations with an appropriate family of linear-fractional transformations of the sphere, one can obtain
a configuration with only one four-circle intersection point (see Fig.~\ref{quadruple}c). This limit configuration cannot be realized by great circles, and the family of four-circle configurations cannot be extended beyond the quadruple intersection to a four-circle family equivalent to a family of generic great circles with the net different from the net shown in Fig.~\ref{quadruple}a.
Note that four-circle configuration in Fig.~\ref{quadruple}c is conformally equivalent to a four-line configuration shown in
Fig.~\ref{euclid}. A (circular but not spherical) quadrilateral mapped to such a configuration is called a ``singular flat quadrilateral''
(see \cite{EG1, EG2, EGMP}). Here ``singular'' refers to the possibility of the developing map of such quadrilateral to have simple poles
inside the quadrilateral or on its sides (but not at the corners).

In Fig.~\ref{quadruple} we assume that $a<d$, thus the angle between the circles $C_2$ and $C_4$ in Fig.~\ref{quadruple}b and Fig.~\ref{quadruple}c is $d-a=b-c$.
If, in addition, $a=d$ (thus $b=c$) then the circles $C_2$ and $C_4$ converge in the limit to a single circle passing through the intersection points of $C_1$ and $C_3$. Combining this family of four-circle configurations with an appropriate family of linear-fractional transformations of the sphere,
one can obtain in the limit a configuration
with two tangent circles $C_2$ and $C_4$, their tangency point being at an intersection of the circles $C_1$ and $C_3$ (see \cite{EG1}, Section 4).
\end{rmk}

\subsection{Relations between adjacent four-circle configurations in a chain.}\label{sub:adjacent}
Consider a generic four-circle configuration $\P$ (see Fig.~\ref{fig:abcd}) with the angles $(a,b,c,d)$
of a quadrilateral face $F$.  When $\P$ is deformed beyond a triple intersection to a generic configuration $\P'$  (see Fig.~\ref{5a}) keeping the angles $(a,b,c,d)$ fixed,
the face $F$ of $\P$ is replaced with a quadrilateral face $F'$ of $\P'$ with different fixed angles: the angles $(1-a,1-b,c,d)$ in the top configuration in Fig.~\ref{5a}, the angles $(a,b,1-c,1-d)$ in the bottom configuration, the angles $(a,1-b,1-c,d)$ in the left configuration, and the angles $(1-a,b,c,1-d)$ in the right configuration.
The general rule for the transformation through a triple intersection is that two fixed angles of $F$ which are at the vertices not passing through a triple intersection are replaced with their complementary angles of $F'$.
If we apply the same rule to the four configurations in Fig.~\ref{5a} that can be obtained from $\P$,
we get more generic four-circle configurations.
Eight distinct generic four-circle configurations obtained this way are shown in Figs.~\ref{5b}a - \ref{5b}h.
Four more configurations shown in the bottom row of Fig.~\ref{5b} are equivalent to configurations in its top row:
Fig.~\ref{5b}i is equivalent to Fig.~\ref{5b}c,
Fig.~\ref{5b}j is equivalent to Fig.~\ref{5b}d,
Fig.~\ref{5b}k is equivalent to Fig.~\ref{5b}a,
Fig.~\ref{5b}l is equivalent to Fig.~\ref{5b}b.
The original configuration $\P$ and its four transformations
(see Fig.\ref{5a}) are shown in Figs.~\ref{5b}f, \ref{5b}b, \ref{5b}j, \ref{5b}e, and \ref{5b}g, respectively.

Any two configurations in Fig.~\ref{5b} adjacent either vertically or horizontally can be obtained from each other
by a deformation passing through a triple intersection when certain inequalities on $(a,b,c,d)$ are satisfied.
In fact, Fig.~\ref{5b} should be considered as part of a double periodic square lattice with periods $(4,0)$ and $(2,2)$. For example, configuration in Fig.~\ref{5b}a (and Fig.~\ref{5b}k) has a quadrilateral face with the fixed angles $(a,1-b,c,1-d)$.
It exists when $a+c>b+d$ and $a-b+c-d<2\min(a,1-b,c,1-d)$. It can be obtained, when $a+c>b+d$, either from the configuration in Fig.\ref{5b}b, replacing the angles $1-a$ and $d$ by their complementary angles, or from
the configuration in Fig.\ref{5b}e, replacing the angles $1-c$ and $d$ by their complementary angles.

Note that configurations in Fig.~\ref{5b} twice removed either horizontally or vertically
have all four fixed angles of a quadrilateral face complementary, reversing the second inequality in (\ref{sum}).
Thus at most one of them may exist for any given values of the fixed angles.

\begin{rmk}\label{rmk:even}\normalfont
All configurations in Fig.~\ref{5b} are ``even'': each of them has a quadrilateral face with four fixed angles, an
even number of them complementary to the angles $(a,b,c,d)$. Replacing one of the angles by its complementary angle
we get eight distinct ``odd'' configurations with the given angles $(a,b,c,d)$. Any transformation through a triple intersection
preserving the angles $(a,b,c,d)$ is possible either between two even configurations or between two odd ones.
\end{rmk}

\begin{figure}
\centering
\includegraphics[width=4.8in]{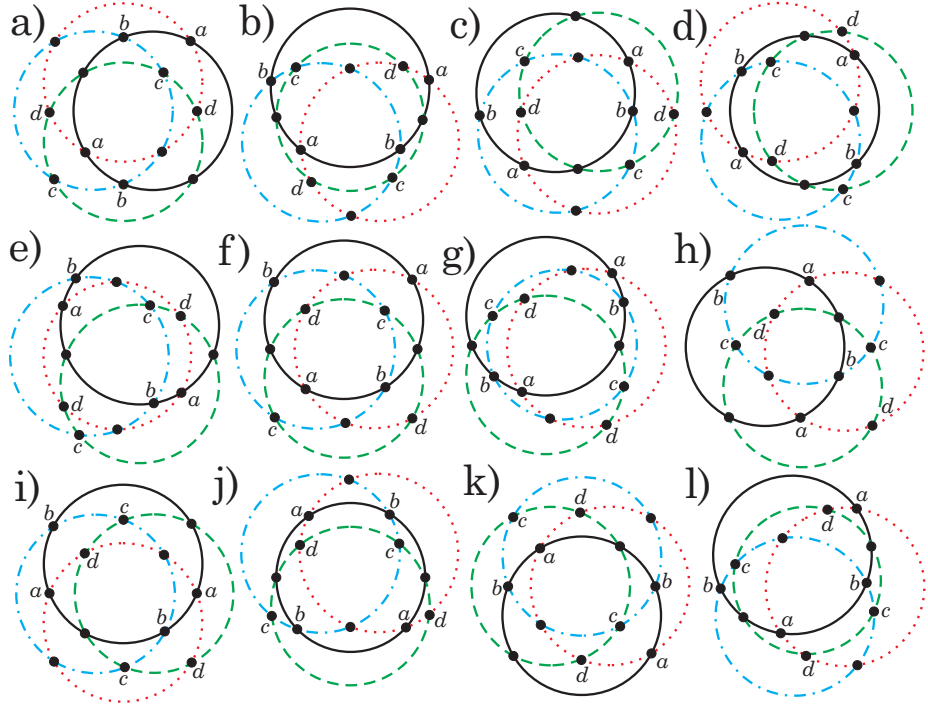}
\caption{Generic ``even'' four-circle configurations related by deformation with fixed angles $(a,b,c,d)$ through triple intersections.}\label{5b}
\end{figure}

Let us return to configuration $\P$ in Fig.~\ref{5b}f and assume that
\begin{equation}\label{ladder}
a+b<c+d,\,\;a+d<b+c,\;a+c<b+d.
\end{equation}
Then configuration $\P$ can be transformed either upward, to configuration in Fig.~\ref{5b}b, or to the right, to configuration in Fig.~\ref{5b}g.
Configuration in Fig.~\ref{5b}b has a quadrilateral face with the angles $(1-a,1-b,c,d)$.
It can be transformed either downward, back to configuration in Fig.~\ref{5b}f,
or to the left, to configuration in Fig.~\ref{5b}a, since $a+c<b+d$ implies $1-a+d>1-b+c$.
Configuration in Fig.~\ref{5b}g has a quadrilateral face with the angles $(1-a,b,c,1-d)$.
It can be transformed either to the left, back to configuration in Fig.~\ref{5b}f,
or downward, to configuration in Fig.~\ref{5b}k, since $a+c<b+d$ implies $1-a+b>c+1-d$.
Note that configuration in Fig.~\ref{5b}k is equivalent to configuration in Fig.~\ref{5b}a.
Thus any sequence of transformations of a configuration satisfying inequalities (\ref{ladder}) is possible within the ``ladder'' pattern in Fig.~\ref{5b}.

Consider now configuration $\P$ in Fig.~\ref{5b}f, with the angles satisfying
\begin{equation}\label{box}
a+b<c+d,\;a+d<b+c,\;a+c>b+d.
\end{equation}
Then configuration $\P$ can be transformed, as before, either upward, to configuration in Fig.~\ref{5b}b, or to the right, to configuration in Fig.~\ref{5b}g, since this depends only on the first two inequalities in (\ref{ladder}) and (\ref{box}).
The configuration in Fig.~\ref{5b}b can be transformed either downward, back to configuration in Fig.~\ref{5b}f,
or to the right, to configuration in Fig.~\ref{5b}c, since $a+c>b+d$ implies $1-a+d<1-b+c$.
Similarly, configuration in Fig.~\ref{5b}g can be transformed either to the left, back to configuration in Fig.~\ref{5b}f,
or upward, to configuration in Fig.~\ref{5b}c, since $a+c>b+d$ implies $1-a+b<c+1-d$.
Thus any sequence of transformations of a configuration
satisfying inequalities (\ref{box}) is possible within the ``box'' pattern in Fig.~\ref{5b}.
This can be summarized as follows.

\begin{prop}\label{ladder-box}
Let $(a,b,c,d)$ be the fixed angles of a quadrilateral face $F$ of a generic four-circle configuration $\P$,
such that the ``opposite'' pairs $(a,c)$ and $(b,d)$ are at the opposite vertices of $F$.
According to Proposition \ref{degen},exactly one of the angles remains unchanged when $\P$ is deformed through both permitted triple intersections.
The ``ladder'' pattern, as in \emph{(\ref{ladder})}, appears when the sum of two angles in the opposite pair containing the ``twice unchanged'' angle is smaller than the sum of two angles in the other opposite pair. Otherwise, the ``box'' pattern, as in \emph{(\ref{box})}, appears.
\end{prop}

\begin{rmk}\label{rmk:pyramid}\normalfont
The conditions on the angles $(a,b,c,d)$ related to the possibility of transforming generic configurations
 through triple intersections can be described as subsets of
the unit cube in $\R^4$.
According to Proposition \ref{abcd}, the point $(a,b,c,d)$ for the partition $\P$
belongs to the open pyramid $\Pi$ with the vertex $(1,1,1,1)$ and the base an octahedron with vertices
$(0,0,1,1)$, $(0,1,0,1)$, $(0,1,1,0)$, $(1,0,0,1)$, $(1,0,1,0)$, $(1,1,0,0)$.
Adding the first inequality in (\ref{ladder}) or (\ref{box}) cuts this pyramid in half: the inequality
$a+b<c+d$ removes the vertex $(1,1,0,0)$ from the octahedron, leaving the convex hull of the remaining vertices of $\Pi$. Note that the inequalities $a+b+c+d-2<2c$ and $a+b+c+d-d<2d$ in (\ref{sum}) are now automatically satisfied, so the last inequality in (\ref{sum}) can be replaced with $a+b+c+d-2<2\min(a,b)$.
Adding the second inequality in (\ref{ladder}) or (\ref{box}) leaves a quarter of $\Pi$, which is a 4-simplex: the inequality $a+d<b+c$ removes the vertex $(1,0,0,1)$ from the octahedron, reducing $\Pi$ to the convex hull of its remaining 5 vertices. The last inequality in (\ref{sum}) can be now replaced with $a+b+c+d-2<2a$, i.e., $b+c+d<2+a$.
Adding the third inequality in (\ref{ladder}) or (\ref{box}) cuts that simplex in half, removing one more vertex of the octahedron (either $(1,0,1,0)$ for the ladder pattern or $(0,1,0,1)$ for the box pattern) and leaving the convex hull of the 4 remaining vertices of $\Pi$ and the center $(1/2,1/2,1/2,1/2)$ of the octahedron.
Note that in the box case the last inequality in (\ref{sum}) is automatically satisfied.
\end{rmk}

The fractional parts $(\alpha,\beta,\gamma,\delta)$ of the angles of a spherical quadrilateral $Q$ are either some of the fixed angles $(a,b,c,d)$ of a quadrilateral face of its four-circle configuration or some of their complementary angles $(1-a,\,1-b,\,1-c,\,1-d)$.
The choice between each angle and its complement is determined by the net of $Q$ (see Propositions \ref{even} and \ref{odd} below).
For the basic primitive quadrilaterals (see Fig.~\ref{basic}) the fractional parts $(\alpha,\beta,\gamma,\delta)$ of the angles are
\begin{multline}\label{basic-angles}
(a,b,c,d)\;{\rm for}\; P_0,\;(1-a,b,1-c,1-d)\;{\rm for}\; X'_{00},\\
(1-a,1-b,1-c,d)\;{\rm for}\; \bar X'_{00},\;(a,1-b,c,1-d)\;{\rm for}\; Z'_{00}.
\end{multline}
Note that the quadrilateral faces with the fixed angles $(a,b,c,d)$ of the four-circle configurations corresponding to the quadrilaterals $X'_{00}$ and $\bar X'_{00}$ are outside their nets.

Attaching a triangle $T_k$ with an integer angle to the side of a quadrilateral does not change the fractional parts of its angles when $k$ is even,
and replaces the fractional part of one of its angles with its complement when $k$ is odd.
According to Theorem \ref{primitive}, every primitive quadrilateral can be obtained by attaching one or two triangles, each with an integer angle,
to the sides of one of the basic quadrilaterals. Corollary \ref{irreducible} states that each irreducible quadrilateral can be obtained from a primitive one by inserting a quadrilateral $P_\mu$ with two of its angles of order $2\mu$ and the other two angles of order 0. Theorem \ref{quad} states that any generic quadrilateral can be obtained from an irreducible one by attaching some digons, with two equal integer angles each, to its sides.
Note that the sum $\Sigma$ of the integer parts of the angles of a quadrilateral
is increased by $k$ when a triangle $T_k$ is attached. Note also that $\Sigma=0$ for $P_0$, $\Sigma=1$ for $X'_{00}$ and $\bar X'_{00}$, $\Sigma=2$ for $Z'_{00}$.
Thus relations (\ref{basic-angles}) imply the following relation between the angles of a generic spherical quadrilateral $Q$ and the fixed angles of a
quadrilateral face of its underlying partition $\P$ of the sphere.

\begin{prop}\label{parity}
Let $Q$ be a generic spherical quadrilateral with the sum $\Sigma$ of the integer parts of its angles, and the corresponding partition $\P$
with a quadrilateral face having angles $(a,b,c,d)$, each of them being either a fractional part of the angle of $Q$ or its complementary angle.
Then the number of the complements among $(a,b,c,d)$ has the same parity as $\Sigma$.
\end{prop}

Relations between the fractional parts $(\alpha,\beta,\gamma,\delta)$ of the angles of a primitive quadrilateral $Q$ from the list in Section \ref{section:primitive} and the fixed angles $(a,b,c,d)$ of a quadrilateral face $F$ of its underlying partition $\P$ are presented in the following two Propositions, separate for even and odd values of the sum $\Sigma$ of the integer parts of the angles of $Q$.

\begin{prop}\label{even}
Let $Q$ be a primitive spherical quadrilateral with the even sum of the integer parts of its angles.
Let $\P$ be the corresponding partition of the sphere having a quadrilateral face with the fixed angles $(a,b,c,d)$ which are either fractional parts $(\alpha,\beta,\gamma,\delta)$ of the angles of $Q$ or their complements. The following list describes the angles $(a,b,c,d)$ in terms of the angles $(\alpha,\beta,\gamma,\delta)$, depending on the net of $Q$:

$(\alpha,\beta,\gamma,\delta)$ for $P_0$, $X_{kl}$ with $k$ and $l$ even, $\bar X_{kl}$ with $k$ and $l$ even, $R_{kl}$ with $k$ and $l$ even, $\bar R_{kl}$ with $k$ and $l$ even, $U_{kl}$ with $k$ and $l$ even, $\bar U_{kl}$ with $k$ and $l$ even;

$(1-\alpha,1-\beta,\gamma,\delta)$ for $R_{kl}$ with $k$ and $l$ odd, $\bar X'_{kl}$ with $k$ even and $l$ odd,  $Z_{kl}$ with $k$ even and $l$ odd, $V_{kl}$ with $k$ and $l$ even, $\bar V_{kl}$ with $k$ and $l$ odd;

$(1-\alpha,\beta,1-\gamma,\delta)$ for $X_{kl}$ with $k$ and $l$ odd, $\bar X_{kl}$ with $k$ and $l$ odd, $U_{kl}$ with $k$ and $l$ odd, $\bar U_{kl}$ with $k$ and $l$ odd;

$(1-\alpha,\beta,\gamma,1-\delta)$ for $\bar R_{kl}$ with $k$ and $L$ odd, $X'_{kl}$ with $k$ even and $l$ odd, $\bar Z_{kl}$ with $k$ even and $l$ odd, $V'_{kl}$ for $k$ and $l$ odd, $\bar V'_{kl}$ for $k$ and $l$ even;

$(\alpha,1-\beta,1-\gamma,\delta)$ for $\bar X'_{kl}$ with $k$ odd and $l$ even,  $Z_{kl}$ with $k$ odd and $l$ even, $\bar S_{kl}$ with $k$ and $l$ even,
$V_{kl}$ with $k$ and $l$ odd, $\bar V_{kl}$ with $k$ and $l$ even;

$(\alpha,1-\beta,\gamma,1-\delta)$ for $Z'_{kl}$ with $k$ and $l$ even, $\bar Z'_{kl}$ with $k$ and $l$ even, $W_{kl}$ with $k$ and $l$ odd, $\bar W_{kl}$ with $k$ and $l$ odd;

$(\alpha,\beta,1-\gamma,1-\delta)$ for $S_{kl}$ with $k$ and $l$ even, $X'_{kl}$ with $k$ odd and $l$ even, $\bar Z_{kl}$ with $k$ odd and $l$ even, $V'_{kl}$ for $k$ and $l$ even, $\bar V'_{kl}$ for $k$ and $l$ odd;

$(1-\alpha,1-\beta,1-\gamma,1-\delta)$ for $S_{kl}$ with $k$ and $l$ odd, $\bar S_{kl}$ with $k$ and $l$ odd, $Z'_{kl}$ with $k$ and $l$ odd, $\bar Z'_{kl}$ with $k$ and $l$ odd, $W_{kl}$ with $k$ and $l$ even, $\bar W_{kl}$ with $k$ and $l$ even.
\end{prop}

\begin{prop}\label{odd}
Let $Q$ be a primitive spherical quadrilateral with the odd sum of the integer parts of its angles.
Let $\P$ be the corresponding partition of the sphere having a quadrilateral face with the fixed angles $(a,b,c,d)$ which are either fractional parts $(\alpha,\beta,\gamma,\delta)$ of the angles of $Q$ or their complements. The following list describes the angles $(a,b,c,d)$ in terms of the angles $(\alpha,\beta,\gamma,\delta)$, depending on the net of $Q$.

$(1-\alpha,\beta,\gamma,\delta)$ for $X_{kl}$ with $k$ odd and $l$ even, $\bar X_{kl}$ with $k$ odd and $l$ even, $R_{kl}$ with $k$ odd and $l$ even, $\bar R_{kl}$ with $k$ odd and $l$ even, $U_{kl}$ with $k$ odd and $l$ even, $\bar U_{kl}$ with $k$ even and $l$ odd;

$(\alpha,1-\beta,\gamma,\delta)$ for $R_{kl}$ with $k$ even and $l$ odd, $\bar X'_{kl}$ with $k$ and $l$ odd, $Z_{kl}$ with $k$ and $l$ odd, $V_{kl}$ with $k$ odd and $l$ even, $\bar V_{kl}$ with $k$ odd and $l$ even;

$(\alpha,\beta,1-\gamma,\delta)$ for $X_{kl}$ with $k$ even and $l$ odd, $\bar X_{kl}$ with $k$ even and $l$ odd, $U_{kl}$ with $k$ even and $l$ odd, $\bar U_{kl}$ with $k$ odd and $l$ even;

$(\alpha,\beta,\gamma,1-\delta)$ for $\bar R_{kl}$ with $k$ even and $l$ odd, $X'_{kl}$ with $k$ and $l$ odd, $\bar Z_{kl}$ with $k$ and $l$ odd, $V'_{kl}$ for $k$ even and $l$ odd,  $\bar V'_{kl}$ with $k$ even and $l$ odd;

$(\alpha,1-\beta,1-\gamma,1-\delta)$ for $S_{kl}$ with $k$ even and $l$ odd, $\bar S_{kl}$ with $k$ even and $l$ odd, $Z'_{kl}$ with $k$ even and $l$ odd, $\bar Z'_{kl}$ with $k$ odd and $l$ even, $W_{kl}$ with $k$ odd and $l$ even, $\bar W_{kl}$ with $k$ even and $l$ odd;

$(1-\alpha,\beta,1-\gamma,1-\delta)$ for $X'_{kl}$ with $k$ and $l$ even, $\bar Z_{kl}$ with $k$ and $l$ even, $V_{kl}$ with $k$ and $l$ even, $\bar V_{kl}$ with $k$ and $l$ even, $S_{kl}$ with $k$ odd and $l$ even, $V'_{kl}$ for $k$ odd and $l$ even, $\bar V'_{kl}$ with $k$ odd and $l$ even;

$(1-\alpha,1-\beta,\gamma,1-\delta)$ for $Z'_{kl}$ with $k$ odd and $l$ even, $\bar Z'_{kl}$ with $k$ even and $l$ odd, $W_{kl}$ with $k$ even and $l$ odd, $\bar W_{kl}$ with $k$ odd and $l$ even;

$(1-\alpha,1-\beta,1-\gamma,\delta)$ for $\bar X'_{kl}$ with $k$ and $l$ even, $Z_{kl}$ with $k$ and $l$ even, $\bar S_{kl}$ with $k$ odd and $l$ even,
$V_{kl}$ with $k$ even and $l$ odd, $\bar V_{kl}$ with $k$ even and $l$ odd.
\end{prop}

\begin{rmk} \normalfont
Note that adding pseudo-diagonals or attaching disks to the sides of a quadrilateral
does not change relations between the angles $(\alpha,\beta,\gamma,\delta)$ and $(a,b,c,d)$ in Propositions \ref{even} and \ref{odd}, so these relations hold for any generic spherical quadrilateral.
\end{rmk}

\begin{example}\label{abcd-x}\normalfont
If the angles of $X_{kl}$ in Fig.~\ref{4circles-x} are $(\alpha,\,\beta,\,\gamma,\,k+l+\delta)$,
 then the angles $(a,b,c,d)$ of the shaded quadrilateral face of its net are
 $(\alpha,\,\beta,\,1-\gamma,\,\delta)$ for $X_{01}$,
$(1-\alpha,\,\beta,\,1-\gamma,\,\delta)$ for $X_{11}$,
$(1-\alpha,\,\beta,\,\gamma,\,\delta)$ for $X_{12}$.
\end{example}

\begin{df}\label{chain}
{\rm A \emph{chain} of quadrilaterals of length $n\ge 0$ is a maximal sequence of segments $I_j$ in the space of generic spherical quadrilaterals with given angles, corresponding to distinct nets $\Gamma_0,\dots,\Gamma_n$, so that the segments $I_j$ and $I_{j+1}$, for $j=0,\dots,n-1$, have a common non-generic non-degenerate quadrilateral $Q_j$ in the limit at their ends, its sides being mapped to a four-circle configuration with a triple intersection.
A chain of length $0$ is a segment $I_0$ of generic spherical quadrilaterals with the net $\Gamma_0$ such that the limits of quadrilaterals at both ends of $I_0$ degenerate.}
\end{df}

\begin{example}\label{chain-x}
{\rm In Fig.~\ref{fig:chainx}a the shaded quadrilateral $Q_0$ with the net $X_{01}$ and angles $(\alpha,\beta,\gamma,1+\delta)$ is shown together with the configuration $\P_0$ to which its sides are mapped.
Note that $\P_0$ has a quadrilateral face with the angles $(a,b,c,d)=(\alpha,\beta,1-\gamma,\delta)$.
According to (\ref{sum}), the angles of $Q_0$ satisfy the inequalities
\begin{equation}\label{eqn:x01}
0<\alpha+\beta-\gamma+\delta-1<2\min(\alpha,\beta,1-\gamma,\delta).
\end{equation}
Proposition \ref{degen} describes the conditions on $(a,b,c,d)$ and $(\alpha,\beta,\gamma,\delta)$
that allow one to transform $\P_0$ to one of the four four-circle configurations with triple intersections.
It is easy to check that only one of these four transformations results in a deformation of $Q_0$
that is non-degenerate in the limit:
if $\alpha+\delta<1-\gamma+\beta$ then $\P_0$ can be deformed to a configuration with a triple intersection shown in Fig.~\ref{fig:chainx}b, and $Q_0$ to the quadrilateral shaded in Fig.~\ref{fig:chainx}b.
Passing through the triple intersection, we get a generic configuration $\P_1$ shown in Fig.~\ref{fig:chainx}c,
and a quadrilateral $Q_1$ (shaded in Fig.~\ref{fig:chainx}c) with the same angles as $Q_0$. The net of $Q_1$ is $X'_{00}$. Configuration $\P_1$ has a quadrilateral face with the angles $(1-\alpha,\beta,1-\gamma,1-\delta)$. If $1+\alpha<\beta+\gamma+\delta$ then $\P_1$ can be transformed to a configuration with a triple intersection shown in Fig.~\ref{fig:chainx}d. The quadrilateral $Q_1$ shaded in Fig.~\ref{fig:chainx}d is non-degenerate.
Passing through the triple intersection we get a generic configuration $\P_2$ shown in Fig.~\ref{fig:chainx}e,
and a quadrilateral $Q_2$ (shaded in Fig.~\ref{fig:chainx}e) with the same angles as $Q_0$. The net of $Q_2$ is $X_{10}$, and any deformation of $\P_2$ to a configuration with a triple intersection, other than that shown in Fig.~\ref{fig:chainx}, results in a degenerate quadrilateral.
Thus the chain of nets for the angles satisfying (\ref{eqn:x01}) and
\begin{equation}\label{eqn:chainx}
\alpha+\gamma+\delta<1+\beta,\quad\beta+\gamma+\delta>1+\alpha
\end{equation}
has length 2.
Remark \ref{rmk:pyramid} implies that the last inequality in (\ref{eqn:x01}) can be replaced by $\beta-\gamma+\delta<1+\alpha$ when the inequalities in (\ref{eqn:chainx}) are satisfied.
If the first inequality in (\ref{eqn:chainx}) is violated then the quadrilateral $Q_1$ with the net $X'_{00}$ does not exist, and the chain breaks down into two chains of length 0 consisting of the quadrilaterals $X_{01}$ and $X_{10}$. If the second
inequality in (\ref{eqn:chainx}) is violated
then our chain of length 2 reduces to a chain $\{X_{01}, X'_{00}\}$of length $1$.
If both the first and the second inequalities are violated, the chain of length 2 reduces to a chain of length 0 consisting of a single net $X_{01}$.}
\end{example}

\begin{example}\label{chain-x-quadruple}
{\rm If $\alpha+\gamma+\delta=1+\beta$ in (\ref{eqn:chainx}) then the configuration $\P_0$ of the quadrilateral $Q_0$ with the net $X_{01}$ in Fig.~\ref{fig:chainx}a (shown also in Fig.~\ref{fig:chainxquadruple}a) can be deformed
(see Remark \ref{rmk:tangent}) so that in the limit the four circles intersect at one point, and the limit quadrilateral (see Fig.~\ref{fig:chainxquadruple}c), with the same angles as $Q_0$, is not degenerate. Note that the limit quadrilateral is not spherical, and the family of quadrilaterals obtained by deformation of $\P_0$ cannot be extended beyond the limit configuration with quadruple intersection as a spherical configuration.}
\end{example}
\begin{figure}
\centering
\includegraphics[width=4.8in]{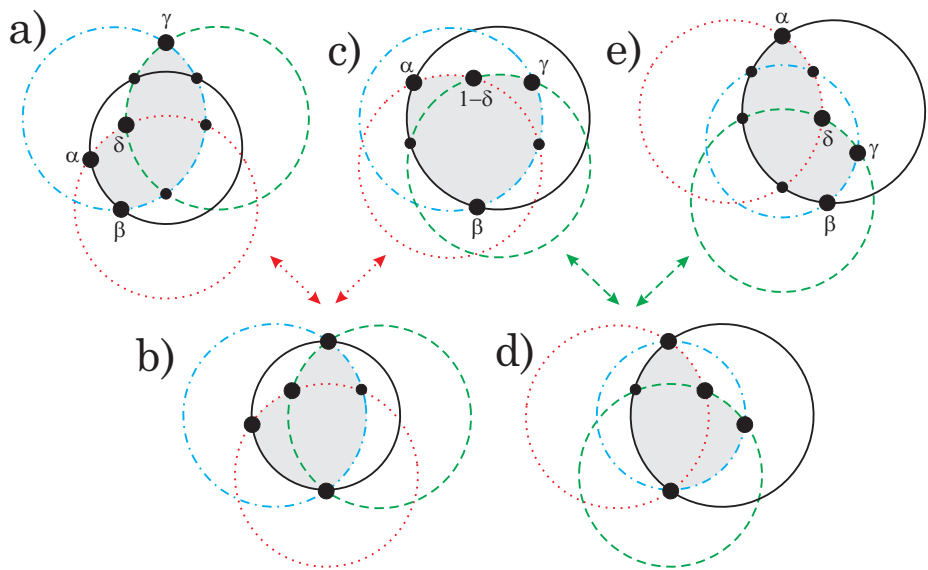}
\caption{The chain of quadrilaterals $X_{01},\,X'_{00},\,X_{10}$.}\label{fig:chainx}
\end{figure}

\begin{figure}
\centering
\includegraphics[width=4.8in]{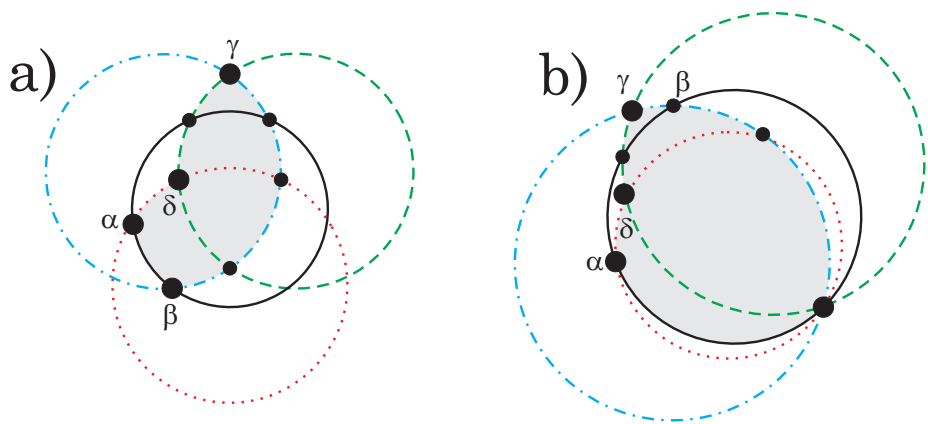}
\caption{Limit of the quadrilateral $X_{01}$ when $\alpha+\delta=\beta+1-\gamma$.}\label{fig:chainxquadruple}
\end{figure}

\begin{figure}
\centering
\includegraphics[width=4.8in]{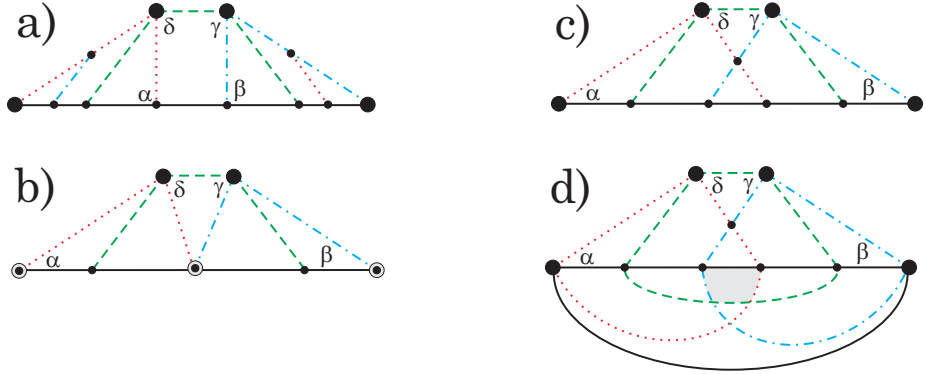}
\caption{The chain of quadrilaterals $R_{11}$ and $S_{11}$.}\label{fig:chainrs}
\end{figure}

\begin{figure}
\centering
\includegraphics[width=4.8in]{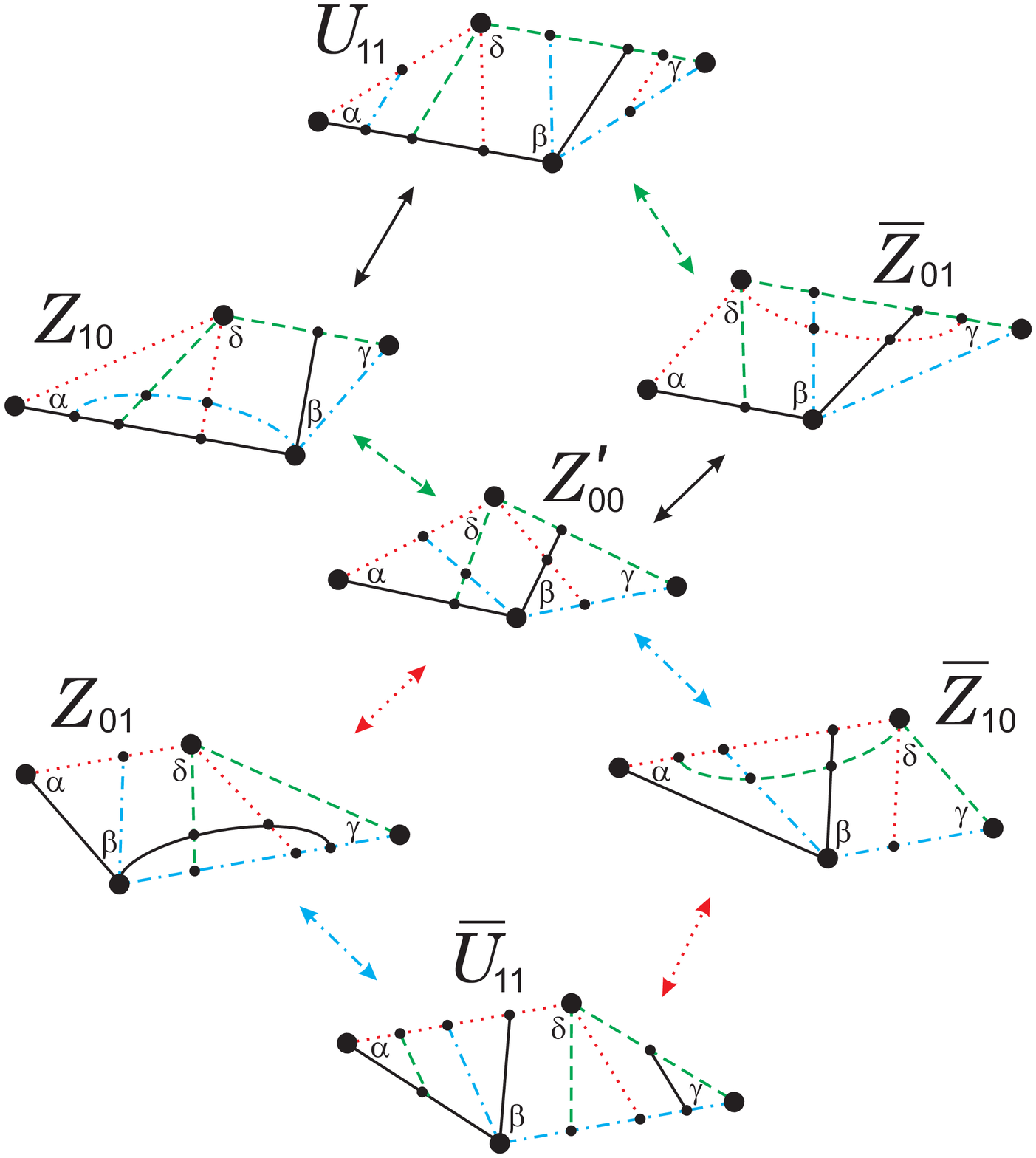}
\caption{The chains of quadrilaterals associated with $Z'_{00}$.}\label{fig:chainuvzprime}
\end{figure}

\begin{figure}
\centering
\includegraphics[width=4.8in]{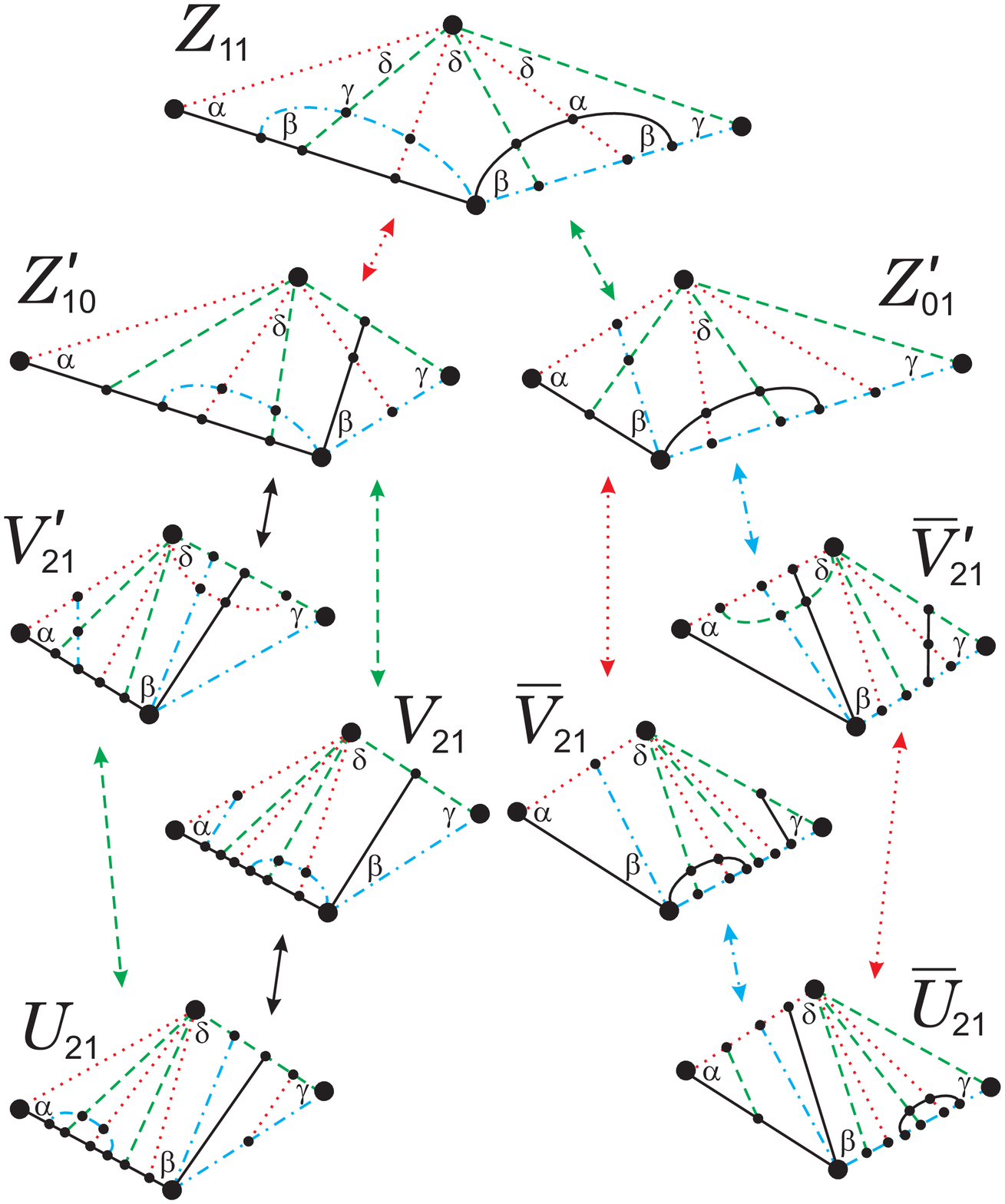}
\caption{The chains of quadrilaterals associated with $Z_{11}$.}\label{fig:chainz}
\end{figure}

\begin{figure}
\centering
\includegraphics[width=4.8in]{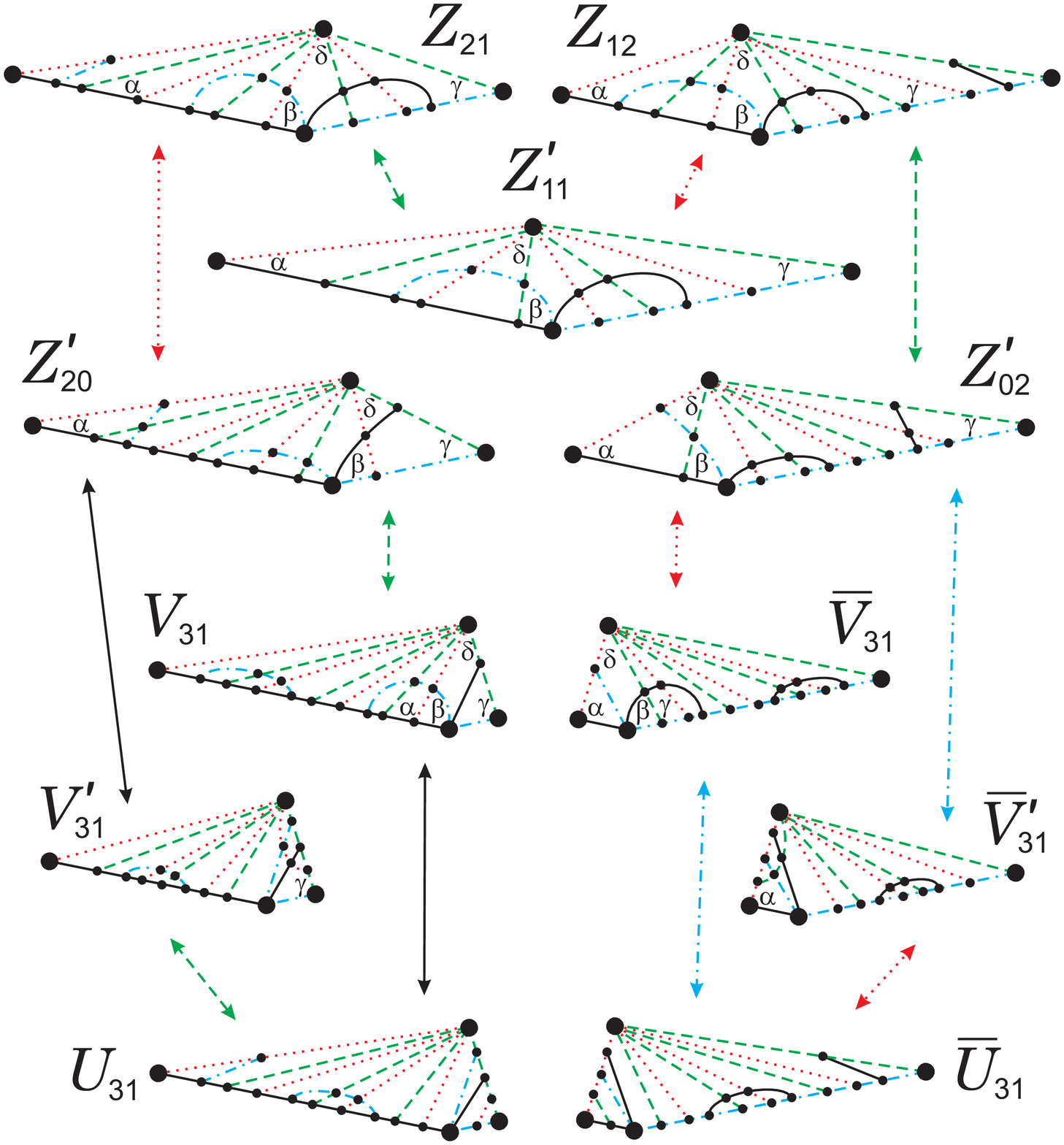}
\caption{The chains of quadrilaterals associated with $Z'_{11}$.}\label{fig:chainzprime}
\end{figure}

\begin{figure}
\centering
\includegraphics[width=4.2in]{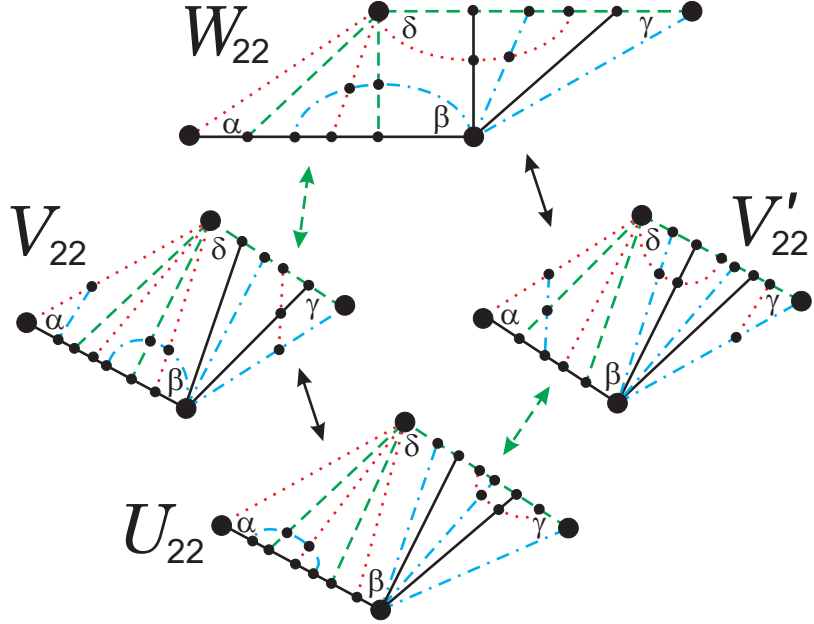}
\caption{The chains of quadrilaterals associated with $W_{22}$.}\label{fig:chainw}
\end{figure}

\begin{example}\label{chain-rs}
{\rm In Fig.~\ref{fig:chainrs}a, a quadrilateral $Q_0$ with the net $R_{11}$ (see Fig.~\ref{4circles-r-s}) with angles $\alpha$, $\beta$, $1+\gamma$, $1+\delta$ is shown. The quadrilateral face with fixed angles of the corresponding configuration $\P_0$ has the angles $(a,b,c,d)=(1-\alpha,1-\beta,\gamma,\delta)$ satisfying the inequalities
\begin{equation}\label{eqn:r}
0<\gamma+\delta-\alpha-\beta<2\min(1-\alpha,1-\beta,\gamma,\delta).
\end{equation}
Deforming $\P_0$ to a configuration with a triple intersection such that one of triangular faces of $Q_0$ adjacent to its top side
is contracted to a point results in conformal degeneration of the quadrilateral with modulus tending to $0$ in the limit.
If $1-\alpha+1-\beta>\gamma+\delta$, i.e.,
\begin{equation}\label{eqn:s}
\alpha+\beta+\gamma+\delta<2
\end{equation}
then $\P_0$ can be deformed to a configuration with a triple intersection such that the corresponding quadrilateral shown in Fig.~\ref{fig:chainrs}b is not degenerate.
Passing through the triple intersection we get a generic configuration $\P_1$ corresponding to the quadrilateral $Q_1$ shown in Fig.~\ref{fig:chainrs}c.
The net of $Q_1$ is $S_{11}$ (see Fig.~\ref{4circles-r-s}). Note that $\P_1$ has a quadrilateral face with the angles $(1-\alpha,1-\beta,1-\gamma,1-\delta)$
which is not part of the net of $Q_1$ but can be seen if a disk $D_{51}$ is attached to the side of order $5$ of $Q_1$ (shaded area in Fig.~\ref{fig:chainrs}d).
Thus $R_{11}$ and $S_{11}$ form a chain of length 1 when the inequalities (\ref{eqn:r}) and (\ref{eqn:s}) are satisfied. Note that, according to Remark \ref{rmk:pyramid}, the last inequality in (\ref{eqn:r}) can be replaced
with $\gamma+\delta-\alpha-\beta<2\min(\gamma,\delta)$.

If $\alpha+\beta+\gamma+\delta>2$ then $\P_0$ could be deformed so that the top side of $Q_0$ is contracted to a point.
This would result in conformal degeneration of $Q_0$ with the modulus tending to $\infty$.
Thus the net $R_{11}$ would constitute a chain of length $0$ in that case.
Note that a quadrilateral with the net $S_{11}$ does not exist in this case.
However, when $\alpha+\beta+\gamma+\delta>2$, in addition to a quadrilateral $Q_0$ with the net $R_{11}$,
a quadrilateral $Q_2$ with the same angles as $Q_0$, such that its net is $P_0\cup D_{15}$, with a disk attached to one side of a quadrilateral $P_0$, may exist.
The chain of $Q_2$ consists of a single net and has length $0$.
Thus there may be either one chain or two chains of length $0$ (either $R_{11}$ or $P_0\cup D_{15}$, or both) in this case.

Attaching disks to the sides of $R_{11}$ and $S_{11}$ does not affect the existence (or non-existence) of the chain of length 1 containing the two nets, except when a disk $D_{51}$ is attached to the side of order 5 of $S_{11}$ (see Fig.~\ref{fig:chainrs}d). In this case, even when both (\ref{eqn:r}) and (\ref{eqn:s}) hold, the quadrilateral $S_{11}+D_{51}$ conformally degenerates (its bottom side is contracted to a point) when configuration $P_1$ is deformed to a configuration with a triple intersection corresponding to the quadrilateral in Fig.~\ref{fig:chainrs}b. Accordingly, no disk can be attached to the side of order 7 of the quadrilateral $R_{11}$.}
\end{example}

\begin{example}\label{chain-uvzprime}
{\rm The chains of quadrilaterals associated with the quadrilateral $Z'_{00}$ are shown in Fig.~\ref{fig:chainuvzprime}.
All quadrilaterals in Fig.~\ref{fig:chainuvzprime} have the angles $\alpha,\,1+\beta,\,\gamma,\,1+\delta$.
The net $Z'_{00}$ of the quadrilateral in the center of Fig.~\ref{fig:chainuvzprime} has a quadrilateral face
with the angles $(\alpha,1-\beta,\gamma,1-\delta)$ satisfying
\begin{equation}\label{eqn:zprime}
0<\alpha-\beta+\gamma-\delta<2\min(\alpha,1-\beta,\gamma,1-\delta)
\end{equation}
If $\alpha+\delta>\beta+\gamma$ then the quadrilateral $Z'_{00}$ can be deformed through a triple intersection to the quadrilateral $Z_{10}$. If $\alpha+\delta<\beta+\gamma$ then the quadrilateral $Z'_{00}$ can be deformed through a triple intersection to the quadrilateral $\bar Z_{01}$. Note that these two deformations are not compatible.

If $\alpha+\beta<\gamma+\delta$ then the quadrilateral $Z'_{00}$ can be deformed through a triple intersection to the quadrilateral $Z_{01}$. If $\alpha+\beta>\gamma+\delta$ then the quadrilateral $Z'_{00}$ can be deformed through a triple intersection to the quadrilateral $\bar Z_{10}$. Note that these two deformations are not compatible.

If instead of (\ref{eqn:zprime}) we have
\begin{equation}\label{eqn:u}
0<\beta-\alpha+\delta-\gamma<2\min(1-\alpha,\beta,1-\gamma,\delta)
\end{equation}
then each of the quadrilaterals $Z_{10}$ and $\bar Z_{01}$ can be deformed through a triple intersection to the quadrilateral $U_{11}$, and each of the quadrilaterals $Z_{01}$ and $\bar Z_{10}$ can be deformed through a triple intersection to the quadrilateral $\bar U_{11}$.
Note that these deformations are not compatible with the deformations from $Z_{10}$, $\bar Z_{01}$, $Z_{01}$ and $\bar Z_{10}$ to $Z'_{00}$.

Thus a chain of length 2 containing the quadrilateral $Z'_{00}$ and either quadrilaterals $Z_{10}$ and $Z_{01}$ or the quadrilaterals $\bar Z_{01}$ and $Z_{10}$ exists when inequalities (\ref{eqn:zprime}) are satisfied and
$\alpha+\delta\ne\beta+\gamma$, $\alpha+\beta\ne\gamma+\delta$.
Two chains of length 1, one of them containing the quadrilateral $U_{11}$ and one of the quadrilaterals $Z_{10}$ and $\bar Z_{01}$,
and another one containing the quadrilateral $\bar U_{11}$ and one of the quadrilaterals $Z_{01}$ and $\bar Z_{10}$, exist when inequalities (\ref{eqn:u}) are satisfied and $\alpha+\delta\ne\beta+\gamma$, $\alpha+\beta\ne\gamma+\delta$.}
\end{example}

\begin{example}\label{chain-z}{\rm
Fig.~\ref{fig:chainz} shows the chains of quadrilaterals associated with the quadrilateral $Z_{11}$.
 All quadrilaterals in Fig.~\ref{fig:chainz} have the angles $\alpha,\,1+\beta,\,\gamma,\,2+\delta$.
The net of $Z_{11}$ has a quadrilateral face with the angles $(\alpha,1-\beta,\gamma,\delta)$ satisfying
\begin{equation}\label{eqn:z}
1<\alpha-\beta+\gamma+\delta>1+2\min(\alpha,1-\beta,\gamma,\delta).
\end{equation}
If $\alpha+\beta+\delta<1+\gamma$, the quadrilateral $Z_{11}$ can be deformed to the quadrilateral $Z'_{10}$.
If $\beta+\gamma+\delta<1+\alpha$, the quadrilateral $Z_{11}$ can be deformed to the quadrilateral $Z'_{01}$.
If $\alpha+\beta+\gamma<1+\delta$, the quadrilateral $Z'_{10}$ can be deformed to the quadrilateral $V_{21}$,
and the quadrilateral $Z'_{01}$ can be deformed to the quadrilateral $\bar V_{21}$.
If $\alpha+\beta+\gamma>1+\delta$, the quadrilateral $Z'_{10}$ can be deformed to the quadrilateral $V'_{21}$,
and the quadrilateral $Z'_{01}$ can be deformed to the quadrilateral $\bar V'_{21}$.
Note that conditions on the angles of $V'_{21}$ and $\bar V'_{21}$ are opposite to those on the angles of $V_{21}$ and $\bar V_{21}$, thus only one of the two deformations is possible for $Z'_{10}$ and $Z'_{01}$.

If $\alpha+\beta+\delta>1+\gamma$, the quadrilaterals $V_{21}$ and $V'_{21}$ can be deformed to the quadrilateral $U_{21}$.
If $\beta+\gamma+\delta>1+\alpha$, the quadrilaterals $\bar V_{21}$ and $\bar V'_{21}$ can be deformed to the quadrilateral $\bar U_{21}$.
Note that the condition on the angles of $U_{21}$ and $\bar U_{21}$ are opposite to those on the angles of $Z'_{10}$ and $Z'_{01}$, respectively.

Summing up, if in addition to (\ref{eqn:z}) the inequalities
\begin{equation}\label{v21}
\alpha+\beta+\delta>1+\gamma,\quad\beta+\gamma+\delta<1+\alpha,\quad\alpha+\beta+\gamma<1+\delta
\end{equation}
are satisfied, then there is a chain $\{V_{21},Z'_{10},Z_{11},Z'_{01},\bar V_{21}\}$ of length 4.
If the last inequality in (\ref{v21}) is replaced by the opposite inequality, $V_{21}$ is replaced by $V'_{21}$ and $\bar V_{21}$ is replaced by $\bar V'_{21}$, still having a chain of length 4.
If the first inequality in (\ref{v21}) is replaced by the opposite inequality, $Z'_{10}$ is replaced by $U_{21}$, and we get a chain
$\{Z_{11},Z'_{01},\bar V_{21}\}$ of length 2 and a chain $\{V_{21},U_{21}\}$ of length 1.
If both the first and the last inequalities are replaced by their opposite inequalities,
we get a chain $\{Z_{11},Z'_{01},\bar V'_{21}\}$ of length 2 and a chain $\{V'_{21},U_{21}\}$ of length 1.

The chains described in this example remain the same if we add pseudo-diagonals and attach disks to the sides of our quadrilaterals,
except if a disk $D$ is attached to the side of order 5 of $Z'_{10}$ (resp., $Z'_{01}$) the deformation to $V_{21}$ (resp., to $\bar V_{21}$)
becomes impossible even when (\ref{v21}) is satisfied, as the quadrilateral $Z'_{10}$ (resp., $Z'_{01}$) degenerates at the triple intersection. Accordingly, no disks can be attached to the ``long'' sides of order 7 of $V_{21}$, $\bar V_{21}$, $U_{21}$, and $\bar U_{21}$.}
\end{example}

\begin{example}\label{chain-zprime}
\normalfont The chains of quadrilaterals associated with the quadrilateral $Z'_{11}$ are shown in Fig.~\ref{fig:chainzprime}.
The angles of all quadrilaterals in Fig.~\ref{fig:chainzprime} are $\alpha,\,1+\beta,\,\gamma,\,3+\delta$.
The net $Z'_{11}$ has a quadrilateral face with the angles $(1-\alpha,1-\beta,1-\gamma,1-\delta)$ satisfying
\begin{equation}\label{eqn:zprime11}
2>\alpha+\beta+\gamma+\delta>2\max(\alpha,\beta,\gamma,\delta).
\end{equation}
If $\alpha+\beta<\gamma+\delta$ then the quadrilateral $Z'_{11}$ can be deformed through a triple intersection to the quadrilateral $Z_{21}$.
If $\beta+\gamma<\alpha+\delta$ then the quadrilateral $Z'_{11}$ can be deformed through a triple intersection to the quadrilateral $Z_{12}$.
Each of the other two deformations of the four-circle configuration to a triple intersection results in degeneration of the quadrilateral $Z'_{11}$.

The net $Z_{21}$ has a quadrilateral face with the angles $(1-\alpha,1-\beta,\gamma,\delta)$.
If $\alpha+\gamma>\beta+\delta$ then the quadrilateral $Z_{21}$ can be deformed through a triple intersection to the quadrilateral $Z'_{20}$.
The net $Z_{12}$ has a quadrilateral face with the angles $(\alpha,1-\beta,1-\gamma,\delta)$.
If $\alpha+\gamma>\beta+\delta$ then the quadrilateral $Z_{12}$ can be deformed through a triple intersection to the quadrilateral $Z'_{02}$.

The net $Z'_{20}$ has a quadrilateral face with the angles $(\alpha,1-\beta,\gamma,1-\delta)$.
If $\alpha+\delta>\beta+\gamma$ then the quadrilateral $Z'_{20}$ can be deformed through a triple intersection
to the quadrilateral $V_{31}$. If $\alpha+\delta<\beta+\gamma$ then the quadrilateral $Z'_{20}$ can be deformed through a triple intersection to the quadrilateral $V'_{31}$.
Note that these two deformations are not compatible.

The net $Z'_{02}$ has a quadrilateral face with the angles $(\alpha,1-\beta,\gamma,1-\delta)$.
If $\alpha+\beta<\gamma+\delta$ then the quadrilateral $Z'_{02}$ can be deformed through a triple intersection
to the quadrilateral $\bar V_{31}$. If $\alpha+\beta>\gamma+\delta$ then the quadrilateral $Z'_{02}$ can be deformed through a triple intersection to the quadrilateral $\bar V'_{31}$.
Note that these two deformations are not compatible.

The nets $V_{31}$, $V'_{31}$, $\bar V_{3,1}$ and $\bar V'_{31}$ have quadrilateral faces with the angles
$(\alpha,1-\beta,1-\gamma,\delta)$, $(1-\alpha,\beta,\gamma,1-\delta)$, $(1-\alpha,1-\beta,\gamma,\delta)$ and $(\alpha,\beta,1-\gamma,1-\delta)$,
respectively. Note that $V_{31}$ and $V'_{31}$ do not exist with the same values of the angles $\alpha,\beta,\gamma,\delta$.
Similarly, $\bar V_{31}$ and $\bar V'_{31}$ do not exist with the same values of the angles $\alpha,\beta,\gamma,\delta$.

Each of the nets $U_{31}$ and $\bar U_{31}$ has a quadrilateral face with the angles $(1-\alpha,\beta,1-\gamma,\delta)$ satisfying
\begin{equation}\label{eqn:u31}
0<\beta-\alpha+\delta-\gamma<2\min(1-\alpha,\beta,1-\gamma,\delta).
\end{equation}
If $\alpha+\gamma<\beta+\delta$ then each of the quadrilaterals $V_{31}$ and $V'_{31}$ can be deformed through a
triple intersection to the quadrilateral $U_{31}$, and  each of the quadrilaterals $\bar V_{31}$ and $\bar V'_{31}$ can be deformed through a triple intersection to the quadrilateral $\bar U_{31}$.
Note that these deformations are not compatible with the deformations between $Z_{21}$ and $Z'_{20}$, and between
$Z_{12}$ and $Z'_{02}$.

Summing up, if in addition to (\ref{eqn:zprime11}) the inequalities
\begin{equation}\label{eqn:v}
\alpha+\beta<\gamma+\delta,\quad\alpha+\delta>\beta+\gamma,\quad\alpha+\gamma>\beta+\delta
\end{equation}
are satisfied, then there is a chain $\{V_{31},Z'_{20},Z_{21},Z'_{11},Z_{12},Z'_{02},\bar V_{31}\}$ of length $6$.
Note that, according to Remark \ref{rmk:pyramid}, the last inequality in (\ref{eqn:zprime11}) can be removed
when the inequalities (\ref{eqn:v}) are satisfied.

If the first inequality in (\ref{eqn:zprime11}) is violated then $Z'_{11}$ is removed from the chain, and we have two chains of length $2$ each. If either the first or the second inequality in (\ref{eqn:v}) is violated then either $Z_{21}$ is removed and $\bar V_{31}$ is replaced by $\bar V'_{31}$,
or $Z_{12}$ is removed and $V_{31}$ is replaced by $V'_{31}$. In both cases we have a chain of length 3 and a chain of length 1.
If the third inequality in (\ref{eqn:v}) is violated then $Z'_{20}$ and $Z'_{02}$ are removed,
$U_{31}$ and $\bar U_{31}$ are added, and we have a chain of length $2$ and two chains of length $1$.

If, in addition to (\ref{eqn:zprime11}), only the first inequality in (\ref{eqn:v}) is satisfied, then there are three chains
$\{Z'_{11},Z_{21}\},\;\{V'_{31},U_{31}\}$ and $\{\bar V_{31},\bar U_{31}\}$ of length $1$.

If, in addition to (\ref{eqn:zprime11}), only the second inequality in (\ref{eqn:v}) is satisfied, then there are three chains
$\{Z'_{11},Z_{12}\},\;\{V_{31},U_{31}\}$ and $\{\bar V'_{31},\bar U_{31}\}$ of length $1$.

If, in addition to (\ref{eqn:zprime11}), only the third inequality in (\ref{eqn:v}) is satisfied, then there is a chain of length $0$ consisting
of a single net $Z'_{11}$, and two chains
$\{Z'_{20}, V'_{31}\}$ and $\{Z'_{02},\bar V'_{31}\}$ of length $1$.
\end{example}

\begin{example}\label{chain-w}
{\rm Fig.~\ref{fig:chainw} shows the chains of quadrilaterals associated with the quadrilateral $W_{22}$ with the angles $\alpha,\,2+\beta,\,\gamma,\,2+\delta$. The net of $W_{22}$ has a quadrilateral face with the angles $(1-\alpha,1-\beta,1-\gamma,1-\delta)$ satisfying
\begin{equation}\label{eqn:w}
2>\alpha+\beta+\gamma+\delta>2\max(\alpha,\beta,\gamma,\delta).
\end{equation}
If $\alpha+\beta<\gamma+\delta$, the quadrilateral $W_{22}$ can be deformed through a triple intersection
to a quadrilateral $V_{22}$. If $\alpha+\beta>\gamma+\delta$, the quadrilateral $W_{22}$ can be deformed through a triple intersection to a quadrilateral $V'_{22}$.
Note that these two transformations are incompatible.
If $\alpha+\beta+\gamma+\delta>2$, each of the quadrilaterals $V_{22}$ and $V'_{22}$ can be deformed
through a triple intersection to a quadrilateral $U_{22}$.
Note that these transformations are incompatible with the transformations between $V_{22}$ and $W_{22}$ or between $V'_{22}$ and $W_{22}$. In any case, we have one chain of length 1.}
\end{example}

\begin{prop}\label{x-even}
The chains containing generic quadrilaterals $X_{kl}$ and $X'_{pq}$ with even $n=k+l=p+q+1$ and fixed angles $\alpha,\,\beta,\,\gamma,\,n+\delta$ may be of the following kind:\newline
{\rm(i)} The chain $X_{0,n},\,X'_{0,n-1},\,X_{1,n-1},\dots,X_{n,0}$, of length $2n$, if
\begin{equation}\label{x-chain}
\alpha+\beta+\gamma+\delta>2,\;\alpha+\delta<\beta+\gamma,\;\alpha+\gamma<\beta+\delta,\;\alpha+\beta>\gamma+\delta.
\end{equation}
{\rm(ii)} If all inequalities except the first in (\ref{x-chain}) are satisfied, there are $n/2$ chains of length 2 obtained from the chain in (i) by removing all entries $X_{2m,n-2m}$ for $m=0,\dots,n/2$.\newline
{\rm(iii)} If all inequalities except the second in (\ref{x-chain}) are satisfied, there are $n/2$ chains of length 2 and one chain of length 0 obtained from the chain in (i) by removing all entries $X'_{2m,n-2m-1}$ for $m=0,\dots,n/2-1$.\newline
{\rm(iv)} If all inequalities except the third  in (\ref{x-chain}) are satisfied, there are $n/2+1$ chains of length 1 obtained from the chain in (i) by removing all entries $X_{2m+1,n-2m-1}$ for $m=0,\dots,n/2-1$.\newline
{\rm(v)} If all inequalities except the fourth in (\ref{x-chain}) are satisfied, there are $n/2$ chains of length 2 and one chain of length 0 obtained from the chain in (i) by removing all entries $X'_{2m+1,n-2m-2}$ for $m=0,\dots,n/2-1$.\newline
{\rm(vi)} If only the first and second inequalities in (\ref{x-chain}) are satisfied, there are $n/2$ chains of length $1$ and one chain of length 0 obtained from the chain in (i) by removing all entries $X_{2m+1,n-2m-1}$ and $X'_{2m+1,n-2m-2}$ for $m=0,\dots,n/2-1$.\newline
{\rm(vii)} If only the first and third inequalities in (\ref{x-chain}) are satisfied, there are $n+1$ chains of length $0$ obtained from the chain in (i) by removing all entries $X'_{m,n-m-1}$ for $m=0,\dots,n-1$.\newline
{\rm(viii)} If only the first and fourth inequalities in (\ref{x-chain}) are satisfied, there are $n/2$ chains of length $1$ and one chain of length 0 obtained from the chain in (i) by removing all entries $X'_{2m,n-2m-1}$ and $X_{2m+1,n-2m-1}$ for $m=0,\dots,n/2-1$.\newline
{\rm(ix)} If only the second and third inequalities in (\ref{x-chain}) are satisfied, there are $n/2$ chains of length $1$ obtained from the chain in (i) by removing all entries $X_{2m,n-2m}$ for $m=0,\dots,n/2$ and $X'_{2m+1,n-2m-2}$, for $m=0,\dots,n/2-1$.\newline
{\rm(x)} If only the second and fourth inequalities in (\ref{x-chain}) are satisfied, there are $n$ chains of length $0$ obtained from the chain in (i) by removing all entries $X_{m,n-m}$ for $m=0,\dots,n$.\newline
{\rm(xi)} If only the third and fourth inequalities in (\ref{x-chain}) are satisfied, there are $n/2$ chains of length $1$ obtained from the chain in (i) by removing all entries $X_{2m,n-2m}$ for $m=0,\dots,n/2$ and $X'_{2m,n-2m-1}$ for $m=0,\dots,n/2-1$.\newline
{\rm(xii)} If only the first inequality in (\ref{x-chain}) is satisfied, there are $n/2+1$ chains of length 0, each of them consisting of a single quadrilateral $X_{2m,n-2m}$ for $m=0,\dots n$.\newline
{\rm(xiii)} If only one inequality in (\ref{x-chain}), either second, third or fourth, is satisfied, there are $n/2$ chains of length 0, each of them consisting of a single quadrilateral.
\end{prop}

\begin{prop}\label{x-odd}
The chains containing generic quadrilaterals $X_{kl}$ and $X'_{pq}$ with odd $n=k+l=p+q+1\ge 3$ and fixed angles $\alpha,\,\beta,\,\gamma,\,n+\delta$ may be of the following kind:\newline
{\rm(i)} The chain $X_{0,n},\,X'_{0,n-1},\,X_{1,n-1},\dots,X_{n,0}$, of length $2n$, if
\begin{equation}\label{x-chain-odd}
\alpha+\beta+\delta>1+\gamma,\;\alpha+\gamma+\delta<1+\beta,\;\beta+\gamma+\delta>1+\alpha,\;\alpha+\beta+\gamma>1+\delta.
\end{equation}
{\rm(ii)} If all inequalities except the first in (\ref{x-chain-odd}) are satisfied, there are $(n-1)/2$ chains of length 2 and one chain of length 1 obtained from the chain in (i) by removing all entries $X_{2m,n-2m}$ for $m=0,\dots,(n-1)/2$.\newline
{\rm(iii)} If all inequalities except the second in (\ref{x-chain-odd}) are satisfied, there are $(n-1)/2$ chains of length 2 and two chains of length 0 obtained from the chain in (i) by removing all entries $X'_{2m,n-2m-1}$ for $m=0,\dots,(n-1)/2$.\newline
{\rm(iv)} If all inequalities except the third in (\ref{x-chain-odd}) are satisfied, there are $(n-1)/2$ chains of length 2 and one chain of length 1 obtained from the chain in (i) by removing all entries $X_{2m+1,n-2m-1}$ for $m=0,\dots,(n-1)/2$.\newline
{\rm(v)} If all inequalities except the fourth in (\ref{x-chain-odd}) are satisfied, there are $(n+1)/2$ chains of length 2 obtained from the chain in (i) by removing all entries $X'_{2m+1,n-2m-2}$ for $m=0,\dots,n/2-1$.\newline
{\rm(vi)} If only the first and second inequalities in (\ref{x-chain-odd}) are satisfied, there are $(n+1)/2$ chains of length $1$ obtained from the chain in (i) by removing all entries $X_{2m+1,n-2m}$ for $m=0,\dots,(n-1)/2$ and $X'_{2m+1,n-2m-2}$ for $m=0,\dots,(n-3)/2$.\newline
{\rm(vii)} If only the first and third inequalities in (\ref{x-chain-odd}) are satisfied, there are $n+1$ chains of length $0$ obtained from the chain in (i) by removing all entries $X'_{m,n-m-1}$ for $m=0,\dots,n-1$.\newline
{\rm(viii)} If only the first and fourth inequalities in (\ref{x-chain-odd}) are satisfied, there are $(n-1)/2$ chains of length $1$ and one chain of length 0 obtained from the chain in (i) by removing all entries $X_{2m+1,n-2m-1}$ for $m=0,\dots,(n-1)/2$ and $X'_{2m,n-2m-1}$ for $m=0,\dots,(n-1)/2$.\newline
{\rm(ix)} If only the second and third inequalities in (\ref{x-chain-odd}) are satisfied, there are $(n+1)/2$ chains of length $1$ obtained from the chain in (i) by removing all entries $X_{2m,n-2m}$ for $m=0,\dots,(n-1)/2$ and $X'_{2m+1,n-2m-2}$ for $m=0,\dots,(n-3)/2$.\newline
{\rm(x)} If only the second and fourth inequalities in (\ref{x-chain-odd}) are satisfied, there are $n$ chains of length $0$ obtained from the chain in (i) by removing all entries $X_{m,n-m}$ for $m=0,\dots,n$.\newline
{\rm(xi)} If only the third and fourth inequalities in (\ref{x-chain-odd}) are satisfied, there are $(n-1)/2$ chains of length $1$ and one chain of length 0 obtained from the chain in (i) by removing all entries $X_{2m,n-2m}$ and $X'_{2m+1,n-2m-2}$ for $m=0,\dots,(n-1)/2$.\newline
{\rm(xii)} If only one inequality, first, second or third in (\ref{x-chain-odd}) is satisfied, there are $(n+1)/2$ chains of length 0, each of them consisting of a single quadrilateral.\newline
{\rm(xiii)} If only fourth inequality in (\ref{x-chain-odd}) is satisfied, there are $(n-1)/2$ chains of length 0, each of them consisting of a single quadrilateral.
\end{prop}

\subsection{Lower bounds on the number of generic spherical quadrilaterals with given angles.}\label{sub:bounds}
Given a generic spherical quadrilateral $Q_0$, we want to understand how many
generic spherical quadrilaterals with the same angles and the same modulus as $Q_0$ may exist.
The chains of quadrilaterals provide a lower bound for that number.

For each quadrilateral $Q$ with the same angles and modulus as $Q_0$, consider the chain $\mathcal C$ of quadrilaterals containing $Q$.
If the quadrilaterals at both ends of $\mathcal C$ conformally degenerate, the modulus of these quadrilaterals converges either to $0$
at both ends, or to $\infty$ at both ends, or else to $0$ at one end and to $\infty$ at another.
We claim that the first two options are realized when $\mathcal C$ has odd length (contains an even number of nets)
while the third possibility is realized when $\mathcal C$ has even length.

Let $Q$ be a spherical quadrilaterals with the corners $a_0,\dots,a_3$ and the sides $[a_{j-1},a_j]$ mapped to
the circle $C_j$ of a generic four-circle configuration $\P$.
According to \cite{EGT2}, Lemma 13.1 (see also \cite{EGP}, Lemma A4), when $\P$ degenerates to a four-circle configuration
with a triple intersection, $Q$ conformally degenerates with the modulus tending to $0$ when intrinsic
distance between its sides mapped to $C_1$ and $C_3$ tends to $0$ while intrinsic distance between its other two sides does not.
Accordingly, $Q$ conformally degenerates with the modulus tending to $\infty$ when intrinsic
distance between its sides mapped to $C_2$ and $C_4$ tends to $0$ while intrinsic distance between its other two sides of does not.

When the configuration $\P$ degenerates to a four-circle configuration with a triple intersection that does not include a circle $C_j$,
the quadrilateral $Q$ conformally degenerates only when an arc $\gamma$ of its net $\Gamma$ is contracted to a point.
This happens when $\gamma$ is mapped to the complement of $C_j$, if the following three conditions are satisfied:
\begin{itemize}
\item $\gamma$ has order 1 (either it is an interior arc without interior vertices or a boundary arc without lateral vertices),
\item two ends of $\gamma$ are on the opposite sides of $Q$, mapped to $C_k$ and $C_\ell$ where $k$ and $\ell$ have opposite parity of $j$,
\item none of the two ends of $\gamma$ is mapped to $C_j$.
\end{itemize}

Since the values $j$ and $j'$ for the two ends of $I_\Gamma$ always have opposite parities,
the arcs $\gamma$ and $\gamma'$ contracted to points in the limits at the two ends of any chain $\mathcal C$ have the ends
on the same opposite sides of $Q$ if $\mathcal C$ has odd length
and on different opposite sides of $Q$ if $\mathcal C$ has even length.

If a chain $\mathcal C$ has even length (contains an odd number of nets) and the quadrilaterals at both ends of $\mathcal C$
 conformally degenerate, then the limit of the values of modulus at one end of $\mathcal C$ is $0$,
 and the limit at its other end is $\infty$.
 This implies that there exists at least one quadrilateral with a net in in $\mathcal C$ with any given modulus $0<K<\infty$.
Thus the number of chains $\mathcal C$ such that length of $\mathcal C$ is even
and quadrilaterals at both ends of $\mathcal C$ conformally degenerate is a lower bound for the number of quadrilaterals with the
given angles and modulus.

If a chain $\mathcal C$ has odd length (contains an even number of nets) and the quadrilaterals at both ends of $\mathcal C$
 conformally degenerate, then the limits of the values of modulus at its ends are either both $0$ or both $\infty$.
 In the first (resp., second) case, there exist two quadrilaterals with the nets in $\mathcal C$ with small enough
 (resp., large enough) modulus.

Finally, if the quadrilaterals at only one end of a chain $\mathcal C$ of any length conformally degenerate,
with the limit of the values of modulus at that end is $0$ (resp., $\infty$) then there is exactly one quadrilateral
with the net in $\mathcal C$ with small enough (resp., large enough) modulus..

Thus chains of generic quadrilaterals allow one to count
the number of quadrilaterals with the given angles and either small enough or large enough modulus.

\begin{example}\label{rs-modulus}{\rm Quadrilaterals with the angles $(\alpha,\beta,1+\gamma,1+\delta)$ satisfying (\ref{eqn:r})
in Example \ref{chain-rs} belong to a single chain $\mathcal C$ of length 1, consisting of the quadrilaterals with the nets $R_{11}$ and $S_{11}$ when $\alpha+\beta+\gamma+\delta<2$.
At both ends of that chain, the modulus tends to $0$, thus the lower bound for the number of quadrilaterals with the given modulus is $0$.
Since $\mathcal C$ is the only chain of quadrilaterals with such angles, there are exactly two quadrilaterals with such angles for small enough values of the modulus, and no quadrilaterals with such angles for large enough values of the modulus.

When $\alpha+\beta+\gamma+\delta>2$,
there may be either one or two chains of length $0$, with the nets either $R_{11}$ or $P_0\cup D_{15}$, or both.
This implies that there may be either at least one or at least two spherical quadrilaterals with such angles for any value of the modulus,
depending on the angles.}
\end{example}

\begin{example}\label{x-modulus-even}
{\rm Quadrilaterals with the angles $(\alpha,\beta,\gamma,n+\delta)$, where $n$ is even, are considered in Proposition \ref{x-even} (see also Example \ref{chain-x} for $n=2$).
If the inequalities (\ref{x-chain}) are satisfied (item (i) of Proposition \ref{x-even}) then there is a single chain of quadrilaterals with such angles of even length $2n$,
such that the modulus tends to $0$ at the $X_{0,n}$ end of the chain and to $\infty$ at the $X_{n,0}$ end.
This gives the lower bound of one quadrilateral for each value of the modulus.
If the first inequality in (\ref{x-chain}) is $\alpha+\beta+\gamma+\delta<2$ instead of $\alpha+\beta+\gamma+\delta>2$ (item (ii) of Proposition \ref{x-even}), then there are $n/2$ chains of length 2, and the lower bound becomes$n/2$. The maximal lower bound $n+1$ appears when the second and fourth inequalities in (\ref{x-chain}) are reversed (item (vii) of Proposition \ref{x-even}).}
\end{example}

\bigskip
\noindent{\em Department of Mathematics, Purdue Univ., West Lafayette, IN 47907-2067 USA}
\smallskip
{\em gabrielov@purdue.edu}
\end{document}